\documentclass[a4paper,a4wide,10pt]{article} 
\usepackage[utf8]{inputenc}
\usepackage{amsmath,amsthm, amssymb}
\usepackage{lmodern}
\usepackage{float}
\usepackage[usenames,dvipsnames]{pstricks}
\usepackage{enumitem} 
\usepackage{chngcntr}
\usepackage{dsfont}
\usepackage{geometry}
\usepackage{color}
\usepackage{mathtools}
\usepackage{graphicx}
\newtheorem{theorem}{Theorem}[section]
\newtheorem{lemma}[theorem]{Lemma}
\newtheorem{prop}[theorem]{Proposition}
\newtheorem{problem}[theorem]{Problem}
\theoremstyle{definition}
\newtheorem{defn}[theorem]{Definition}
\newtheorem{cor}[theorem]{Corollary}
\newtheorem{rem}[theorem]{Remark}

\DeclareMathOperator*\cp{Cap}
\DeclareMathOperator*\dist{dist}
\DeclareMathOperator*\hdim{dim_{H}}
\DeclareMathOperator*\s{\Sigma}
\DeclareMathOperator*\Om{\Omega}
\DeclareMathOperator*\diam{diam}

\DeclareMathOperator*\lip{Lip}

\DeclareMathOperator*\argmax{argmax}
\DeclarePairedDelimiter\floor{\lfloor}{\rfloor}

\newcommand*\diff{\mathop{}\, d}
\newcommand\numberthis{\addtocounter{equation}{1}\tag{\theequation}}

\numberwithin{equation}{section}
\numberwithin{figure}{section}

\author{Bohdan Bulanyi\footnote{LJLL UMR 7598, Universit\'e de Paris, France. e-mail: bulanyi@math.univ-paris-diderot.fr} }

\date{\today}

\begin{document}
\font\myfont=cmr12 at 18pt
\title{\myfont Partial regularity for the optimal $p$-compliance problem with length penalization}
\maketitle

\begin{abstract}
We establish a partial $C^{1,\alpha}$ regularity result for minimizers of the optimal $p$-compliance problem with length penalization in any spatial dimension $N\geq 2$, extending some of the results obtained in \cite{Opt,p-compl}. The key feature is that the $C^{1,\alpha}$ regularity of minimizers for some free boundary type problem is investigated with a free boundary set of codimension $N-1$. We prove that every optimal set cannot contain closed loops, and it is $C^{1,\alpha}$ regular at $\mathcal{H}^{1}$-a.e. point for every $p\in (N-1,+\infty)$.  
\end{abstract}

\tableofcontents

\section{Introduction}

A spatial dimension $N\geq 2$ and an exponent $p \in (1,+\infty)$ are given. Let $\Omega$ be an open bounded set in $\mathbb{R}^{N}$ and let $f$ belong to $L^{q_{0}}(\Om)$ with 
\begin{equation} \label{Eq 1.1}
q_{0}=(p^{*})^{\prime}\,\ \text{if}\,\ 1<p<N, \qquad q_{0}>1 \,\ \text{if} \,\ p=N, \qquad q_{0}=1 \,\ \text{if} \,\ p>N,
\end{equation}
where $p^{*}=Np/(N-p)$ and $(1/p^{*})+(1/(p^{*})^{\prime})=1$. We define the energy functional $E_{f,\Omega}$  over $W^{1,p}_{0}(\Omega)$ as follows 
\[
E_{f,\Omega}(u)=\frac{1}{p}\int_{\Omega} |\nabla u|^{p}\diff x - \int_{\Omega} fu \diff x.
\]
Thanks to the Sobolev embeddings (see \cite[Theorem 7.10]{PDE}), $E_{f,\Omega}$ is finite on $W^{1,p}_{0}(\Omega)$. It is well known that for any closed proper subset $\Sigma$ of $\overline{\Omega}$ the functional $E_{f,\Omega}$ admits a unique minimizer $u_{f,\Omega,\Sigma}$ over $W^{1,p}_{0}(\Omega\backslash \Sigma)$. Also $u_{f,\Omega,\Sigma}$ is a unique solution to the Dirichlet problem
\begin{equation} \label{1.2}
\begin{cases}
-\Delta_{p}u & =\,\ f \,\ \text{in}\,\ \Omega\backslash \Sigma \\
\qquad u & =\,\ 0 \,\ \text{on}\,\ \Sigma \cup \partial\Omega,
\end{cases}
\end{equation}
which means that $u_{f,\Omega,\Sigma} \in W^{1,p}_{0}(\Omega\backslash \Sigma)$ and
\begin{equation} \label{1.3}
\int_{\Om} |\nabla u_{f,\Omega,\Sigma}|^{p-2} \nabla u_{f,\Om,\Sigma} \nabla \varphi \diff x= \int_{\Om}f\varphi \diff x
\end{equation}
for all $\varphi \in W^{1,p}_{0}(\Om \backslash \s)$. However, if a closed set $\Sigma \subset \overline{\Omega}$ has zero $p$-capacity (for the definition of capacity, see Section 2), then $u_{f,\Omega,\Sigma}=u_{f,\Omega,\emptyset}$ (see Remark~\ref{rem 2.9}). The dependence of $u_{f,\Omega,\Sigma}$ on $p$ is neglected in this paper and in the sequel, when it is appropriate, in order to lighten the notation, we shall simply write $u_{\Sigma}$ instead of $u_{f,\Omega, \Sigma}$. For each closed proper subset $\Sigma$ of $\overline{\Omega}$ we define the $p$-compliance functional at $\Sigma$ by
\[
C_{f,\Omega}(\Sigma)=-E_{f,\Omega}(u_{\Sigma})=\frac{1}{p^{\prime}}\int_{\Omega}|\nabla u_{\Sigma}|^{p}\diff x = \frac{1}{p^{\prime}}\int_{\Omega}fu_{\Sigma}\diff x.
\]
 In two dimensions, following \cite{Opt}, we can interpret $\Omega$ as a membrane which is attached along $\s \cup \partial \Om$ to some fixed base and subjected to a given force $f$. Then $u_{\s}$ is the displacement of the membrane. The rigidity of the membrane is measured through the $p$-compliance functional which is equal to the product of the coefficient $\frac{1}{p^{\prime}}$ and the work $\int_{\Omega}fu_{\Sigma}\diff x$ performed by the force $f$. 
 
 We study the following shape optimization problem. 
\begin{problem} \label{P 1.1}
Let $p \in (N-1, +\infty)$. Given $\lambda>0$, find a set $\s \subset \overline{\Om}$ minimizing the functional $\mathcal{F}_{\lambda,f,\Omega}$ defined by
\[
\mathcal{F}_{\lambda,f,\Omega}({\s}^{\prime})=C_{f,\Omega}({\s}^{\prime})+ \lambda \mathcal{H}^{1}({\s}^{\prime})
\]
among all sets $\Sigma^{\prime}$ in the class $\mathcal{K}(\Omega)$ of all closed connected proper subsets of $\overline{\Om}$.
\end{problem}

It is worth noting that any closed set $\Sigma^{\prime}\subset\overline{\Omega}$ with $\mathcal{H}^{1}(\Sigma^{\prime})<+\infty$ is removable for the Sobolev space $W^{1,p}_{0}(\Omega)$ if $p\in (1,N-1]$ (see Theorem~\ref{thm 2.4} and Remark~\ref{rem 2.9}), namely, $W^{1,p}_{0}(\Omega \backslash \Sigma^{\prime})=W^{1,p}_{0}(\Omega)$ and this implies that $C_{f,\Omega}(\Sigma^{\prime})=C_{f,\Omega}(\emptyset)$. Thus, defining Problem~\ref{P 1.1} for some exponent $p\in (1, N-1]$, we would get only trivial solutions to this problem: every point $x_{0}$ in $\overline{\Omega}$ and the empty set. On the other hand, if $\Sigma^{\prime}\subset \overline{\Omega}$ is a closed set such that $\Sigma^{\prime}\cap \Omega$ is of Hausdorff dimension one and with finite $\mathcal{H}^{1}$-measure, then $\Sigma^{\prime}$ is not removable for $W^{1,p}_{0}(\Omega)$ if and only if $p\in (N-1,+\infty)$ (see Corollary~\ref{cor 2.6} and Remark~\ref{rem 2.9}). Furthermore, a point $x_{0}\in \Omega$ is not removable for $W^{1,p}_{0}(\Omega)$ if and only if $p\in (N,+\infty)$ (see Theorem~\ref{thm 2.4}, Remark~\ref{rem 2.5} and Remark~\ref{rem 2.9}). Therefore, Problem~\ref{P 1.1} is interesting only in the case when $p\in (N-1,+\infty)$.

We assume that $f \neq 0$ in $L^{q_{0}}(\Omega)$, because otherwise the $p$-compliance functional $C_{f,\Omega}(\cdot)$ would be reduced to zero, and then each solution to Problem~\ref{P 1.1} would be either a point $x_{0} \in \overline{\Omega}$ or the empty set.

The constrained form of the same problem was studied in \cite{Butazzo-Santambrogio, MR3063566, MR3195349}, focusing on other questions. It is worth mentioning that, according to the $\Gamma$-convergence result established in  \cite{Butazzo-Santambrogio},  in some sense the limit of Problem~\ref{P 1.1} as $p\to +\infty$ corresponds to the minimization of the functional 
\[
\mathcal{K}(\Omega)\ni \Sigma \mapsto \int_{\Omega}\dist(x, \Sigma \cup \partial \Omega)f(x)\diff x + \lambda \mathcal{H}^{1}(\Sigma).
\]
This functional was also widely studied in the literature and for which it is known that minimizers may not be $C^1$ regular (see \cite{MR3165284}).

One of the main questions about minimizers of Problem~\ref{P 1.1} is the question of whether a minimizer containing at least two points is a finite union of $C^{1,\alpha}$ curves. In dimension 2 and for $p= 2$ in \cite{Opt}, the authors established that locally inside $\Omega$ a minimizer of Problem~\ref{P 1.1} containing at least two points is a finite union of $C^{1,\alpha}$ curves that can only meet at their ends, by sets of three and with $120^{\circ}$ angles. Again in dimension 2 this result was partially generalized in \cite{p-compl} for all exponents $p \in (1, +\infty)$, namely, it was proved that every solution to Problem~\ref{P 1.1} cannot contain closed loops, is Ahlfors regular if contains at least two points (up to the boundary for a Lipschitz domain $\Omega$), and is $C^{1,\alpha}$ regular at $\mathcal{H}^{1}$-a.e. point inside $\Omega$ for every $p\in (1,+\infty)$. The main tool that was used in \cite{Opt} to prove the $\varepsilon$-regularity theorem (if $\Sigma$ is close enough, in a ball $\overline{B}_{r}(x_{0})$ such that $B_{r}(x_{0})\subset \Omega$, and in the Hausdorff distance, to a diameter of $\overline{B}_{r}(x_{0})$, then $\Sigma\cap \overline{B}_{r/10}(x_{0})$ is a $C^{1,\alpha}$ arc)  is a so-called \emph{monotonicity formula} that was inspired by Bonnet on the Mumford-Shah functional (see \cite{Bonnoet}). This monotonicity formula was also a key tool in the classification of blow-up limits (in the case when $N=p=2$), because it implies that for any point $x_{0} \in \Sigma$ there exists the limit 
\[
\lim_{r \to 0+} \frac{1}{r} \int_{B_{r}(x_{0})}|\nabla u_{\Sigma}|^{2}\diff x =e(x_{0}) \in [0,+\infty).
\] 
According to \cite{Opt}, all blow-up limits at any $x_{0} \in \Sigma \cap \Omega$ are  of the same type: either $e(x_{0})>0$ and all blow-up limits at $x_{0}$ must be a half-line, or $e(x_{0})=0$. In the latter case, either there is a blow-up at $x_{0}$ which is a line, and then all other blow-ups at $x_{0}$ must also be a line, or there is no line, and then all blow-ups at $x_{0}$ are propellers (i.e., a union of three half-lines emanating from $x_{0}$ and making $120^{\circ}$ angles). More precisely, given any point $x_{0} \in \Sigma \cap \Omega$ we only have one of the following three possibilities:
\begin{enumerate}[label=(\roman*)]
	\item $x_{0}$ belongs to the interior of a single smooth arc; in this case $x_{0}$ is called a \textit{regular} (or \textit{flat}) point.
	\item $x_{0}$ is a common endpoint of three distinct arcs which form at $x_{0}$ three equal angles of $120^{\circ}$; in this case $x_{0}$ is called a \textit{triple point}.
	\item $x_{0}$ is the endpoint of one and only one arc; in this case $x_{0}$ is called a \textit{crack-tip}.
\end{enumerate}
\
However, the approach in \cite{Opt} does not work  for the cases when $p\not =2$. The main obstruction to a full generalization of the result established in \cite{Opt} is the lack of a good monotonicity formula, when the Dirichlet energy is not quadratic $(p\neq 2)$. 

Notice that in two dimensions and for $p\neq 2$ some monotonicity formula can still be established for the $p$-energy. Indeed, assume for simplicity that $f \in L^{\infty}(\Omega)$, $N=2$, $p\in(1,+\infty)$, $\Sigma$ is a closed proper subset of $\overline{\Omega}$, $x_{0}\in \overline{\Omega}$, $0\leq 2r_{0}< r_{1}\leq 1$, $(\Sigma \cup \partial \Omega) \cap \partial B_{r}(x_{0}) \neq \emptyset$ for all $r \in (r_{0},r_{1})$ and $\gamma \in [\gamma_{\Sigma}(x_{0}, r_{0},r_{1}), 2\pi]\backslash \{0\}$, where
\[
\gamma_{\Sigma}(x_{0},r_{0},r_{1})=\sup\biggl\{\frac{\mathcal{H}^{1}(S)}{r}: r \in (2r_{0},r_{1}), \,\ S \,\ \text{is a connected component of} \,\ \partial B_{r}(x_{0}) \backslash (\Sigma \cup \partial \Omega)\biggr\}.
\]
Assume also that $2>\lambda_{p}/\gamma$, where $\lambda_{p}$ denotes the $L^{p}$ version of the Poincaré-Wirtinger constant  and is defined by
\[
\lambda_{p}=\min\biggl\{\frac{\|g^{\prime}\|_{L^{p}(0,1)}}{\|g\|_{L^{p}(0,1)}}: g \in W^{1,p}_{0}(0,1)\backslash \{0\}\biggr\}.
\]
The value of $\lambda_{p}$ was computed explicitly, for example, in \cite[Corollary 2.7]{Dacorogna} and \cite[Inequality (7a)]{bestconstant}, where the following equality was established
\[
\lambda_{p}=2\biggl(\frac{1}{p^{\prime}}\biggr)^{\frac{1}{p}}\biggl(\frac{1}{p}\biggr)^{\frac{1}{p^{\prime}}}\Gamma\biggl(\frac{1}{p^{\prime}}\biggr)\Gamma\biggl(\frac{1}{p}\biggr),
\]
in which $\Gamma$ is the usual Gamma function. Then we can prove that the function

\[
r \in (2r_{0},r_{1}) \mapsto \frac{1}{r^{\beta}}\int_{B_{r}(x_{0})} |\nabla u_{\Sigma}|^{p}\diff x + C r ^{2-\beta}
\]
is nondecreasing with $\beta=\lambda_{p}/\gamma$ and $C=C(p,\lambda_{p},\|f\|_{\infty},|\Omega|,\gamma)>0$. However, if $B_{r_{1}}(x_{0})\subset \Omega$ and $\Sigma\cap \overline{B}_{r_{1}}(x_{0})$ is a diameter of $\overline{B}_{r_{1}}(x_{0})$, then $\gamma_{\Sigma}(x_{0},0,r_{1})=\pi$ and $\lambda_{p}/\gamma$ is strictly less than one for all $\gamma \in [\gamma_{\Sigma}(x_{0},0,r_{1}), 2\pi]$ if $p\in (1,2)\cup (2,+\infty)$. So the resulting power of $r$ in
this monotonicity formula is not large enough and this formula cannot be used to prove $C^{1,\alpha}$ estimates as in the case $N=p=2$. Furthermore, a monotonicity formula for the $p$-energy can be established in any spatial dimension $N\geq3$, provided that $p>N$ and, of course, the resulting power of $r$ in this formula is not large enough to prove $C^{1,\alpha}$ estimates. Thus, a great tool has been missed that would allow us to establish a classification of blow-up limits, and that is why, as in \cite{p-compl}, we also prove only  $C^{1,\alpha}$ regularity at $\mathcal{H}^{1}$-a.e. point. Although we guess that any minimizer of Problem~\ref{P 1.1} with at least two points is a finite union of $C^{1,\alpha}$ curves.  

In this paper, we establish a partial regularity result for minimizers of Problem~\ref{P 1.1} in any spatial dimension $N\geq 2$ and for every $p \in (N-1,+\infty)$, thus generalizing some of the results obtained in \cite{Opt, p-compl}. In particular, we prove that a minimizer cannot contain closed loops (Theorem \ref{thm absence of loops}), and we establish some $C^{1,\alpha}$ regularity properties. Several of our results will hold under some integrability condition on the source term $f$. We define
\begin{equation} \label{1.4}
q_{1}=\frac{Np}{Np-N+1}\,\ \text{ if }\,\ 2\leq p<+\infty, \qquad q_{1}=\frac{2p}{3p-3} \,\ \text{ if }\,\ 1<p < 2.
\end{equation}
It is worth noting that $q_{1}\geq q_0$. The condition $f \in L^{q_{1}}(\Omega)$ for $p\in [2,+\infty)$ is natural, since $q_{1}$ in this case seems to be the right exponent which implies an estimate of the type $\int_{B_{r}(x_{0})}|\nabla u|^{p}\diff x \leq Cr$ for the solution $u$ to the Dirichlet problem 
\[
-\Delta_{p}v=f \,\ \text{in} \,\ B_{r}(x_{0}),\,\ v \in W^{1,p}_{0}(B_{r}(x_{0})),
\]
this kind of estimate we are looking for to establish regularity properties on a minimizer $\Sigma$ of Problem~\ref{P 1.1}. The main regularity result established in this paper is the following.

\begin{theorem} \label{thm 1.2} Let $\Om \subset \mathbb{R}^{N}$ be open and bounded,\,\ $ p \in (N-1,+\infty),\, f \in L^{q}(\Omega)$ with    $q> q_{1}$, where $q_{1}$ is  defined in (\ref{1.4}). Then there exists a constant $\alpha \in (0,1)$ such that the following holds. Let $\s $ be a solution to Problem \ref{P 1.1}. Then for $\mathcal{H}^{1}$-a.e. point $x \in \s \cap \Omega $ one can find a radius $r_{0}>0$ depending on $x$ such that $\s \cap \overline{B}_{r_{0}}(x)$ is a $C^{1,\alpha}$ regular curve.
\end{theorem}

It is one of the first times that the regularity of minimizers for some free boundary type problem is investigated with a free boundary set of codimension $N-1$. Notice that in Theorem~\ref{thm 1.2}, when we say that a solution $\Sigma$ to Problem~\ref{P 1.1} is $C^{1,\alpha}$ regular at $\mathcal{H}^{1}$-a.e. point $x \in \Sigma\cap\Omega$, we mean that the set of points $\Sigma\cap \Om$ around which $\Sigma$ is not a  $C^{1,\alpha}$ regular curve has zero $\mathcal{H}^{1}$-measure. Thus, Theorem~\ref{thm 1.2} is interesting only in the case when $\diam(\Sigma)>0$, which happens to be true at least for some small enough values of $\lambda$ (see Proposition \ref{prop 2.19}). Furthermore, it was proved in \cite{importance} that the connectedness assumption in the statement of Problem~\ref{P 1.1} is necessary for the existence of solutions. 

After all, when we are talking about one dimensional sets, it is not so obvious how horrible they can be.
Incidentally, in codimension 2 it seems like a good idea to understand whether the minimizer contains knots,  since our story  of loops suggests the question.

Our approach differs from that used in \cite{p-compl}, although we use some of the techniques developed in \cite{p-compl}. The approach in \cite{p-compl} is based on the fact that only in dimension 2 the ``free boundary" $\Sigma$ is of codimension~1, thus many standard arguments and competitors are available. Let us emphasize the most important places where the approach used in \cite{p-compl} does not extend in a trivial manner to higher dimensions. Firstly, in the proof of Ahlfors regularity in the ``internal case" in \cite{p-compl} (see \cite[Theorem~3.3]{p-compl}), the set $(\Sigma \backslash B_{r}(x)) \cup \partial B_{r}(x)$ was used as a competitor for a minimizer $\Sigma$ of Problem~\ref{P 1.1}, which contains at least two points. But in dimension $N\geq 3$ we cannot effectively use such a competitor, because $\partial B_{r}(x)$ has infinite $\mathcal{H}^{1}$-measure. Secondly, in \cite{p-compl}, a reflection technique was used to estimate a $p$-harmonic function in $(B_{1}\backslash [a_{1},a_{2}])\subset \mathbb{R}^{2}$ that vanishes on $[a_{1},a_{2}]\cap B_{1}$, where $[a_{1},a_{2}]$ is a diameter of $\overline{B}_{1}$, which is no more available if $N\geq 3$ for a $p$-harmonic function in $(B_{1} \backslash [a_{1},a_{2}]) \subset \mathbb{R}^{N}$ which still vanishes on $[a_{1},a_{2}]\cap B_{1}$, where $[a_{1},a_{2}]$ is a diameter of $\overline{B}_{1}$. Thirdly, in the density estimate in \cite{p-compl}, when the minimizer $\Sigma$ is $\varepsilon r$-close, in a ball $\overline{B}_{r}(x)$ and in the Hausdorff distance, to a diameter $[a,b]$ of $\overline{B}_{r}(x)$, the set $\Sigma^{\prime}=(\Sigma \backslash B_{r}(x)) \cup [a,b] \cup W$ was used as a competitor for $\Sigma$, where $W=\partial B_{r}(x) \cap \{y:\dist(y,[a,b])\leq \varepsilon r\}$ (see \cite[Proposition 6.8]{p-compl}).
However, in dimension $N\geq 3$ we cannot effectively use the above competitor $\Sigma^{\prime}$, because it has infinite $\mathcal{H}^{1}$-measure.

Let us now explain how we prove the partial $C^{1,\alpha}$ regularity of the minimizers of Problem~\ref{P 1.1} and indicate the main difficulties arising in the proof.

\noindent{\bf Comments about the proof of partial regularity.} First, we need to establish a decay behavior of the $p$-energy $r\mapsto \int_{B_{r}(x_{0})}|\nabla u_{\Sigma}|^{p}\diff x$ under flatness control on $\Sigma$ at $x_{0}$. For this we use the following strategy consisting of four steps, which is a generalization of the one used in \cite{p-compl}. In this paper, we use  some barrier in order to estimate a nonnegative $p$-subharmonic function vanishing on a 1-dimensional plane (an affine line). Let us say more precisely.\\
\textit{Step 1.}  We prove that there exist $\alpha,\delta \in(0,1)$ and $C>0$, depending only on $N$ and $p$, such that for any weak solution $u$ to the $p$-Laplace equation in $B_{1}\backslash (\{0\}^{N-1}\times (-1,1))$ vanishing $p$-q.e. on $\{0\}^{N-1}\times (-1,1)$,  the estimate 
\[
\int_{B_{r}}|\nabla u|^{p}\diff x \leq Cr^{1+\alpha} \int_{B_{1}} |\nabla u|^{p} \diff x 
\]
 holds for all $r\in (0,\delta]$.  \\
\textit{Step 2.} Arguing by contradiction and compactness, we establish a similar estimate as in \textit{Step 1} for a weak solution to the $p$-Laplace equation in $B_{r}(x_{0})\backslash \Sigma$ that vanishes on $\Sigma \cap B_{r}(x_{0})$ in the case when $\Sigma\cap \overline{B}_{r}(x_{0})$ is fairly close in the Hausdorff distance to a diameter of $\overline{B}_{r}(x_{0})$. Recall that the Hausdorff distance for any two nonempty sets $A,\, B\subset \mathbb{R}^{N}$ is defined by
\[
d_{H}(A,B)=\max\biggl\{\sup_{x \in A} \dist(x,B),\, \sup_{y\in B}\dist(y,A)\biggr\}.
\]
For each nonempty set $A\subset \mathbb{R}^{N}$, we immediately agree to define $d_{H}(\emptyset, A)=d_{H}(A,\emptyset)=+\infty$ and $d_{H}(\emptyset,\emptyset)=0$. Let $\alpha, \delta, C$ be as in \textit{Step 1}. We prove that for each $\varrho \in (0,\delta]$ there exists $\varepsilon_{0} \in (0, \varrho)$ such that if $u$ is a weak solution to the $p$-Laplace equation in $B_{r}(x_{0})\backslash\s$ vanishing $p\text{-q.e. on}\,\ \s\cap B_{r}(x_{0})$, where $\Sigma$ is a closed set such that $(\Sigma\cap B_{r}(x_{0})) \cup \partial B_{r}(x_{0})$ is connected and 
\[
\frac{1}{r}d_{H}(\Sigma \cap \overline{B}_{r}(x_{0}), L\cap \overline{B}_{r}(x_{0}))\leq \varepsilon_{0}
\]
for some affine line $L\subset \mathbb{R}^{N}$ passing through $x_{0}$, then the following estimate holds
\begin{equation*}
\int_{B_{\varrho r}(x_{0})} |\nabla u|^{p}\diff x\leq (C \varrho)^{1+\alpha}\int_{B_{r}(x_{0})} |\nabla u|^{p} \diff x.
\end{equation*}

\textit{Step 3.} Recall that we want to establish a decay estimate for a weak solution $u_{\Sigma}$ to the $p$-Poisson equation $-\Delta_{p} u=f \,\ \text{in}\,\ \Omega \backslash \Sigma$, $u\in W^{1,p}_{0}(\Omega\backslash \Sigma)$ in a ball $B_{r}(x_{0})\subset \Omega$ whenever $x_{0}$ is a flat point, i.e., when $\Sigma$ is sufficiently close, in $\overline{B}_{r}(x_{0})$ and in the Hausdorff distance, to a diameter of $\overline{B}_{r}(x_{0})$. For that purpose, we first control the difference between $u_{\Sigma}$ and its $p$-Dirichlet replacement in $B_{r}(x_{0})\backslash \Sigma$, where by the $p$-Dirichlet replacement of $u_{\s}$ in $B_{r}(x_{0})\backslash \Sigma$ we mean the solution $w \in W^{1,p}(B_{r}(x_{0}))$ to the Dirichlet problem $-\Delta_{p}u=0\,\ \text{in}\,\ B_{r}(x_{0})\backslash \Sigma$, $u-u_{\Sigma}\in W^{1,p}_{0}(B_{r}(x_{0})\backslash \Sigma)$. 
Then, for some sufficiently small $a=a(N,p)\in (0,1)$, using the estimate for the local energy $\int_{B_{ar}(x_{0})}|\nabla w|^{p}\diff x $ coming from \textit{Step 2} and also the estimate for the difference between $u_{\Sigma}$ and $w$ in $B_{r}(x_{0})\backslash \Sigma$, we arrive at the following decay estimate for $u_{\Sigma}$:
\begin{equation*}
\frac{1}{ar}\int_{B_{ar}(x_0)} |\nabla u_{\s}|^p \diff x \leq \frac{1}{2} \biggl( \frac{1}{r}\int_{B_{r}(x_0)} |\nabla u_{\s}|^p \diff x \biggr)+ C r^{\gamma(N,p,q)},
\end{equation*}
where $\gamma(N, p , q)\in (0,1)$ provided that $q>q_{1}$ (for the definition of $q_{1}$ see (1.4)).

\textit{Step 4.} Finally, by iterating the result of \textit{Step 3} in a sequence of balls $\{B_{a^{l}r_{1}}(x_{0})\}_{l}$, we obtain the desired  decay behavior of the $p$-energy $r\mapsto \int_{B_{r}(x_{0})}|\nabla u_{\Sigma}|^{p}\diff x$ under flatness control on $\Sigma$ at $x_{0}$. Namely, there exist $b\in (0,1)$ and $C>0$ such that if $\Sigma \cap \overline{B}_{r}(x_{0})$ remains fairly flat for all $r$ in $[r_{0}, r_{1}]$, $B_{r_{1}}(x_{0})\subset \Omega$ and $r_{1}$ is sufficiently small, then the following estimate holds
\[
\int_{B_{r}(x_{0})}|\nabla u_{\Sigma}|^{p}\diff x \leq C \Bigl(\frac{r}{r_{1}}\Big)^{1+b}  \int_{B_{r_{1}}(x_{0})}|\nabla u_{\Sigma}|^{p}\diff x + Cr^{1+b} \,\ \text{for  all} \,\ r \in [r_{0}, r_{1}].
\]
Thus, if $x_{0}\in \Sigma\cap \Omega$ and $\Sigma \cap \overline{B}_{r}(x_{0})$ remains fairly flat for all sufficiently small $r>0$, then the energy $r\mapsto \frac{1}{r}\int_{B_{r}(x_{0})}|\nabla u_{\Sigma}|^{p}\diff x$ converges to zero no slower than $Cr^{b}$ for some $b\in (0,1)$ and $C>0$. This will be used to prove the desired $C^{1,\alpha}$ result, and the same kind of estimate will also be used to prove the absence of closed loops. 

Now let us briefly explain how we use the above decay estimate to prove the partial $C^{1,\alpha}$ regularity of the minimizers  inside $\Omega$. The idea is to show that every minimizer $\Sigma$ of Problem~\ref{P 1.1} with $\diam(\Sigma)>0$ is an \textit{almost minimizer} for the length at any flat point inside $\Omega$. More precisely, we need to prove that there exists $\beta\in (0,1)$ such that for any competitor $\Sigma^{\prime}$ being $\tau r$-close, in a ball $\overline{B}_{r}(x_{0})\subset \Omega$ and in the Hausdorff distance, to a diameter of $\overline{B}_{r}(x_{0})$ for some small $\tau \in (0,1)$ and satisfying $\Sigma^{\prime} \Delta \Sigma \subset \overline{B}_{r}(x_{0})$, it holds 
\[
\mathcal{H}^{1}(\Sigma \cap \overline{B}_{r}(x_{0}))\leq \mathcal{H}^{1}(\Sigma^{\prime}\cap \overline{B}_{r}(x_{0}))+Cr^{1+\beta}
\]
whenever $x_{0} \in \Sigma \cap \Omega$ is a flat point. In our framework, the term $Cr^{1+\beta}$ may only come from the $p$-compliance part of the functional $\mathcal{F}_{\lambda,f,\Omega}$. Thus, we need to prove that 
\[
C_{f,\Omega}(\Sigma^{\prime})-C_{f,\Omega}(\Sigma) \leq C r^{1+\beta}
\]
whenever $x_{0} \in \Sigma \cap \Omega$ is a flat point, $\Sigma^{\prime} \in \mathcal{K}(\Omega)$ is $\tau r$-close, in $\overline{B}_{r}(x_{0})\subset \Omega$ and in the Hausdorff distance, to a diameter of $\overline{B}_{r}(x_{0})$ for some small $\tau \in (0,1)$ and $\Sigma^{\prime}\Delta \Sigma\subset \overline{B}_{r}(x_{0})$.  Hereinafter in this section, $C$ denotes a positive constant that can only depend on $N,\, p,\, q_{0},\, q, \, \|f\|_{q},\, |\Omega|$ ($q_{0}$ is defined in (\ref{Eq 1.1}), $q\geq q_{0}, \, f\in L^{q}(\Omega)$) and can be different from line to line.  Notice that one of the difficulties in obtaining the above estimate is a nonlocal behavior of the $p$-compliance functional. Namely, changing $\Sigma$ locally in $\Omega$, we change $u_{\Sigma}$ in the whole $\Omega$. This can be overcome, using a cut-off argument. Actually, we have shown that if $\Sigma^{\prime}$ is a competitor for $\Sigma$ and $\Sigma^{\prime}\Delta\Sigma\subset \overline{B}_{r}(x_{0})$, then 
\[
C_{f,\Omega}(\Sigma^{\prime})-C_{f,\Omega}(\Sigma)\leq C\int_{B_{2r}(x_{0})}|\nabla u_{\Sigma^{\prime}}|^{p}\diff x + Cr^{N+p^{\prime}-\frac{Np^{\prime}}{q}}.
\]
However, the right-hand side in the above estimate depends on the competitor $\Sigma^{\prime}$, which pushes us to introduce the quantity
\begin{equation*}
w^{\tau}_{\s}(x_{0},r)=\sup_{\substack{\s^{\prime}\in \mathcal{K}(\Omega),\, \s^{\prime}\Delta \s \subset \overline{B}_{r}(x_{0}), \\\mathcal{H}^{1}(\Sigma^{\prime})\leq 100\mathcal{H}^{1}(\Sigma),\,  \beta_{\s^{\prime}}(x_{0},r)\leq \tau}}  \frac{1}{r} \int_{B_{r}(x_{0})}|\nabla u_{\Sigma^{\prime}}|^{p}\diff x, 
\end{equation*}
where $\beta_{\Sigma^{\prime}}(x_{0},r)$ is the flatness defined by 
\begin{equation*}
\beta_{\Sigma^{\prime}}(x_{0},r)=\inf_{L\ni x_{0}} \frac{1}{r}d_{H}(\Sigma^{\prime}\cap \overline{B}_{r}(x_{0}), L\cap \overline{B}_{r}(x_{0})),
\end{equation*}
where the infimum is taken over the set of all affine lines (1-dimensional planes) $L$ in $\mathbb{R}^{N}$ passing through~$x_{0}$. The quantity $w^{\tau}_{\Sigma}(x_{0},r)$ is a variant of the one introduced in \cite{Opt} and already used in \cite{p-compl}.
Also notice that the assumption $~\mathcal{H}^{1}(\Sigma^{\prime}) \leq 100 \mathcal{H}^{1}(\Sigma)$ in the definition of $w^{\tau}_{\Sigma}(x_{0},r)$ is rather optional, however, it guarantees  that if $\Sigma^{\prime}$ is a maximizer in this definition, then it is arcwise connected. Thus, if $\Sigma^{\prime}\in \mathcal{K}(\Omega)$, $\Sigma^{\prime}\Delta\Sigma\subset \overline{B}_{r}(x_{0})$, $\mathcal{H}^{1}(\Sigma^{\prime})\leq 100 \mathcal{H}^{1}(\Sigma)$ and $\beta_{\Sigma^{\prime}}(x_{0},2r)\leq \tau$, we arrive at the estimate
\[
C_{f,\Omega}(\Sigma^{\prime})-C_{f,\Omega}(\Sigma)\leq Crw^{\tau}_{\Sigma}(x_{0}, 2r) + Cr^{N+p^{\prime}-\frac{Np^{\prime}}{q}}.
\]
Next, applying the decay estimate established in \textit{Step 4} above to the function $u_{\widetilde{\Sigma}}$, where $\widetilde{\Sigma}$ is a maximizer in the definition of $w^{\tau}_{\Sigma}(x_{0}, 2r)$, we obtain the following control \[w^{\tau}_{\Sigma}(x_{0},2r) \leq C\Bigl(\frac{r}{r_{1}}\Bigr)^{b}w^{\tau}_{\Sigma}(x_{0},r_{1})+Cr^{b},\] provided that $\beta_{\Sigma}(x_{0},\varrho)$ remains fairly small for all  $\varrho \in [2r, r_{1}]$, $r_{1}>0$ is small enough, $B_{r_{1}}(x_{0})\subset~\Omega$ and $r>0$ is sufficiently  small with respect to $r_{1}$. Using also that $b<N-1+p^{\prime}-Np^{\prime}/q$, altogether we get 
\begin{equation*}
\mathcal{H}^{1}(\Sigma \cap \overline{B}_{r}(x_{0})) \leq \mathcal{H}^{1}(\Sigma^{\prime}\cap \overline{B}_{r}(x_{0}))+ Cr\Bigl(\frac{r}{r_{1}}\Bigr)^{b}w^{\tau}_{\Sigma}(x_{0},r_{1}) +Cr^{1+b}
\end{equation*}
whenever $\Sigma$ is a minimizer of Problem~\ref{P 1.1}, $\beta_{\Sigma}(x_{0},\varrho)$ remains fairly small for all $\varrho \in [r,r_{1}]$, $r_{1}>0$ is small enough, $B_{r_{1}}(x_{0})\subset \Omega$, $r>0$ is sufficiently small with respect to $r_{1}$, $\Sigma^{\prime}\in \mathcal{K}(\Omega)$ is $\tau r$-close, in $\overline{B}_{r}(x_{0})$ and in the Hausdorff distance, to a diameter of $\overline{B}_{r}(x_{0})$, $\Sigma^{\prime}\Delta \Sigma \subset \overline{B}_{r}(x_{0})$ and $\mathcal{H}^{1}(\Sigma^{\prime})\leq 100\mathcal{H}^{1}(\Sigma)$.

The next step is to find a nice competitor $\Sigma^{\prime}$ for a minimizer $\Sigma$. More precisely, assume that $x_{0} \in \Sigma,\,\ \overline{B}_{r}(x_{0})\subset \Omega$, $r$ is sufficiently small, $\beta_{\Sigma}(x_{0},r)$ is small enough and remains fairly small on a large scale. The task is to find a competitor $\Sigma^{\prime}$ such that $\Sigma^{\prime}\Delta \Sigma \subset \overline{B}_{r}(x_{0})$, $\Sigma^{\prime}$ is $\tau r$-close, in $\overline{B}_{r}(x_{0})$ and in the Hausdorff distance, to a diameter of $\overline{B}_{r}(x_{0})$ for some small $\tau \in (0,1)$ and, in addition, the length (i.e., $\mathcal{H}^{1}$-measure) of $\Sigma^{\prime}\cap \overline{B}_{r}(x_{0})$ is fairly close to the length of a diameter of $\overline{B}_{r}(x_{0})$. Recall that in two dimensions we can take 
\[
\Sigma^{\prime}=(\Sigma \backslash B_{r}(x_{0}))\cup (\partial B_{r}(x_{0})\cap \{x: \dist(x,L)\leq \beta_{\Sigma}(x_{0},r) r\}) \cup (L\cap B_{r}(x_{0}))
\]
provided $\beta_{\Sigma}(x_{0},r)=d_{H}(\Sigma\cap \overline{B}_{r}(x_{0}), L\cap \overline{B}_{r}(x_{0}))/r$. But in dimension $N\geq 3$ such a competitor is not admissible, since it has Hausdorff dimension $N-1 \geq 2$. Notice that the main difficulty arising in the construction of a nice competitor in dimension $N\geq 3$ is that we do not know whether the quantity $\mathcal{H}^{0}(\Sigma \cap \partial B_{\varrho}(x_{0}))$ is uniformly bounded from above for $x_{0}\in \Sigma$ and $\varrho>0$. However, according to the coarea inequality (see, for instance, \cite[Theorem 2.1]{Paolini-Stepanov}), we know that for all $\varrho>0$,
\[
\mathcal{H}^{1}(\Sigma\cap B_{\varrho}(x_{0}))\geq  \int^{\varrho}_{0}\mathcal{H}^{0}(\Sigma \cap \partial B_{t}(x_{0}))\diff t.
\]
If, moreover, $\varrho< \diam(\Sigma)/2$, then $\Sigma \cap \partial B_{t}(x_{0})\neq \emptyset$ for all $t \in (0,\varrho]$, since $x_{0}\in \Sigma$ and $\Sigma$ is arcwise connected (see Remark~\ref{arcwise connected}). Thus, assuming that  $\varrho<\diam(\Sigma)/2$ and $\kappa \in (0,1/4]$, for any $s \in [\kappa \varrho, 2\kappa \varrho]$ we deduce the following
\[
\mathcal{H}^{1}(\Sigma\cap B_{\varrho}(x_{0})) \geq \int^{\varrho}_{0}\mathcal{H}^{0}(\Sigma \cap \partial B_{t}(x_{0}))\diff t > \int^{(1+\kappa)s}_{s}\mathcal{H}^{0}(\Sigma \cap \partial B_{t}(x_{0}))\diff t.
\]
The latter inequality implies that there exists $t \in [s, (1+\kappa)s]$ for which
\[
\mathcal{H}^{0}(\Sigma \cap \partial B_{t}(x_{0})) \leq \frac{1}{\kappa^{2}} \theta_{\Sigma}(x_{0},\varrho),\,\ \text{where}\,\ \theta_{\Sigma}(x_{0},\varrho)=\frac{1}{\varrho}\mathcal{H}^{1}(\Sigma \cap B_{\varrho}(x_{0})).
\]
So if $\kappa \in (0,1/4]$, $x_{0} \in \Sigma$, $r>0$ is sufficiently small and $B_{r}(x_{0})\subset \Omega$, then for all $s\in [\kappa r, 2\kappa r]$ we can construct the following competitor
\[
\Sigma^{\prime}=(\Sigma \backslash B_{t}(x_{0})) \cup \Biggl(\bigcup^{\mathcal{H}^{0}(\Sigma \cap \partial B_{t}(x_{0}))}_{i=1} [z_{i},z_{i}^{\prime}]\Biggr) \cup (L\cap \overline{B}_{t}(x_{0})),
\]
where $t \in [s, (1+\kappa)s]$ is such that $\mathcal{H}^{0}(\Sigma \cap \partial B_{t}(x_{0}))\leq \theta_{\Sigma}(x_{0},r)/\kappa^{2}$, $L$ is an affine line realizing the infimum in the definition of $\beta_{\Sigma}(x_{0},t)$, $z_{i}\in \Sigma \cap \partial B_{t}(x_{0})$ and $z^{\prime}_{i}$ denotes the projection of $z_{i}$ to $L\cap \overline{B}_{t}(x_{0})$. The flatness $\beta_{\Sigma^{\prime}}(x_{0},t)$ is less than or equal to $\beta_{\Sigma}(x_{0},t)$ by construction. Assuming in addition that $\beta_{\Sigma}(x_{0},r)$ is fairly small and $\theta_{\Sigma}(x_{0},r)$ is also small enough, for the competitor $\Sigma^{\prime}$ constructed above it holds: $\beta_{\Sigma^{\prime}}(x_{0},t)$ is sufficiently small, since $\beta_{\Sigma^{\prime}}(x_{0},t)\leq \beta_{\Sigma}(x_{0},t)$ and $\beta_{\Sigma}(x_{0},\varrho)$ remains small for all $\varrho$ in $(0,r)$ which are not too far from $r$;  the length  of $\Sigma^{\prime}\cap \overline{B}_{t}(x_{0})$ is fairly close to the length of a diameter of $\overline{B}_{t}(x_{0})$; the following estimate holds
\[
\mathcal{H}^{1}(\Sigma\cap B_{s}(x_{0}))\leq \mathcal{H}^{1}(\Sigma\cap B_{t}(x_{0}))\leq \mathcal{H}^{1}(\Sigma^{\prime}\cap B_{t}(x_{0})) + C t\Bigl(\frac{t}{r}\Bigr)^{b}w^{\tau}_{\Sigma}(x_{0},r)+Ct^{1+b}.
\]
This allows us to prove the following fact: there exist $\varepsilon$, $\kappa \in (0,1/100)$ such that if $\Sigma$ is a minimizer of Problem~\ref{P 1.1}, $x_{0} \in \Sigma$, $r>0$ is sufficiently small, $B_{r}(x_{0})\subset \Omega$ and the following condition holds
\[
\beta_{\Sigma}(x_{0},r)+w^{\tau}_{\Sigma}(x_{0},r)\leq \varepsilon,\,\ \theta_{\Sigma}(x_{0},r) \leq 10\overline{\mu} \tag{$\mathcal{C}$}
\] 
with $\overline{\mu}$ being a unique positive solution to the equation $\mu=5+ \mu^{1-\frac{1}{N}}$ (we shall explain a bit later why we take this particular bound), then there exists $s \in [\kappa r, 2\kappa r]$ for which $\mathcal{H}^{0}(\Sigma \cap \partial B_{s}(x_{0}))=2$, the two points $\{\xi_{1}, \xi_{2}\}$ of $\Sigma \cap \partial B_{s}(x_{0})$ belong to two different connected components~of 
\[
\partial B_{s}(x_{0})\cap \{x : \dist(x,L)\leq \beta_{\Sigma}(x_{0},s)s\},
\]
where $L$ is an affine line realizing the infimum in the definition of $\beta_{\Sigma}(x_{0},s)$. Moreover, $(\Sigma \backslash B_{s}(x_{0}))\cup [\xi_{1},\xi_{2}]$ is a nice competitor for $\Sigma$. Using this result together with the decay behavior of the local energy $w^{\tau}_{\Sigma}$, we prove that there exists a constant $a \in (0,1/100)$ such that if $x_{0}\in \Sigma$, $r>0$ is small enough, $B_{r}(x_{0})\subset \Omega$ and the condition $(\mathcal{C})$ holds with some sufficiently small $\varepsilon>0$, then
\[
\beta_{\Sigma}(x_{0},ar)\leq C(w^{\tau}_{\Sigma}(x_{0},r))^{\frac{1}{2}}+Cr^{\frac{b}{2}} \,\ \text{and} \,\ w^{\tau}_{\Sigma}(x_{0},ar)\leq \frac{1}{2}w^{\tau}_{\Sigma}(x_{0},r) +C(ar)^{b}.
\]

Next, we need to control the density $\theta_{\Sigma}$ from above on a smaller scale by its value on a larger scale. Notice that in this paper we do not prove the Ahlfors regularity for a minimizer of Problem~\ref{P 1.1} in the spatial dimension $N\geq 3$ (for a proof in dimension 2, see \cite[Theorem 3.3]{p-compl}), namely, that there exist constants $0<c_{1}< c_{2} $ and a radius $r_{0}>0$ such that if $\Sigma$ is a minimizer of Problem~\ref{P 1.1} with at least two points, then  for all $x \in \Sigma$ and $r \in (0, r_{0})$,  it holds
\[
c_{1}\leq \theta_{\Sigma}(x,r) \leq c_{2}.
\]
In dimension $N\geq 3$ this problem seems very difficult and interesting. However, adapting some of the approaches of Stepanov and Paolini in \cite{Stepanov-Paolini}, we prove the following fact: for each $a \in (0,1/20]$ there exists $\varepsilon \in (0,1/100)$ such that if $x_{0}\in \Sigma$, $B_{r}(x_{0})\subset \Omega$, $r>0$ is sufficiently small and $\beta_{\Sigma}(x_{0},r) + w^{\tau}_{\Sigma}(x_{0},r)\leq \varepsilon$, then
\[
\theta_{\Sigma}(x_{0},ar)\leq 5 + \theta_{\Sigma}(x_{0},r)^{1-\frac{1}{N}}.
\]
Notice that if $\theta_{\Sigma}(x_{0},r)\leq 10\overline{\mu}$, then, using the above estimate, we get
\[
\theta_{\Sigma}(x_{0},ar)\leq 5+ (10\overline{\mu})^{1-\frac{1}{N}}\leq 10(5+\overline{\mu}^{1-\frac{1}{N}})=10\overline{\mu}.
\]
The factor $10$ in the estimate $\theta_{\Sigma}(x_{0},r)\leq 10 \overline{\mu}$ is rather important, it appears in the proof of Corollary~\ref{cor 6.13}. 

Altogether we prove that there exist constants $a,\varepsilon \in (0,1/100)$, $b \in (0,1)$ such that if  $x_{0}\in \Sigma$, $r>0$ is sufficiently small, $B_{r}(x_{0}) \subset \Omega$ and the condition $(\mathcal{C})$ holds with $\varepsilon$, then for all $n \in \mathbb{N}$,
\[
\beta_{\Sigma}(x_{0},a^{n+1}r)\leq C(w^{\tau}_{\Sigma}(x_{0},a^{n}r))^{\frac{1}{2}}+C(a^{n}r)^{\frac{b}{2}} \,\ \text{and} \,\ w^{\tau}_{\Sigma}(x_{0},a^{n+1}r)\leq \frac{1}{2}w^{\tau}_{\Sigma}(x_{0},a^{n}r) +C(a^{n+1}r)^{b}.
\]
This, in particular, implies that $\beta_{\Sigma}(x_{0},\varrho)\leq \widetilde{C}\varrho^{\alpha}$ for some $\alpha \in (0,1), \,\ \widetilde{C}=\widetilde{C}(N,p,q_{0},q,\|f\|_{q},|\Omega|,r)>0$ and for all sufficiently small $\varrho>0$ with respect to $r$.  

Finally, we arrive to the so-called ``$\varepsilon$-regularity'' theorem, which says the following: there exist constants $\tau,\,a, \varepsilon, \alpha, \, \overline{r}_{0} \in (0,1)$ such that whenever $x\in \Sigma$, $0<r<\overline{r}_{0}$, $B_{r}(x)\subset \Omega$,
\[
\beta_{\Sigma}(x, r)+w^{\tau}_{\Sigma}(x,r)\leq \varepsilon \,\ \text{and} \,\ \theta_{\Sigma}(x,r)\leq \overline{\mu},
\]
then  for some $\widetilde{C}=\widetilde{C}(N,p,q_{0},q,\|f\|_{q},|\Omega|,r)>0$, $\beta_{\Sigma}(y,\varrho)\leq \widetilde{C}\varrho^{\alpha}$ for any point $y \in \Sigma\cap B_{ar}(x)$ and any radius $\varrho\in (0,ar)$. In particular, there exists $t \in (0,1)$ such that $\Sigma \cap \overline{B}_{t}(x)$ is a $C^{1,\alpha}$ regular curve. On the other hand, notice that, since closed connected sets with finite $\mathcal{H}^{1}$-measure are $\mathcal{H}^{1}$-rectifiable (see, for instance, \cite[Proposition~30.1]{David}), $\beta_{\Sigma}(x,r)\to 0$ as $r\to 0+$ at $\mathcal{H}^{1}$-a.e. $x \in \Sigma$ and hence at $\mathcal{H}^{1}$-a.e. $x \in \Sigma\cap\Omega$, $w^{\tau}_{\Sigma}(x,r)\to 0$ as $r \to 0+$. Moreover, in view of Besicovitch-Marstrand-Mattila Theorem (see \cite[Theorem 2.63]{APD}),
\[
\theta_{\Sigma}(x,r)\to 2 \,\ \text{as}\,\ r \to 0+\,\ \text{at}\,\ \mathcal{H}^{1}\text{-a.e}.\,\ x \in \Sigma.
\]
At the end, observing that for each $N\geq 2$, the unique positive solution $\overline{\mu}$ to the equation $\mu=5+\mu^{1-\frac{1}{N}}$ is strictly greater than 5, we bootstrap all the estimates and prove that every minimizer  $\Sigma$ of Problem~\ref{P 1.1} is $C^{1,\alpha}$ regular at $\mathcal{H}^{1}$-a.e. point $x\in \Sigma \cap \Omega.$

\section{Preliminaries}
In this paper, $B_{r}(x)$, $\overline{B}_{r}(x)$ and $\partial{B}_{r}(x)$ will denote, respectively, the open ball, the closed ball and the $N$-sphere with center $x$ and radius $r$. If the center is at the origin, we write $B_{r}$, $\overline{B}_{r}$ and $\partial B_{r}$ the corresponding balls and the $N$-sphere. For each set $A \subset \mathbb{R}^{N}$, the set $A^{c}$ will denote its complement in $\mathbb{R}^{N}$, that is, $A^{c}=\mathbb{R}^{N}\backslash A$. We shall sometimes write points of $\mathbb{R}^{N}$ as $x=(x^{\prime}, x_{N})$ with $x^{\prime} \in \mathbb{R}^{N-1}$ and $x_{N}\in \mathbb{R}$. We use the standard notation for Sobolev spaces. For an open set $U\subset \mathbb{R}^{N}$, denote by $W^{1,p}_{0}(U)$ the closure of $C^{\infty}_{0}(U)$ in the Sobolev space $W^{1,p}(U)$, where $C^{\infty}_{0}(U)$ is the space of functions in $C^{\infty}(U)$ with compact support in $U$. Recall that $W^{1,p}_{loc}(U)$ is the space of functions $u$ such that $u\in W^{1,p}(V)$ for all $V\subset\subset U$. We shall denote by $\mathcal{H}^{d}(A)$ the $d$-dimensional Hausdorff measure of~$A$. In this paper, we say that a value is positive if it is strictly greater than zero, and a value is nonnegative if it is greater than or equal to zero.

We begin by defining weak solutions to the $p$-Laplace equation, 
\begin{equation*}
\Delta_{p}u=div(|\nabla u|^{p-2}\nabla u)=0.
\end{equation*}
\begin{defn} 
Let $U \subset \mathbb{R}^{N}$ be open and bounded, $p \in (1, +\infty)$.  We say that $u$ is a \textit{weak subsolution (supersolution)} to the $p$-Laplace equation in $U$ provided $u \in W^{1,p}_{loc}(U)$ and 
\[
\int_{U}|\nabla u|^{p-2}\nabla u \nabla \varphi \diff x \leq (\geq) 0,
\]
whenever $\varphi \in C^{\infty}_{0}(U)$ is nonnegative.
\end{defn}
A function $u$ is a \textit{weak solution} to the $p$-Laplace equation if it is both a subsolution and a supersolution. If $u$ is an upper (lower) semicontinuous weak subsolution (supersolution) to the $p$-Laplace equation in $U$, then we say that $u$ is \textit{$p$-subharmonic ($p$-superharmonic)} in $U$. If $u$ is a continuous weak solution to the $p$-Laplace equation in $U$, then we say that $u$ is $p$-harmonic in $U$.

The following basic result for weak solutions holds (see \cite[Theorem 2.7]{Lindqvist}).
\begin{theorem} \label{thm 2.2}
	Let $U$ be a bounded open set in $\mathbb{R}^{N}$ and let $u \in W^{1,p}(U)$. The following two assertions are equivalent.
	\begin{enumerate}[label=(\roman*)]
		\item $u$ is minimizing:
		\[
		\int_{U} |\nabla u|^{p}\diff x \leq \int_{U}|\nabla v|^{p}\diff x,\,\ \text{when}\,\ v-u \in W^{1,p}_{0}(U).
		\]
		\item the first variation vanishes:
		\[
		\int_{U} |\nabla u|^{p-2}\nabla u \nabla \zeta \diff x=0,\,\ \text{when}\,\ \zeta \in W^{1,p}_{0}(U).
		\]
	\end{enumerate}
\end{theorem}
Now we introduce the notion of the Bessel capacity (see e.g. \cite{Potential}, \cite{Ziemer}). 
\begin{defn} \label{def 2.3} \textit{ For $p\in (1,+\infty)$, the Bessel $(1,p)$-capacity of a set $E\subset \mathbb{R}^{N}$ is defined as
\[
{\rm Cap}_{p}(E)=\inf \bigl\{\|f\|^{p}_{L^{p}(\mathbb{R}^{N})} :\, g*f \geq 1\,\ \text{on}\,\ E,\,\  f\in L^{p}(\mathbb{R}^{N}), \,\ f \geq 0\bigr\},
\]
where the Bessel kernel $g$ is defined as that function whose Fourier transform is}
\[
\hat{g}(\xi)=(2\pi)^{-\frac{N}{2}}(1+|\xi|^{2})^{-\frac{1}{2}}.
\]
\end{defn}
We say that a property holds $p$-quasi everywhere (abbreviated as $p$-q.e.) if it holds except on a set $A$ where $\cp_{p}(A)=0$. It is worth mentioning that by \cite[Corollary 2.6.8]{Potential}, for every $p\in (1,+\infty)$, the notion of the Bessel capacity ${\rm Cap}_{p}$ is equivalent to the following 
\[
\widetilde{{\rm Cap}_{p}}(E)=\inf_{u \in W^{1,p}(\mathbb{R}^{N})}\biggl\{\int_{\mathbb{R}^{N}}|\nabla u|^{p}\diff x + \int_{\mathbb{R}^{N}}|u|^{p}\diff x : u \geq 1 \,\ \text{on some neighborhood of $E$}\biggr\}
\]
in the sense that there exists $C=C(N,p)>0$ such that for any set $E\subset \mathbb{R}^{N}$,
\[
\frac{1}{C}\widetilde{{\rm Cap}_{p}}(E) \leq {\rm Cap}_{p}(E) \leq C\widetilde{{\rm Cap}_{p}}(E).
\]
The notion of capacity is crucial in the investigation of the pointwise behavior of Sobolev functions.

For convenience, we recall the next theorems and propositions.
\begin{theorem} \label{thm 2.4} Let $E\subset \mathbb{R}^{N}$ and $p\in (1,N]$. Then $\cp_{p}(E)=0$ if $\mathcal{H}^{N-p}(E)<+\infty$. Conversely, if $\cp_{p}(E)=0$, then $\mathcal{H}^{N- p+\varepsilon}(E)=0$ for every $\varepsilon>0$.
\end{theorem}
\begin{proof} For a proof of the fact that  ${\rm Cap}_p(E)=0$ if $\mathcal{H}^{N-p}(E)<+\infty$, we refer to \cite[Theorem 5.1.9]{Potential}. The fact that if ${\rm Cap}_{p}(E)=0,$ then $\mathcal{H}^{N-p+\varepsilon}(E)=0$ for every $\varepsilon>0$ follows from \cite[Theorem 5.1.13]{Potential}. 
\end{proof}
\begin{rem} \label{rem 2.5} Let $p \in (N,+\infty)$. Then there exists $C=C(N,p)>0$ such that for any nonempty set $E\subset \mathbb{R}^{N}$, ${\rm Cap}_p(E)\geq C$. We can take $C={\rm Cap}_{p}(\{0\})$, which is positive by \cite[Proposition~2.6.1 (a)]{Potential}, and use the fact that the Bessel $(1,p)$-capacity is an invariant under translations and is nondecreasing with respect to set inclusion.
\end{rem}

Recall that for all $E\subset \mathbb{R}^{N}$ the number
\[
\hdim(E)=\sup\{s \in [0,+\infty): \mathcal{H}^{s}(E)=+\infty\}=\inf\{t\in [0,+\infty): \mathcal{H}^{t}(E)=0\}
\]
is called the Hausdorff dimension of $E$.
\begin{cor} \label{cor 2.6} \textit{ Let $E \subset \mathbb{R}^{N}$, $\hdim(E)=1$ and $\mathcal{H}^{1}(E)<+\infty$. Then ${\rm Cap}_p(E)>0$ if and only if $p\in (N-1, +\infty)$.}
\end{cor}
\begin{proof}[Proof of Corollary \ref{cor 2.6}] If $p>N$, then by Remark~\ref{rem 2.5}, ${\rm Cap}_{p}(E)>0$. Assume by contradiction that ${\rm Cap}_{p}(E)=0$ for some $p \in (N-1,N]$. Taking $\varepsilon= (p-N+1)/2$  so that $0<N-p+\varepsilon<1$, by Theorem~\ref{thm 2.4} we get, $\mathcal{H}^{N-p+\varepsilon}(E)=0$, but this leads to a contradiction with the fact that $\hdim(E)=1$. On the other hand, if $p \in (1, N-1]$, then $\mathcal{H}^{N-p}(E)<+\infty$ and by Theorem~\ref{thm 2.4}, ${\rm Cap}_{p}(E)=0$. This completes the proof of Corollary~\ref{cor 2.6}.
	\end{proof}

\begin{prop} \label{prop lower bound for capacities} Let $\Sigma \subset \mathbb{R}^{N},\,x_{0}\in \mathbb{R}^{N}$, $0\leq r_{0}<r_{1}$ and $p\in (1,N]$. Assume that \[\Sigma\cap \partial B_{r}(x_{0})\neq \emptyset \,\ \text{for all}\,\ r \in (r_{0},r_{1}).\]
Then there exists a constant $C>0$, possibly depending only on $N$ and $p$, such that 
\[
{\rm Cap}_{p}(\{0\}^{N-1}\times [0,r_{1}-r_{0}])\leq C {\rm Cap}_{p}(\Sigma\cap B_{r_{1}}(x_{0})).
\]
\end{prop}
\begin{proof} The proof is straightforward if $p\in (1,N-1]$, since in this case ${\rm Cap}_{p}(\{0\}^{N-1}\times [0,r_{1}-r_{0}])=0$ according to Corollary~\ref{cor 2.6}. Assume that $p\in (N-1,N]$. Let $A(x_{0},r_{0})=\overline{B}_{r_{0}}(x_{0})$ if $r_{0}>0$ and $A(x_{0},r_{0})=\{x_{0}\}$ if $r_{0}=0$. For each $x \in \Sigma \cap (B_{r_{1}}(x_{0})\backslash A(x_{0},r_{0}))$, we define $\Phi(x)=(\{0\}^{N-1}, |x-x_{0}|)$. Since $\Phi$ is 1-Lipschitz, by \cite[Theorem 5.2.1]{Potential}, there exists $C=C(N,p)>0$ such that
\[
{\rm Cap}_{p}(\{0\}^{N-1}\times (r_{0},r_{1}))={\rm Cap}_{p}(\Phi(\Sigma \cap (B_{r_{1}}(x_{0})\backslash A(x_{0},r_{0}))))\leq C {\rm Cap}_{p}(\Sigma \cap (B_{r_{1}}(x_{0})\backslash A(x_{0},r_{0}))).
\]
Notice that ${\rm Cap}_{p}(\{0\}^{N-1}\times [r_{0},r_{1}])\leq {\rm Cap}_{p}(\{0\}^{N-1}\times (r_{0},r_{1}))$, since ${\rm Cap}_{p}(\cdot)$ is a subadditive set function (see, for instance, \cite[Proposition 2.3.6]{Potential}) and ${\rm Cap}_{p}(\{0\}^{N-1}\times\{r_{i}\})=0$ for $i=0,1$  by Theorem~\ref{thm 2.4}. So ${\rm Cap}_{p}(\{0\}^{N-1}\times [r_{0},r_{1}])\leq C {\rm Cap}_{p}(\Sigma \cap (B_{r_{1}}(x_{0})\backslash A(x_{0},r_{0})))$ for some $C=C(N,p)>0$. Then, using the fact that the Bessel capacity is nondecreasing with respect to set inclusion and, if necessary, the fact that it is an invariant under translations, we recover the desired estimate. This completes the proof of Proposition~\ref{prop lower bound for capacities}.
\end{proof}
\begin{cor}\label{cor lower bound for capacities} \textit{ Let $\Sigma \subset \mathbb{R}^{N}$, $x_{0}\in \mathbb{R}^{N}$, $0\leq r_{0}<r_{1}$ and $p\in (1,N]$. Assume that $\Sigma\cap\overline{B}_{r_{0}}(x_{0})\neq \emptyset$ if $r_{0}>0$ and $x_{0}\in \Sigma$ if $r_{0}=0$. Assume also that $(\Sigma\cap B_{r_{1}}(x_{0}))\cup \partial B_{r_{1}}(x_{0})$ is connected. Then there exists a constant $C>0$, possibly depending only on $N$ and $p$, such that
		\[
		{\rm Cap}_{p}(\{0\}^{N-1}\times [0,r_{1}-r_{0}])\leq C {\rm Cap}_{p}(\Sigma\cap B_{r_{1}}(x_{0})).
		\]
	}
\end{cor}
\begin{proof}[Proof of Corollary \ref{cor lower bound for capacities}] According to the conditions of Corollary~\ref{cor lower bound for capacities}, $\Sigma \cap \partial B_{r}(x_{0})\neq \emptyset$ for all $r \in (r_{0},r_{1})$. Then it only remains to use Proposition~\ref{prop lower bound for capacities}. This completes the proof of Corollary~\ref{cor lower bound for capacities}.
	\end{proof}
\begin{prop} \label{prop 2.11} Let $r \in (0,1]$ and $A_{r}=\{0\}^{N-1}\times [0, r]$. The following assertions hold.
	\begin{enumerate}[label=(\roman*)]
		\item If $p \in (N-1, N)$, then there exists a constant $C=C(N,p)>0$ such that
		\[
		r^{N-p} \leq C {\rm Cap}_{p}(A_{r}).
		\]
		\item If $p=N$, then there exists a constant $C=C(N)>0$ such that
		\[
		\biggl(\log\biggl(\frac{C}{r}\biggr)\biggr)^{1-p} \leq C {\rm Cap}_{p}(A_{r}).
		\]
	\end{enumerate}
\end{prop}
\begin{proof} Since $\diam(A_{r})\leq 1$, $(i)$ and $(ii)$ follows from \cite[Corollary 5.1.14]{Potential}.
\end{proof}

\begin{cor} \label{cor p-thickness} Let $p\in (N-1,N]$ and $\Sigma = (\{0\}^{N-1}\times (-1,1)) \cup \partial B_{1}$. Then there exist $r_{0},C_{0}>0$ such that 
\begin{equation}\label{estimate of cor p-thickness}
\frac{{\rm Cap}_{p}(\Sigma \cap B_{r}(x_{0}))}{{\rm Cap}_{p}(B_{r}(x_{0}))}\geq C_{0}
\end{equation}
whenever $0<r<r_{0}$ and $x_{0}\in \Sigma$.
\end{cor}
\begin{proof}[Proof of Corollary \ref{cor p-thickness}] Since $\Sigma$ is arcwise connected and $\diam(\Sigma)=1$, setting $r_{0}=1/2$, we observe that $\Sigma \cap \partial B_{r}(x_{0})\neq \emptyset$ whenever $0<r<r_{0}$ and $x_{0}\in \Sigma$. Then Proposition~\ref{prop lower bound for capacities} says that for some $C=C(N,p)>0$,  
	\[
	{\rm Cap}_{p}(\{0\}^{N-1}\times [0,r])\leq C {\rm Cap}_{p}(\Sigma\cap B_{r}(x_{0}))
	\]
whenever $0<r<r_{0}$ and $x_{0}\in \Sigma$. However, this, together with Proposition~\ref{prop 2.11}, \cite[Proposition 5.1.2]{Potential}, ~\cite[Proposition~5.1.3]{Potential} and \cite[Proposition 5.1.4]{Potential}, proves that there exists a constant $C_{0}>0$ such that the desired estimate (\ref{estimate of cor p-thickness}) holds for $C_{0}$ whenever $0<r<r_{0}$ and $x_{0}\in \Sigma$. This completes the proof of Corollary~\ref{cor p-thickness}.
\end{proof}
\begin{defn} \label{def 2.7} Let the function $u$ be defined $p$-q.e. on $\mathbb{R}^{N}$ or on some open subset. Then $u$ is said to be $p$-quasi continuous if for every $\varepsilon>0$ there is an open set $A$ with ${\rm Cap}_{p}(A)<\varepsilon$ such that the restriction of $u$ to the complement of $A$ is continuous in the induced topology.
\end{defn}

\begin{theorem} \label{thm 2.8} Let $Y\subset \mathbb{R}^{N}$ be an open set and $p \in (1,+\infty)$. Then for each $u \in W^{1,p}(Y)$ there exists a $p$-quasi continuous function $\widetilde{u} \in W^{1,p}(Y)$, which is uniquely defined up to a set of ${\rm Cap}_{p}$-capacity zero and $u=\widetilde{u}$ a.e. in $Y$.
\end{theorem}
\begin{proof} We refer the reader, for instance, to the proof of \cite[Theorem 2.8]{p-compl}, which actually applies for the general spatial dimension $N\geq 2$.
	\end{proof}

\begin{rem} \label{rem 2.9} A Sobolev function $u \in W^{1,p}(\mathbb{R}^{N})$ belongs to $W^{1,p}_{0}(Y)$ if and only if its $p$-quasi continuous representative $\widetilde{u}$ vanishes $p$-q.e. on $Y^{c}$ (see \cite[Theorem 4]{BAGBY} and \cite[Lemma 4]{Hedberg}). Thus, if $Y^{\prime}$ is an open subset of $Y$ and $u \in W^{1,p}_{0}(Y)$ such that $\widetilde{u}=0$ $p$-q.e. on $Y\backslash Y^{\prime}$, then the restriction of $u$ to $ Y^{\prime}$ belongs to $W^{1,p}_{0}(Y^{\prime})$ and conversely, if we extend a function $u \in W^{1,p}_{0}(Y^{\prime})$ by zero in $Y\backslash Y^{\prime}$, then $u \in W^{1,p}_{0}(Y)$. It is worth mentioning that if $\Sigma \subset \overline{Y}$ and ${\rm Cap}_{p}(\s)=0$, then $W^{1,p}_{0}(Y) = W^{1,p}_{0}(Y \backslash \s)$. Indeed, $u \in W^{1,p}_{0}(Y)$ if and only if $u \in W^{1,p}(\mathbb{R}^{N})$ and $\widetilde{u}=0$ $p$-q.e. on $Y^{c}$ that is equivalent to say $u\in W^{1,p}(\mathbb{R}^{N})$ and $\widetilde{u}=0$ $p$-q.e. on $Y^{c} \cup \s$ (since ${\rm Cap}_{p}(\Sigma)=0$ and ${\rm Cap}_{p}(\cdot)$ is a subadditive set function, see \cite[Proposition 2.3.6]{Potential}) or $u \in W^{1,p}_{0}(Y \backslash \s)$. In the sequel we shall always identify $u \in W^{1,p}(Y)$ with its $p$-quasi continuous representative $\widetilde{u}$.
\end{rem}

\begin{prop} \label{prop 2.10}  Let $D\subset \mathbb{R}^{N}$ be a bounded extension domain and let $u \in W^{1,p}(D)$. Consider $E=\overline{D} \cap \{x: u(x)=0\}$. If $\cp_{p}(E)>0$, then there exists a constant $C=C(N,p,D)>0$ such that 
\[
\int_{D}|u|^{p}\diff x \leq C ({\rm Cap}_p(E))^{-1} \int_{D} |\nabla u|^{p}\diff x.
\]
\end{prop}
\begin{proof} For a proof, see, for instance, \cite[Corollary 4.5.3, p. 195]{Ziemer}.
	\end{proof}

It is also worth recalling the following fact, which will be used several times in this paper.

\begin{rem}\label{arcwise connected} Every closed and connected set $\Sigma \subset \mathbb{R}^{N}$ satisfying $\mathcal{H}^{1}(\Sigma)<+\infty$ is arcwise connected (see, for instance, \cite[Corollary 30.2, p. 186]{David}).
\end{rem}


\subsection{Poincaré inequality}

\begin{prop} \label{prop 2.12} Let $\Sigma \subset \mathbb{R}^{N}$,\ $\xi \in \mathbb{R}^{N}$ and $r>0$ be such that $\Sigma \cap \partial B_{s}(\xi) \not = \emptyset$ for every $s\in [r,2r]$. Let $p\in (N-1, +\infty)$ and $u \in W^{1,p}(B_{2r}(\xi))$ satisfying $u=0\,\ p$-q.e. on $\Sigma \cap B_{2r}(\xi)$. Then there exists a constant $C=C(N,p)>0$ such that 
\[
\int_{B_{2r}(\xi)} |u|^{p} \diff x \leq C r^{p} \int_{B_{2r}(\xi)} |\nabla u|^{p} \diff x.
\]
\end{prop}
\begin{proof} We refer the reader to the proof of \cite[Corollary 2.12]{p-compl}, which also applies for the present geometric assumptions.
\end{proof}

\subsection{Estimate for $E_{f,\Omega}(u_{\Sigma})-E_{f,\Omega}(u_{\Sigma^{\prime}})$}
We begin by proving  the following ``localization lemma".
\begin{lemma}\label{lem 2.13} Let $\Omega \subset \mathbb{R}^{N}$ be open and bounded, $p \in (1,+\infty)$ and $f\in L^{q_{0}}(\Omega)$ with $q_{0}$ defined in (\ref{Eq 1.1}). Let $\s$ and $\s^{\prime}$ be closed proper subsets of $\overline{\Omega}$ and $x_{0}\in \mathbb{R}^N$. Assume that $0<r_{0}<r_{1}$ and $\s^{\prime} \Delta \s \subset \overline{B}_{r_{0}}(x_{0})$. Then there exists $C=C(p)>0$ such that for any $\varphi \in \lip(\mathbb{R}^{N})$ satisfying $\varphi=1$ over $B^{c}_{r_{1}}(x_{0}),\,\ \varphi=0$ over $B_{r_{0}}(x_{0})$ and $\|\varphi\|_{\infty}\leq 1$ on $\mathbb{R}^{N}$, one has
	\begin{equation*}
	E_{f,\Omega}(u_{\s})-E_{f,\Omega}(u_{\s^{\prime}})\leq C\int_{ B_{r_{1}}(x_{0}) } |\nabla u_{\s^{\prime}}|^{p}\diff x + C\int_{B_{r_{1}}(x_{0}) } | u_{\s^{\prime}}|^{p}|\nabla \varphi|^{p}\diff x +  \int_{B_{r_{1}}(x_{0})}fu_{\s^{\prime}}(1-\varphi)\diff x.
	\end{equation*}
\end{lemma}
\begin{proof} We refer the reader to the proof of \cite[Lemma 4.1]{p-compl}, that actually applies for the general spatial dimension $N\geq 2$.
\end{proof}

\begin{lemma} \label{lem 2.14} Let $\Omega \subset \mathbb{R}^{N}$ be open and bounded, $p\in (N-1,+\infty)$ and $f\in L^{q}(\Omega)$ with $q\geq q_{0}$, where $q_{0}$ is defined in (\ref{Eq 1.1}). Assume that  $\s$ is a closed arcwise connected proper subset of $\overline{\Omega}$ such that for some $x_{0}\in \mathbb{R}^{N}$ and $0<2r_{0}\leq r_{1}\leq 1$ it holds
	\begin{equation} 
	\s\cap B_{r_{0}}(x_{0})\not =\emptyset,\,\ \s\backslash B_{r_{1}}(x_{0}) \not = \emptyset. \label{2.1}
	\end{equation}
	Then for any $r\in [r_{0}, r_{1}/2]$, for any $\varphi \in \lip(\mathbb{R}^{N})$ such that $\|\varphi\|_{\infty}\leq 1$, $\varphi=1$ over $B^{c}_{2r}(x_{0}), \,\ \varphi= 0$ over $B_{r}(x_{0})$ and $\|\nabla \varphi\|_{\infty}\leq 1/r$, the following assertions hold.
	\begin{enumerate}[label=(\roman*)]
		\item There exists $C=C(N,p)>0$ such that
		\begin{equation}\label{2.2}
		\int_{B_{2r}(x_{0})}|u_{\s}|^{p}|\nabla \varphi|^{p}\diff x \leq C\int_{B_{2r}(x_{0})}|\nabla u_{\s}|^{p}\diff x. 
		\end{equation}
		\item There exists $C=C(N,p,q_{0},q,\|f\|_{q})>0$ such that
		\begin{equation}\label{2.3}
		\int_{B_{2r}(x_{0})} fu_{\s}(1-\varphi)\diff x \leq C\int_{B_{2r}(x_{0})}|\nabla u_{\Sigma}|^{p}\diff x + C r^{N+p^{\prime}-\frac{Np^{\prime}}{q}}.
		\end{equation}
	\end{enumerate}
\end{lemma}

\begin{proof} 
	Thanks to (\ref{2.1}), $\Sigma  \cap \partial B_{s}(x_{0}) \neq \emptyset$ for all $s \in [r,2r]$. Then, since   $u_{\Sigma}=0$ $p$-q.e. on $\Sigma$ and $u_{\Sigma}~\in~W^{1,p}(B_{2r}(x_{0}))$, we can use Proposition~\ref{prop 2.12}, which says that there exists $C=C(N,p)>0$ such that
	\begin{equation}
	\int_{B_{2r}(x_{0})}|u_{\Sigma}|^{p}\diff x\leq Cr^{p} \int_{B_{2r}(x_{0})}|\nabla u_{\Sigma} |^{p}\diff x. \label{2.4}
	\end{equation}
	Therefore, 
	\begin{equation*}
	\int_{B_{2r}(x_{0})} |u_{\Sigma}|^{p}|\nabla \varphi|^{p}\diff x \leq \frac{1}{r^{p}} \int_{B_{2r}(x_{0})} |u_{\Sigma}|^{p} \diff x \leq C \int_{B_{2r}(x_{0})} |\nabla u_{\Sigma}|^{p} \diff x.
	\end{equation*}
	This proves (\ref{2.2}).

	Let us now prove (\ref{2.3}). First, notice that thanks to (\ref{2.4}) and the fact that $2r\leq1$, there exists $C_{0}=C_{0}(N,p)>0$ such that
	\begin{equation}
	\|u_{\Sigma}\|_{W^{1,p}(B_{2r}(x_{0}))}\leq C_{0} \|\nabla u_{\Sigma}\|_{L^{p}(B_{2r}(x_{0}))}. \label {2.5}
	\end{equation}
	Next, using the Sobolev embeddings (see \cite[Theorem 7.26]{PDE}) together with (\ref{2.5}) and the fact that $u_{\Sigma}=0$\, $p$-q.e. on $\Sigma$, we deduce that there exists $\widetilde{C}=\widetilde{C}(N, p, q_{0})>0$ such that
	\begin{equation} \label{2.6}
	\|u_{\Sigma}\|_{L^{q^{\prime}_{0}}(B_{2r}(x_{0}))} \leq \widetilde{C} r^{\beta}\|\nabla u_{\Sigma}\|_{L^{p}(B_{2r}(x_{0}))},
	\end{equation}
	where
	\begin{equation*} 
	 \beta=0 \,\ \text{if} \,\  N-1<p<N, \,\ \,\ \beta=\frac{N}{q^{\prime}_{0}} \,\ \text{if} \,\ p=N, \,\ \,\ \beta=1-\frac{N}{p} \,\  \text{if} \,\ N<p<+\infty.
	\end{equation*}
	Thus, using the fact that $|f u_{\Sigma} (1-\varphi)|\leq |fu_{\Sigma}|$, H\"{o}lder's inequality, the estimate (\ref{2.6}) and Young's inequality, we get
	\begin{align*}
	\int_{B_{2r}(x_{0})}fu_{\Sigma}(1-\varphi)\diff x \leq \|f\|_{L^{q_{0}}(B_{2r}(x_{0}))}\|u_{\Sigma}\|_{L^{q^{\prime}_{0}}(B_{2r}(x_{0}))} &\leq |B_{2r}(x_{0})|^{\frac{1}{q_{0}}-\frac{1}{q}}\|f\|_{L^{q}(\Omega)}\|u_{\Sigma}\|_{L^{q^{\prime}_{0}}(B_{2r}(x_{0}))}\\
	&\leq Cr^{N(\frac{1}{q_{0}}-\frac{1}{q})+\beta}\|\nabla u_{\Sigma}\|_{L^{p}(B_{2r}(x_{0}))}\\&= Cr^{\frac{N}{p^{\prime}}+1-\frac{N}{q}}\|\nabla u_{\Sigma}\|_{L^{p}(B_{2r}(x_{0}))}\\
	&\leq C  r^{N+p^{\prime}-\frac{Np^{\prime}}{q}} +  C\|\nabla u_{\Sigma}\|_{L^{p}(B_{2r}(x_{0}))}^p, 
	\end{align*}
	where $C=C(N,p,q_{0},q,\|f\|_{q})>0$.   This concludes the proof of Lemma~\ref{lem 2.14}. \end{proof}

The following corollary  follows directly from Lemma \ref{lem 2.13} and Lemma \ref{lem 2.14}, thus, we omit the proof.

\begin{cor} \label{cor 2.15} \textit{  Let $\Omega \subset \mathbb{R}^{N}$ be open and bounded, $p\in (N-1,+\infty)$ and $f\in L^{q}(\Omega)$ with $q\geq q_{0}$, where $q_{0}$ is defined in (\ref{Eq 1.1}). Let $\s$ and $\s^{\prime}$ be closed arcwise connected proper subsets of $\overline{\Omega}$ and let $x_{0}\in \mathbb{R}^{N}$. Suppose that $0<2r_{0}\leq r_{1}\leq 1$, $\s^{\prime} \Delta \s \subset \overline{B}_{r_{0}}(x_{0})$ and}
	\[
	\Sigma^{\prime}\cap B_{r_{0}}(x_{0})\not =\emptyset,\,\ \Sigma^{\prime}\backslash B_{r_{1}}(x_{0}) \not = \emptyset.
	\]
	\textit{Then for every $r \in [r_{0}, r_{1}/2]$, 
		\begin{equation} \label{2.7}
		E_{f,\Omega}(u_{\s}) - E_{f, \Omega}(u_{\Sigma^{\prime}}) \leq C \int_{B_{2r}(x_{0})} |\nabla u_{\Sigma^{\prime}}|^{p}\diff x + Cr^{N+p^{\prime}-\frac{Np^{\prime}}{q}},
		\end{equation}
		where $C=C(N,p,q_{0},q,\|f\|_{q})>0$.}
\end{cor}

\subsection{Uniform boundedness of $u_{f,\Omega,\Sigma}$ with respect to $\Sigma$}
In this subsection, we prove a uniform bound, with respect to $\Sigma$, for a unique solution $u_{f,\Omega,\Sigma}$ to the Dirichlet problem $-\Delta_{p}u=f\,\ \text{in}\,\ \Omega\backslash \Sigma,\,\ u\in W^{1,p}_{0}(\Omega\backslash \Sigma)$. It is worth mentioning that the estimate (\ref{2.10}) will never be used in the sequel, however, we find it interesting enough to keep it in the present paper. Also notice that we can extend $u_{f,\Omega, \Sigma}$ by zero outside $\Omega\backslash \Sigma$ to an element of $W^{1,p}(\mathbb{R}^{N})$, we shall use the same notation for this extension as for $u_{f,\Omega,\Sigma}$. 
		\begin{prop} \label{prop 2.16} Let $\Om\subset \mathbb{R}^{N}$ be open and bounded, $\Sigma$ be a closed proper subset of $\overline{\Omega}$, $p \in (1,+\infty)$ and $f \in L^{q_{0}}(\Om)$ with $q_{0}$  defined in~(\ref{Eq 1.1}). Then there exists a constant $C>0$, possibly depending only on $N,\, p$ and $q_{0}$, such that
			\begin{equation}\label{2.8}
			\int_{\Om}|\nabla u_{f,\Omega,\Sigma}|^{p}\diff x \leq C |\Om|^{\alpha} \|f\|^{\beta}_{L^{q_{0}}(\Om)}, 
			\end{equation}
			where 
			\begin{equation} \label{2.9}
			(\alpha, \beta)=(0, p^{\prime})\,\ \text{if} \,\ 1<p<N, \,\
			(\alpha, \beta)=\biggl(\frac{N^{\prime}}{q^{\prime}_{0}}, N^{\prime}\biggr) \,\ \text{if} \,\  p=N, \,\
			(\alpha, \beta)=\biggl(\frac{p-N}{N(p-1)}, p^{\prime}\biggr)\,\ \text{if} \,\ p>N.
			\end{equation}
			Moreover, if $f\in L^{q}(\Omega)$ with $q>\frac{N}{p}$ if $p \in (1,N]$ and $q=1$ if $p>N$, then there exists a constant $C=C(N,p,q,\|f\|_{q},|\Om|)>0$ such that 
			\begin{equation}\label{2.10}
			\|u_{f,\Omega,\Sigma}\|_{L^{\infty}(\mathbb{R}^{N})} \leq C. 
			\end{equation} 
		\end{prop}
		\begin{proof} To establish the estimate (\ref{2.10}), we use \cite[Lemma A.2]{p-compl} with $U=\Omega \backslash \Sigma$, which provides a constant $C=C(N,p,q,\|f\|_{q},|U|)>0$ such that $\|u_{f,\Omega,\Sigma}\|_{L^{\infty}(\mathbb{R}^{N})} \leq C$, but observing that $C$ is increasing with respect to $|U|$, we recover (\ref{2.10}). Now let $f \in L^{q_{0}}(\Om)$. Using $u_{f,\Omega,\Sigma}$ as a test function in (\ref{1.3}), we get
			\begin{align*}
			\int_{\Om} |\nabla u_{f,\Omega,\Sigma}|^{p}\diff x = \int_{\Om} fu_{f,\Omega,\Sigma} \diff x \leq \|f\|_{L^{q_{0}}(\Om)} \|u_{f,\Omega,\Sigma}\|_{L^{q_{0}^{\prime}}(\Om)}, \numberthis \label{2.11}
			\end{align*}
			where the above estimate comes by using H\"{o}lder's inequality. Next, recalling that by the Sobolev inequalities (see \cite[Theorem 7.10]{PDE}) there is $C=C(N,p)>0$ such that $\|u_{f,\Omega,\Sigma}\|_{L^{q_{0}^{\prime}}(\Om)} \leq C|\Omega|^{\gamma} \|\nabla u_{f,\Omega,\Sigma}\|_{L^{p}(\Om)}$, where
			\begin{equation*}  
			\gamma=0 \,\ \,\  \text{if} \,\ \,\ 1<p<N, \qquad \gamma=\frac{1}{N}-\frac{1}{p}\,\ \,\ \text{if} \,\ \,\ p>N,
			\end{equation*}
			and using (\ref{2.11}), we recover (\ref{2.8}) in the case when $p\neq N$. If $p=N$ and $1<q_{0}\leq N$, then for $\varepsilon\in (0, N-1]$ such that $\frac{1}{q^{\prime}_{0}}=\frac{1}{N-\varepsilon}-\frac{1}{N}$, we get
			\begin{align*}
			\|u_{f,\Omega,\Sigma}\|_{L^{q^{\prime}_{0}}(\Omega)} &\leq C \|\nabla u_{f,\Omega,\Sigma}\|_{L^{N-\varepsilon}(\Omega)} \,\ (\text{by the Sobolev inequality})\\
			& \leq C|\Om|^{\frac{1}{q^{\prime}_{0}}} \|\nabla u_{f,\Omega,\Sigma}\|_{L^{N}(\Om)} \,\ (\text{by H\"{o}lder's inequality}).  
			\end{align*}
			The latter estimate together with (\ref{2.11}) yields (\ref{2.8}) in the case when $p=N$ and $1<q_{0}\leq N$. Assume now that $p=N$ and $q_{0}>N$. Then $q^{\prime}_{0}<N^{\prime}\leq N$. Using H\"{o}lder's inequality and the fact that \[\|u_{f,\Omega,\Sigma}\|_{L^{N^{\prime}}(\Omega)}\leq C|\Om|^{\frac{1}{N^{\prime}}} \|\nabla u_{f,\Omega,\Sigma}\|_{L^{N}(\Om)},\] which was proved above, we obtain that
			\[
			\|u_{f,\Omega,\Sigma}\|_{L^{q^{\prime}_{0}}(\Omega)} \leq |\Omega|^{\frac{1}{q^{\prime}_{0}}-\frac{1}{N^{\prime}}}\|u_{f,\Omega,\Sigma}\|_{L^{N^{\prime}}(\Omega)} \leq C|\Omega|^{\frac{1}{q^{\prime}_{0}}} \|\nabla u_{f,\Omega,\Sigma}\|_{L^{N}(\Omega)}.
			\]
			This, together with (\ref{2.11}), yields (\ref{2.8}) in the case when $p=N$ and $q_{0}>N$, and completes the proof of Proposition~\ref{prop 2.16}.  
		\end{proof}



\subsection{Existence}
\begin{theorem} \label{thm 2.17} Let $\Om\subset \mathbb{R}^{N}$ be open and bounded, $p\in (N-1,+\infty)$,  $f \in L^{q_{0}}(\Om)$ with $q_{0}$ defined in (\ref{Eq 1.1}). Let $({\s}_{n})_{n}\subset\mathcal{K}(\Omega)$ be a sequence converging to $\s \in \mathcal{K}(\Omega)$ in the Hausdorff distance. Then $u_{\s_{n}} \underset{n \to +\infty}{\longrightarrow} u_{\s}\,\ \text{strongly in}\,\ W^{1,p}(\Om).$
\end{theorem}
\begin{proof} For a proof, see \cite{sverak} for the case $N=p=2$ and \cite{Bucur} for the general case.
\end{proof}
\begin{prop} \label{prop 2.18}
Problem \ref{P 1.1} admits a minimizer.
\end{prop}
\begin{proof} 
	 Let $({\s}_{n})_{n}\subset \mathcal{K}(\Omega)$ be a minimizing sequence for Problem \ref{P 1.1}. We can assume that $\Sigma_{n}\neq \emptyset$ and $C_{f, \Omega}(\Sigma_{n})+\lambda \mathcal{H}^{1}(\Sigma_{n})\leq C_{f,\Omega}(\emptyset)$ at least for a subsequence still denoted by $n$, because otherwise the empty set would be a minimizer. Then, by Blaschke's theorem (see \cite[Theorem 6.1]{APD}), there exists $\Sigma \in \mathcal{K}(\Omega)$ such that, up to a subsequence still denoted by the same index, ${\s}_{n}$ converges to $\s$ in the Hausdorff distance as $n \to +\infty$. Furthermore, by Theorem~\ref{thm 2.17}, $u_{{\s_{n}}}$ converges to $u_{\s}$ strongly in $W^{1,p}_{0}(\Om)$ and hence $C_{f,\Omega}(\Sigma_{n})\to C_{f,\Omega}(\Sigma)$ as $n\to +\infty$. Then, using the fact that $\mathcal{H}^{1}$-measure is lower semicontinuous with respect to the topology generated by the Hausdorff distance, we deduce that $\s$ is a minimizer of Problem~\ref{P 1.1}.
\end{proof}
The next proposition says that, at least for some range of values of $\lambda>0$, solutions to Problem~1.1 are nontrivial.
\begin{prop} \label{prop 2.19} Let $\Omega \subset \mathbb{R}^{N}$ be open and bounded, $p\in (N-1, +\infty)$, $f \in L^{q_{0}}(\Omega),\,\ f\neq 0$ and $q_{0}$ is defined in (\ref{Eq 1.1}). Then there exists a number $\lambda_{0}=\lambda_{0}(N,p,f,\Omega)>0$ such that if Problem~\ref{P 1.1} is defined for $\lambda \in (0, \lambda_{0}]$, then every solution to this problem has positive $\mathcal{H}^{1}$-measure. Moreover, if $p>N$, then the empty set will not be a solution to Problem~\ref{P 1.1}.
\end{prop}
\begin{proof}
For a proof in the case when $N=2$, we refer the reader to \cite[Proposition 2.17]{p-compl}, the proof for the general case is similar.
\end{proof}

\section{Decay of the $p$-energy under flatness control}
In this section we establish the desired decay behavior of the $p$-energy $r\mapsto \int_{B_{r}(x_{0})}|\nabla u_{\Sigma}|^{p}\diff x$ by controlling the flatness of $\Sigma$ at $x_{0}$.

 We begin by establishing a control for the functional $r\mapsto \int_{B_{r}}|\nabla u|^{p}\diff x$, where $u$ is a weak solution to the $p$-Laplace equation in $B_{1}\backslash (\{0\}^{N-1}\times (-1,1))$ vanishing $p$-q.e. on $\{0\}^{N-1}\times (-1,1)$. In \cite{Lundstrom} it was shown that if $u$ is a positive $p$-harmonic function in $B_{1}\backslash(\{0\}^{N-1}\times(-1,1))$, continuous in $B_{1}$ with $u=0$ on $\{0\}^{N-1}\times (-1,1)$, then there exists $\delta=\delta(N,p)\in (0,1)$ such that $u \in C^{0,\beta}(B_{\delta})$, where $\beta=(p-N+1)/(p-1)$ and $C^{0,\beta}(U)$ denotes the space of H\"{o}lder continuous functions in the open set $U$. Furthermore, $\beta$ is the optimal H\"{o}lder exponent for $u$. In fact, comparing the function $u$ with the $p$-superharmonic and $p$-subharmonic functions constructed in \cite[Lemma~3.4]{Lundstrom}, \cite[Lemma 3.7]{Lundstrom7}, it was shown that there exists $C=C(N,p)>0$ and $\delta=\delta(N,p)\in (0,1)$ such that
 \begin{equation}\label{4.1}
 C^{-1}\dist(x, \{0\}^{N-1}\times (-1,1))^{\beta} \leq \frac{u(x)}{u(A_{1/2})} \leq C \dist(x, \{0\}^{N-1}\times (-1,1))^{\beta}
 \end{equation} 
whenever $x \in B_{\delta}$, where $A_{1/2}$ is a point in $\{|x^{\prime}|=1/2\}\cap \partial B_{1/2}$. The upper bound in (\ref{4.1}) implies that $u \in C^{0,\beta}(B_{\delta})$ (see \cite[Corollary 3.7]{Lundstrom}), and the lower bound proves that $\beta$ is optimal. 

However, for the purposes of this paper, the optimal regularity for a $p$-harmonic function vanishing on a 1-dimensional plane is not necessary. It so happened that for every $p\in (N-1,+\infty)$ we also constructed a nice barrier in order to estimate a nonnegative $p$-subharmonic function vanishing on a 1-dimensional plane. More precisely, for any fixed $\gamma \in (0, \beta)$ and some $\delta=\delta(N,p,\gamma) \in (0,1)$, we constructed a $p$-superharmonic function in Lemma \ref{lem A.1}, comparing with which a nonnegative $p$-subharmonic function $u$ in $B_{1}\backslash (\{0\}^{N-1}\times (-1,1))$, continuous in $B_{1}$ and with $u=0$ on $\{0\}^{N-1}\times (-1,1)$, we obtain the following control
\[
u(x)\leq Cu(A_{1/2})\dist(x, \{0\}^{N-1}\times (-1,1))^{\gamma},
\]
where $x \in B_{\delta}$ and $C=C(N,p, \gamma)>0$. If $\gamma$ is close enough to $\beta$, using the above control, we deduce the estimate (\ref{4.2}) which is sufficient to obtain the desired decay behavior of the $p$-energy under flatness control. Finally, since our barrier is slightly simpler than those in \cite[Lemma 3.4]{Lundstrom} and \cite[Lemma 3.7]{Lundstrom7} and in order to make the presentation self-contained, we shall use it in the proof of Lemma \ref{lem 4.1}.

\begin{lemma} \label{lem 4.1}
Let $p\in (N-1,+\infty)$. There exist $\alpha, \delta \in (0,1)$ and $C>0$, depending only on $N$ and $p$, such that if $u\in W^{1,p}(B_{1})$ is a weak solution to the $p$-Laplace equation in $B_{1}\backslash (\{0\}^{N-1}\times (-1,1))$ satisfying $u=0\,\ p\text{-q.e. on}\,\ \{0\}^{N-1}\times (-1,1)$, then
\begin{equation}\label{4.2}
\int_{B_{r}}|\nabla u|^{p}\diff x \leq Cr^{1+\alpha} \int_{B_{1}} |\nabla u|^{p} \diff x \,\ \text{for all} \,\ r \in (0, \delta]. 
\end{equation}
\end{lemma}
Before starting to prove Lemma \ref{lem 4.1}, we recall the following Carleson estimate (see \cite[Lemma 3.3]{Lundstrom}).
\begin{lemma} \label{lem 4.2} Let $p\in (N-1,+\infty)$. There exist $\varepsilon=\varepsilon(N,p)\in (0,1)$ and $C=~C(N,p)>0$ such that  if $u$ is a nonnegative $p$-harmonic function in 
	$B_{1}\backslash(\{0\}^{N-1} \times (-1,1))$, continuous in $B_{1}$ and satisfying $u=0$ on $\{0\}^{N-1}\times (-1,1)$, then 
	\[
	\max_{x\in \overline{B}_{\varepsilon}}u(x) \leq Cu(A_{\varepsilon}),
	\]
where $A_{\varepsilon}$ denotes a point such that $\dist(A_{\varepsilon}, \{0\}^{N-1}\times \mathbb{R})=\varepsilon$ and $A_{\varepsilon}\in \partial B_{\varepsilon}$.
\end{lemma}
\begin{proof}[Proof of Lemma \ref{lem 4.1}] To lighten the notation, we denote $\{0\}^{N-1}\times (-1,1)$ by $S$.\\
	\textit{Step 1.}  We prove the estimate (\ref{4.2}) in the case when $u$ is continuous and nonnegative in $B_{1}$ with $u=0$ on $S$. Let $\gamma=\frac{1}{2}\Bigl(\tfrac{p-N+1}{p}+\tfrac{p-N+1}{p-1}\Bigr)$. By Lemma \ref{lem A.1}, there exists $\delta_{0}=\delta_{0}(N,p)\in (0,1)$ such that $\hat{u}(x)=|x^{\prime}|^{\gamma}+x_{N}^{2}$ is $p$-superharmonic in $\{0<|x^{\prime}|<\delta_{0}\}\cap \{|x_{N}|<1\}$. On the other hand, according to 
	Lemma \ref{lem 4.2}, there exists $\varepsilon=\varepsilon(N,p) \in (0,1)$ and $C=C(N,p)>0$ such that $u\leq C u(A_{\varepsilon})$ in $\overline{B}_{\varepsilon}$, where $A_{\varepsilon}$ denotes a point with $\dist(A_{\varepsilon}, \{0\}^{N-1}\times \mathbb{R})=\varepsilon$ and $A_{\varepsilon}\in \partial B_{\varepsilon}$. Without loss of generality, we can assume that $\delta_{0}\leq \varepsilon$. Hereinafter in this proof, $C$ denotes a positive constant that can only depend on $N,\, p$ and can be different from line to line.  Using Harnack's inequality (see, for instance, \cite[Theorem 6.2]{NPT}), we deduce that $u(A_{\varepsilon})\leq C u(A_{1/2})$ and hence $u \leq Cu(A_{1/2})$ in $\overline{B}_{\delta_{0}}$ for a point $A_{1/2} \in \{|x^{\prime}|=1/2\}\cap \partial B_{1/2}$. Next, since 
	\[
	\hat{u}(x) = \biggl(\frac{\delta_{0}}{\sqrt{2}}\biggr)^{\gamma}+x_{N}^{2}\geq \biggl(\frac{\delta_{0}}{\sqrt{2}}\biggr)^{\gamma} \,\ \text{if} \,\ |x^{\prime}|=\frac{\delta_{0}}{\sqrt{2}} \,\ \text{and} \,\ \hat{u}(x)=|x^{\prime}|^{\gamma}+\frac{\delta_{0}^{2}}{2}\geq \frac{\delta_{0}^{2}}{2}\,\ \text{if} \,\ |x_{N}|=\frac{\delta_{0}}{\sqrt{2}},
	\]
	the estimate $u \leq C u(A_{1/2})\hat{u}$ holds on $\partial(\{|x^{\prime}|<\delta_{0}/\sqrt{2}\}\cap \{|x_{N}|<\delta_{0}/\sqrt{2}\})$; see Figure~\ref{Figure step 1}. 
	\begin{figure}
		\centering
		\includegraphics[width=.8\textwidth]{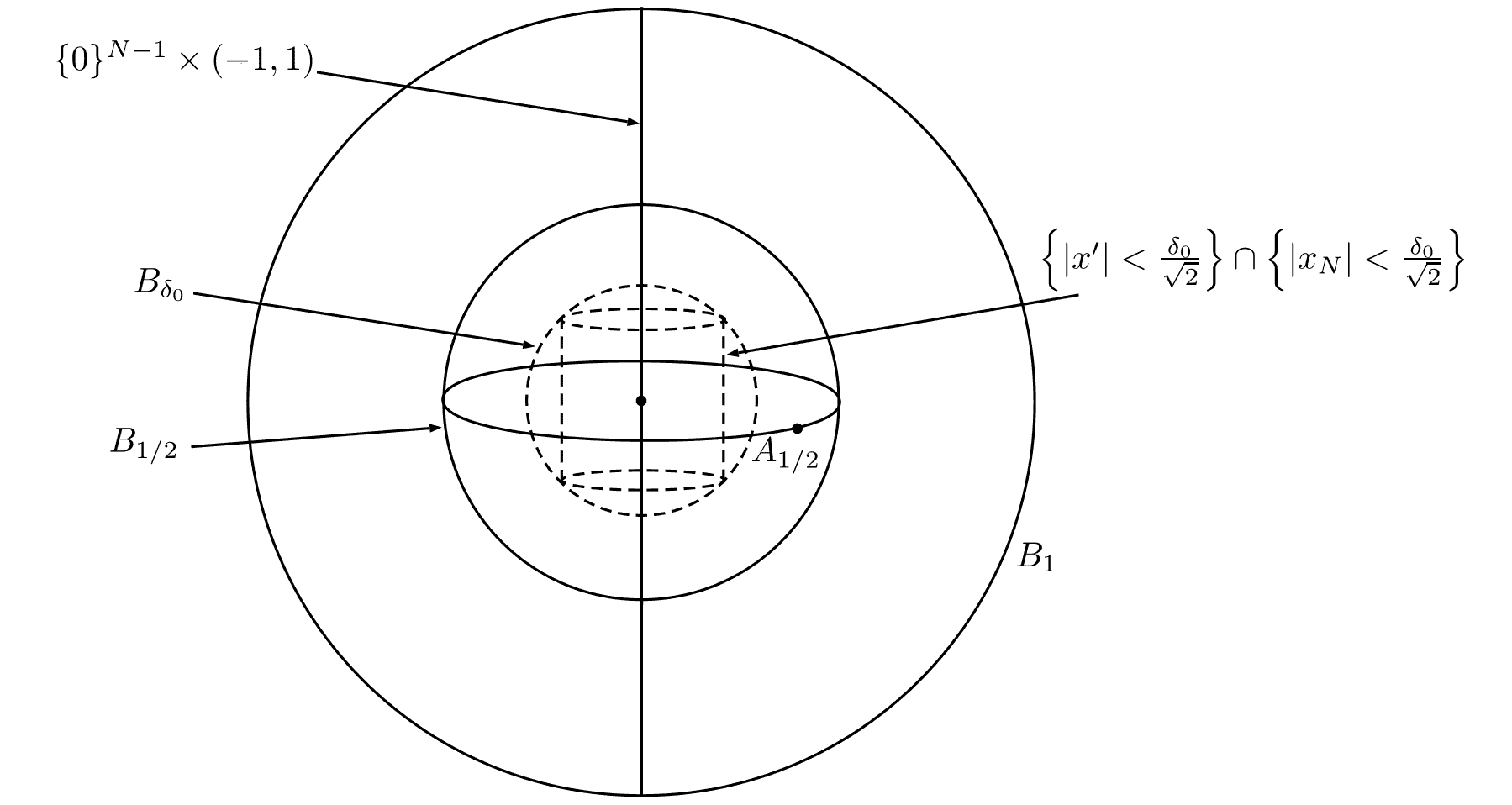}
		\caption{In the proof of Lemma~\ref{lem 4.1} we estimate on $\partial\Bigl(\Bigl\{|x^{\prime}|<\frac{\delta_{0}}{\sqrt{2}}\Bigr\}\cap \Bigl\{|x_{N}|<\frac{\delta_{0}}{\sqrt{2}}\Bigr\}\Bigr)$ a nonnegative ~$p$-harmonic function $u$ in $B_{1}\backslash (\{0\}^{N-1}\times (-1,1))$, continuous in $B_{1}$ with $u=0$ on $\{0\}^{N-1}\times (-1,1)$.}
		\label{Figure step 1}
	\end{figure}
Notice also that $u(x)\leq C u(A_{1/2})\hat{u}(x)$ if $x \in S$. Thus, using the comparison principle (see e.g. \cite[Theorem~7.6]{NPT}), we obtain
\begin{equation}\label{4.3}
u\leq Cu(A_{1/2})\hat{u}\,\ \text{in}\,\ \{|x^{\prime}|<\delta_{0}/\sqrt{2}\}\cap \{|x_{N}|<\delta_{0}/\sqrt{2}\}.
\end{equation} 
	Now we set $\delta:=\delta_{0}/10$. According to Lemma~\ref{lem A.2}, $u$ is a $p$-subharmonic function in $B_{1}$. Then, using Caccioppoli's inequality (see e.g. \cite[Lemma 2.9]{Lindqvist} or \cite[Lemma 3.27]{NPT}), which is applicable to nonnegative $p$-subharmonic functions, and also using (\ref{4.3}), for all $r \in (0, \delta]$, we deduce that
	\begin{align*}
	\int_{B_{r}}|\nabla u|^{p}\diff x \leq p^{p}r^{-p}\int_{B_{2r}}u^{p}\diff x \leq C u^{p}(A_{1/2}) r^{-p}\int_{B_{2r}}\hat{u}^{p}\diff x& \leq C u^{p}(A_{1/2})r^{-p}\int_{B_{2r}}(r^{\gamma}+r^{2})^{p}\diff x\\
	&\leq Cu^{p}(A_{1/2})r^{\gamma p +N -p}\\ &=Cu^{p}(A_{1/2})r^{1+\alpha},  \numberthis \label{4.4}
	\end{align*}
where $\alpha=\gamma p -p +N-1$ is positive, since $\gamma>(p-N+1)/p$. On the other hand, by Harnack's inequality, $u(A_{1/2})\leq Cu(x)$ for all $x \in B_{1/4}(A_{1/2})$ and then
\begin{align*}
u^{p}(A_{1/2})=\frac{1}{|B_{1/4}|}\int_{B_{1/4}(A_{1/2})}u^{p}(A_{1/2})\diff x \leq C\int_{B_{1/4}(A_{1/2})}u^{p}\diff x 
&\leq C \int_{B_{1}}u^{p}\diff x\\
&\leq C \int_{B_{1}}|\nabla u|^{p}\diff x, \numberthis \label{4.5}
\end{align*}
where we have used Proposition \ref{prop 2.12}. Gathering together (\ref{4.4}) and (\ref{4.5}), we deduce (\ref{4.2}). \\
\textit{Step 2.} We prove (\ref{4.2}) in the case when $u \in W^{1,p}(B_{1})$ and $u=0$ $p$-q.e. on $S$. Let us fix a sequence $(\varphi_{n})_{n\in \mathbb{N}} \subset C^{\infty}(\overline{B}_{1})$  such that for each $n\in \mathbb{N}$, $\varphi_{n}=0$ on $S$ and, furthermore, $\varphi_{n}\to u$ in $W^{1,p}(B_{1})$. Let us briefly explain  why such a sequence exists. For an arbitrary open set $U$ with $B_{1}\subset \subset U$, according to ~\cite[Theorem 7.25]{PDE}, there exists $\widetilde{u} \in W^{1,p}_{0}(U)$ such that $\widetilde{u}=u$ a.e. in $B_{1}$. By \cite[Theorem 6.1.4]{Potential}, $\widetilde{u}=u$ $p$-q.e. in $B_{1}$ and hence $\widetilde{u}=0$ $p$-q.e. on $S$. Then $\widetilde{u}\in W^{1,p}_{0}(U\backslash S)$ (see Remark~\ref{rem 2.9}). So there exists a sequence $(\varphi_{n})_{n} \subset C^{\infty}_{0}(U\backslash S)$ such that $\varphi_{n}\to \widetilde{u}$ in $W^{1,p}(U)$. It remains to note that $\varphi_{n}\to u$ in $W^{1,p}(B_{1})$. Next, for each $n\in \mathbb{N}$, let $u_{n}$ be a unique solution to the Dirichlet problem
\begin{equation*}
\begin{cases}
-\Delta_{p}v&=\,\ 0 \,\ \text{in} \,\ B_{1}\backslash S\\
\qquad v&=\,\ \varphi_{n} \,\ \text{on} \,\ S \cup \partial B_{1},
\end{cases}
\end{equation*}
which means that $u_{n}-\varphi_{n}\in W^{1,p}_{0}(B_{1}\backslash S)$ and 
\[
\int_{B_{1}}|\nabla u_{n}|^{p-2}\nabla u_{n}\nabla \zeta \diff x =0 \,\ \text{for all} \,\ \zeta \in W^{1,p}_{0}(B_{1}\backslash S).
\]
Notice that, by \cite[Theorem~2.19]{Lindqvist}, $u_{n}$ is continuous in $B_{1}\backslash S$. On the other hand, since $\varphi_{n}$ is continuous in $\overline{B}_{1}$, according to \cite[Theorem 6.27]{NPT}, we can show that $u_{n}$ is continuous in $\overline{B}_{1}$ if we can prove that there exist constants $C_{0}>0$ and $r_{0}>0$ such that
\begin{equation}\label{4.6}
\frac{{\rm Cap}_{p}((S \cup \partial B_{1})\cap B_{r}(x))}{{\rm Cap}_{p}(B_{r}(x))}\geq C_{0}
\end{equation}
whenever $0<r<r_{0}$ and $x\in S \cup \partial B_{1}$. However, the estimate (\ref{4.6}) in the case when $p \in (N-1,N]$ follows from Corollary~\ref{cor p-thickness}; in the case when $p>N$, using Remark~\ref{rem 2.5} and the fact that the Bessel capacity is an invariant under translations and is nondecreasing with respect to set inclusion, it is easy to see that the estimate (\ref{4.6}) holds for $C_{0}={\rm Cap}_{p}(\{0\})/{\rm Cap}_{p}(B_{1})$ whenever $x\in S \cup \partial B_{1}$ and $0<r<1$. Thus, for each $n \in \mathbb{N}$,  $u_{n}$ is continuous in $\overline{B}_{1}$. Then, by Lemma~\ref{lem A.2}, $u_{n}^{+}=\max\{u_{n}, 0\}$ and $u_{n}^{-}=-\min\{u_{n}, 0\}$ are nonnegative $p$-subharmonic functions in $B_{1}$ such that $u_{n}^{+}=u_{n}^{-}=0$ on $S$. Now let $v_{n}$ be a unique solution to the Dirichlet problem
\begin{equation*}
\begin{cases}
-\Delta_{p}v &= \,\ 0 \,\ \text{in} \,\ B_{1}\backslash S\\
\qquad v &= \,\ u_{n}^{+}\,\ \text{on} \,\ S \cup \partial B_{1}.
\end{cases}
\end{equation*} 
As before, by \cite[Theorem~2.19]{Lindqvist} and \cite[Theorem~6.27]{NPT}, $v_{n}$ is continuous in $\overline{B}_{1}$ and also $v_{n}=u_{n}^{+}$ on $S \cup \partial B_{1}$. Then, by the comparison principle, $u_{n}^{+}\leq v_{n}$ in $\overline{B}_{1}$. Let $\delta=\delta(N,p),\, \alpha=\alpha(N,p)\in (0,1)$ be the constants from \textit{Step 1}. Next, applying Caccioppoli's inequality to $u_{n}^{+}$, using the fact that $u_{n}^{+}\leq v_{n}$ in $\overline{B}_{1}$ and applying the result of \textit{Step 1} to $v_{n}$, for all $r \in (0, \delta]$ we deduce that
\begin{align*}
\int_{B_{r}}|\nabla u_{n}^{+}|^{p}\diff x \leq p^{p}r^{-p}\int_{B_{2r}}u^{+ p}_{n}\diff x \leq p^{p}r^{-p}\int_{B_{2r}}v_{n}^{ p}\diff x &\leq Cr^{1+\alpha} \int_{B_{1}}|\nabla v_{n}|^{p}\diff x \\
& \leq Cr^{1+\alpha} \int_{B_{1}}|\nabla u^{+}_{n}|^{p}\diff x, \numberthis\label{4.7}
\end{align*}
where the last estimate comes from the fact that $v_{n}$ minimizes the functional $v\mapsto \int_{B_{1}}|\nabla v|^{p}\diff x$ among all $v \in W^{1,p}(B_{1})$ such that $v-u_{n}^{+} \in W^{1,p}_{0}(B_{1}\backslash S)$ (see Theorem \ref{thm 2.2}) and $u_{n}^{+}$ is a competitor. Arguing by the same way as for $u_{n}^{+}$, we deduce that for all $r \in (0, \delta]$,
\begin{equation}\label{4.8}
\int_{B_{r}}|\nabla u_{n}^{-}|^{p}\diff x \leq Cr^{1+\alpha}\int_{B_{1}}|\nabla u_{n}^{-}|^{p}\diff x.
\end{equation}
Next, since $\varphi_{n} \to u$ in $W^{1,p}(B_{1})$ and $u$ solves the Dirichlet problem $-\Delta_{p}v=0\,\ \text{in}\,\ B_{1}\backslash S$ with its own trace on $S\cup \partial B_{1}$, by \cite[Theorem~3.5]{Bucur}, $u_{n}\to u$ in $W^{1,p}(B_{1})$ and hence $u_{n}^{+} \to u^{+}$, $u_{n}^{-}\to u^{-}$ in $W^{1,p}(B_{1})$. This, together with (\ref{4.7}) and (\ref{4.8}), implies that for all $r \in (0, \delta]$,
\begin{align*}
\int_{B_{r}}|\nabla u|^{p}\diff x = \int_{B_{r}}|\nabla u^{+}-\nabla u^{-}|^{p}\diff x &\leq 2^{p-1}\int_{B_{r}}|\nabla u^{+}|^{p}\diff x + 2^{p-1}\int_{B_{r}}|\nabla u^{-}|^{p}\diff x \\
& \leq C r^{1+\alpha}\int_{B_{1}}|\nabla u^{+}|^{p}\diff x+ Cr^{1+\alpha}\int_{B_{1}}|\nabla u^{-}|^{p}\diff x\\
&\leq Cr^{1+\alpha}\int_{B_{1}}|\nabla u|^{p}\diff x.
\end{align*}
This completes the proof of Lemma \ref{lem 4.1}.
\end{proof}

Now we establish an estimate for a weak solution to the $p$-Laplace equation in $B_{r}(x_{0})\backslash \Sigma$ that vanishes on $\Sigma \cap B_{r}(x_{0})$ in the case when $\Sigma$ is close enough, in $\overline{B}_{r}(x_{0})$ and in the Hausdorff distance, to a diameter of $\overline{B}_{r}(x_{0})$.

\begin{lemma}\label{lem 4.3}
 Let $p\in (N-1,+\infty)$ and let $\alpha, \delta \in (0,1)$, $C>1$ be as in Lemma \ref{lem 4.1}. Then for each $\varrho \in (0,\delta]$ there exists $\varepsilon_{0} \in (0, \varrho)$ such that the following holds. Let $\s \subset \mathbb{R}^{N}$ be a closed set such that $(\s \cap B_{r}(x_{0})) \cup \partial B_{r}(x_{0})$ is connected and assume that for some affine line $L$ passing through $x_{0}$, $d_{H}(\s\cap \overline{B}_{r}(x_{0}), L\cap \overline{B}_{r}(x_{0}))\leq \varepsilon_{0}r$. Then for any weak solution $u\in W^{1,p}(B_{r}(x_{0}))$ to the $p$-Laplace equation in $B_{r}(x_{0})\backslash\s$ vanishing $p\text{-q.e. on}\,\ \s\cap B_{r}(x_{0})$, the following estimate holds
\begin{equation*}
\int_{B_{\varrho r}(x_{0})} |\nabla u|^{p}\diff x\leq (C \varrho)^{1+\alpha}\int_{B_{r}(x_{0})} |\nabla u|^{p} \diff x.
\end{equation*}
\end{lemma}
\begin{figure}[H]
	\centering
	\includegraphics[width=.777\textwidth]{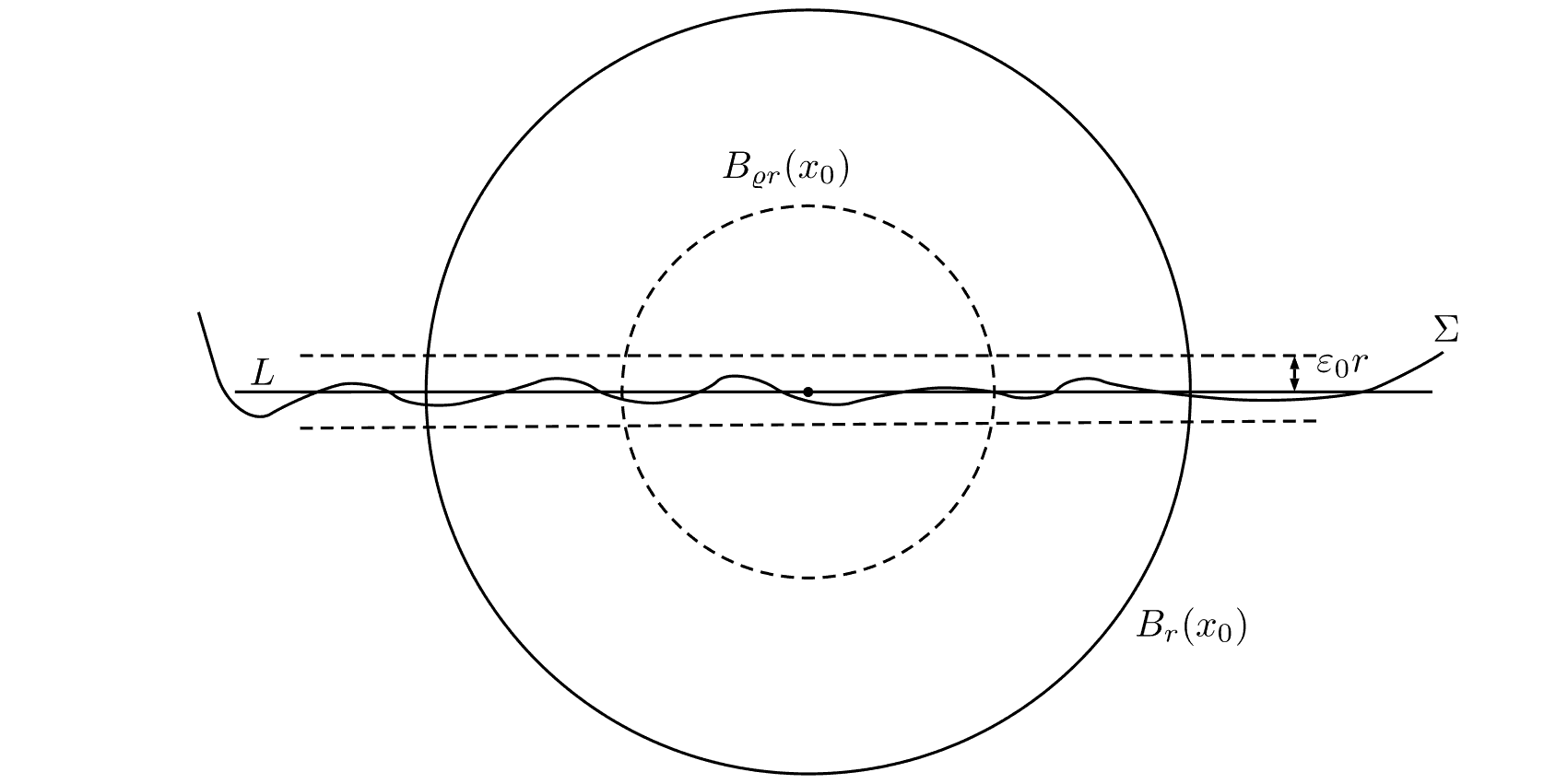}
	\caption{The geometry in Lemma \ref{lem 4.3}.}
	\label{Figure step 2}
\end{figure}
\begin{proof} Since the $p$-Laplacian is an invariant under scalings, rotations and translations, we can assume that $B_{r}(x_{0})=B_{1}$ and $L\cap \overline{B}_{r}(x_{0})=\{0\}^{N-1}\times [-1,1]$. To simplify the notation, we denote $\{0\}^{N-1}\times [-1,1]$ by $S$. By contradiction, suppose that for some $\varrho \in (0,\delta]$ there exist sequences $(\varepsilon_{n})_{n},$ $(\s_{n})_{n}$ and $(u_{n})_{n}$ such that for each $n\in \mathbb{N}$: $\varepsilon_{n}\in (0,\varrho)$, $\varepsilon_{n}\downarrow 0$ as $n \to +\infty$; $\s_{n}$ is closed, $(\s_{n}\cap B_{1})\cup \partial B_{1}$ is connected, $d_{H}(\s_{n}\cap \overline{B}_{1}, S)\leq \varepsilon_{n}$ implying that
	\begin{equation}
	d_{H}({\s}_{n}\cap \overline{B}_{1}, S)\to 0 \,\ \text{as}\,\ n\to +\infty; \label{4.9}
	\end{equation}
	$u_{n}$ is a weak solution to the $p$-Laplace equation in $B_{1}\backslash \s_{n},\,\ u_{n}=0\,\ p$-q.e. on $\s_{n}\cap B_{1}$ and 
			\begin{equation}
			\int_{B_{\varrho}}|\nabla u_{n}|^{p}\diff x>(C\varrho)^{1+\alpha}\int_{B_{1}}|\nabla u_{n}|^{p}\diff x. \label{4.10} 
			\end{equation} 
	Next, for each $n\in \mathbb{N}$, we define $v_{n} \in W^{1,p}(B_{1})$ as
	\begin{equation}
	v_{n}(\cdot)=\frac{u_{n}(\cdot)}{\Bigl(\int_{B_{1}}|\nabla u_{n}|^{p}\diff x\Bigr)^{\frac{1}{p}}}. \label{4.11}
	\end{equation}
	Notice that $v_{n}=0\,\ p \text{-q.e. on}\,\ \s_{n}\cap B_{1}$ and 
	\begin{equation}
	\int_{B_{1}}|\nabla v_{n}|^{p}\diff x=1. \label{4.12}
	\end{equation}
	On the other hand, for each $n\in \mathbb{N}$, $\Sigma_{n}\cap B_{\delta}\neq \emptyset$. This, together with the fact that $(\Sigma_{n}\cap B_{1})\cup \partial B_{1}$ is connected, according to Corollary~\ref{cor lower bound for capacities} and Proposition~\ref{prop 2.11} in the case when $p\in (N-1,N]$, and according to Remark~\ref{rem 2.5} in the case when $p\in (N,+\infty)$, implies that there exists a constant $\widetilde{C}>0$ (independent of $n$) such that for each $n\in \mathbb{N}$,
	\begin{equation*}
	{\rm Cap}_{p}({\s}_{n}\cap B_{1})\geq \widetilde{C}.
	\end{equation*}
	Using the above estimate together with Proposition~\ref{prop 2.10} and with (\ref{4.12}), we conclude that the sequence $(v_{n})_{n}$ is bounded in $W^{1,p}(B_{1})$. Hence, up to a subsequence still denoted by the same index, we have
	\begin{align}
	v_{n} & \rightharpoonup v \,\ \text{weakly in} \,\ W^{1,p}(B_{1}) \label{4.13} \\ 
	v_{n} & \to v \,\ \text{strongly in} \,\ L^{p}(B_{1}), \label{4.14}
	\end{align}
	for some $v \in W^{1,p}(B_{1})$. Let us now show that $v=0\,\ p \text{-q.e. on}\,\ S\cap B_{1}$. For each $t\in (0,1)$, we fix a function $\psi \in C^{1}_{0}(B_{1})$ such that $\psi=1$ on $\overline{B}_{t}$ and $0\leq \psi \leq 1$. Since $(\s_{n}\cap B_{1})\cup \partial B_{1}$ is connected for each $n \in \mathbb{N}$ and $d_{H}(\s_{n}\cap \overline{B}_{1}, S) \to 0$ as $n\to +\infty$, it follows (see Section 6 in \cite{Bucur}) that the sequence of Sobolev spaces $W^{1,p}_{0}(B_{1}\backslash \s_{n})$ converges in the sense of Mosco to $W^{1,p}_{0}(B_{1}\backslash S)$. Notice that for each $n \in \mathbb{N}$, $v_{n}\psi \in W^{1,p}_{0}(B_{1}\backslash \Sigma_{n})$ and by (\ref{4.13}), $v_{n}\psi \rightharpoonup v\psi$\,\ weakly in\,\ $W^{1,p}(\mathbb{R}^{N})$. Then, using the definition of limit in the sense of Mosco, we deduce that $v\psi \in W^{1,p}_{0}(B_{1}\backslash S)$. This implies that $v=0\,\ p\text{-q.e. on}\,\ \{0\}^{N-1}\times [-t,t]$ (see Remark \ref{rem 2.9}). As $t\in (0,1)$ was arbitrarily chosen, $v=0\,\ p \text{-q.e. on}\,\ S\cap B_{1}$.
	
	We claim that $v$ is a weak solution to the $p$-Laplace equation in $B_{1}\backslash S$, that is,
	\begin{equation}\label{4.15}
	\int_{B_{1}}|\nabla v|^{p-2}\nabla v \nabla \varphi\diff x=0 \,\ \text{for all} \,\ \varphi \in C^{\infty}_{0}(B_{1}\backslash S).
	\end{equation}	
	In order to get the equality above, it suffices to show that $|\nabla v_{n}|^{p-2}\nabla v_{n} \rightharpoonup |\nabla v|^{p-2}\nabla v$ weakly in $L^{p^{\prime}}(B_{1};\mathbb{R}^{N})$. In fact, if $\varphi \in C^{\infty}_{0}(B_{1}\backslash S)$, then $\{\varphi \neq 0\} \subset \subset B_{1}\backslash S$ and thanks to (\ref{4.9}), for all $n$ large enough, $\{\varphi \neq 0\}\subset\subset B_{1}\backslash \Sigma_{n}$, so we can write the following
	\[
	\int_{B_{1}}|\nabla v_{n}|^{p-2}\nabla v_{n}\nabla \varphi \diff x = 0.
	\]
	Next, letting $n$ tend to $+\infty$ in the above equality and using that $|\nabla v_{n}|^{p-2}\nabla v_{n}\rightharpoonup |\nabla v|^{p-2}\nabla v$ weakly in $L^{p^{\prime}}(B_{1};\mathbb{R}^{N})$, we would obtain (\ref{4.15}). We first prove that, at least for a subsequence, $\nabla v_{n}\to \nabla v$ a.e. in $B_{1}$. For each integer $m\geq 10$, we define $\Om_{m}:=\{x \in B_{1}: \dist(x, S)>1/m\}$. Notice that $v_{n}\rightharpoonup v$ weakly in $W^{1,p}(\Om_{m})$ and for all $n$ large enough (with respect to $m$), $v_{n}$ is a weak solution in $\Om_{m}$. Then, according to ~\cite[Theorem 2.1]{BOCCARDO}, there exists a subsequence $(v_{n(m,k)})_{k\in \mathbb{N}}$ such that $\nabla v_{n(m,k)}\to \nabla v$ a.e. in $\Om_{m}$. For each $m$ as above, let $(v_{n(m+1,k)})_{k\in \mathbb{N}}$ be a subsequence of $(v_{n(m,k)})_{k\in \mathbb{N}}$ satisfying $\nabla v_{n(m+1,k)}\to \nabla v$ a.e. in $\Omega_{m+1}$. Thus, for the diagonal subsequence $(v_{n(m,m)})_{m\in \mathbb{N}}$, $\nabla v_{n(m,m)} \to \nabla v$ a.e. in $B_{1}$. So, at least for a subsequence, $\nabla v_{n} \rightharpoonup \nabla v$ weakly in $L^{p}(B_{1};\mathbb{R}^{N})$ and $\nabla v_{n} \to \nabla v$ a.e. in $B_{1}$, but then, using Mazur's lemma, we deduce that $|\nabla v_{n}|^{p-2}\nabla v_{n} \rightharpoonup |\nabla v|^{p-2}\nabla v$ weakly in $L^{p^{\prime}}(B_{1};\mathbb{R}^{N})$. This proves the claim.

	We now  want to prove the strong convergence of $\nabla v_n$ to $\nabla v$ in $L^p(B_{\delta};\mathbb{R}^{N})$. Since $\nabla v_{n}\rightharpoonup \nabla v$ weakly in $L^{p}(B_{\delta};\mathbb{R}^{N})$, we only need to prove that  $\|\nabla v_{n}\|_{L^{p}(B_{\delta};\mathbb{R}^{N})}$ tends to $\|\nabla v\|_{L^{p}(B_{\delta};\mathbb{R}^{N})}$. By the weak convergence, we already have that
			$$\int_{B_{\delta}}|\nabla v|^{p}\diff x \leq \liminf_{n\to +\infty} \int_{B_{\delta}}|\nabla v_n|^{p}\diff x.$$
			Thus, it remains to prove the reverse inequality with a limsup. For this, for an arbitrary $\varepsilon \in (\delta, 1)$, we fix $\chi_{\varepsilon} \in C^{\infty}_{0}(B_{1})$ such that $0\leq \chi_{\varepsilon} \leq 1$, $\chi_{\varepsilon}=1$ on $\overline{B}_{\delta}$, $\chi_{\varepsilon}=0$ on $B^{c}_{\varepsilon}$ and $\|\nabla \chi_{\varepsilon}\|_{\infty} \leq 2/(\varepsilon-\delta)$. Notice that $v_{n}\chi_{\varepsilon} \in W^{1,p}_{0}(B_{1}\backslash \Sigma_{n})$. Then, since $v_{n} \in W^{1,p}(B_{1})$ is a weak solution in $B_{1}\backslash \Sigma_{n}$ and $\chi_{\varepsilon}=0$ on $B^{c}_{\varepsilon}$,
			\begin{align*}
			\int_{B_{\varepsilon}}\chi_{\varepsilon} |\nabla v_{n}|^{p}\diff x = - \int_{B_{\varepsilon}}v_{n}|\nabla v_{n}|^{p-2} \nabla v_{n} \nabla \chi_{\varepsilon} \diff x.
			\end{align*}
			On the other hand, from the fact that $\|\nabla \chi_{\varepsilon}\|_{\infty}\leq 2/(\varepsilon-\delta)$, (\ref{4.12}), (\ref{4.14}) and since $|\nabla v_{n}|^{p-2}\nabla v_{n}$ weakly converges to $|\nabla v|^{p-2}\nabla v$ in $L^{p^{\prime}}(B_{\varepsilon};\mathbb{R}^{N})$, it follows that
			\[
			\lim_{n \to +\infty} - \int_{B_{\varepsilon}} v_{n} |\nabla v_{n}|^{p-2}\nabla v_{n} \nabla \chi_{\varepsilon} \diff x = - \int_{B_{\varepsilon}} v |\nabla v|^{p-2} \nabla v \nabla \chi_{\varepsilon} \diff x=\int_{B_{\varepsilon}}\chi_{\varepsilon} |\nabla v|^{p}\diff x,
			\]
			where to get the latter equality we have used that $v \in W^{1,p}(B_{1})$ is a weak solution to the $p$-Laplace equation in $B_{1}\backslash S$, $v\chi_{\varepsilon} \in W^{1,p}_{0}(B_{1}\backslash S)$ and $\chi_{\varepsilon}=0$ on $B^{c}_{\varepsilon}$. So we obtain that 
			$$
			\limsup_{n \to +\infty} \int_{B_{\delta}} |\nabla v_{n}|^{p}\diff x \leq \limsup_{n\to +\infty}\int_{B_{\varepsilon}}\chi_{\varepsilon}|\nabla v_{n}|^{p}\diff x = \int_{B_{\varepsilon}}\chi_{\varepsilon}|\nabla v|^{p}\diff x.
			$$
			Next, letting $\varepsilon$ tend to $\delta+$ and using Lebesgue's dominated convergence theorem, we get
			\begin{equation*}
			\limsup_{n\to +\infty} \int_{B_{\delta}}|\nabla v_{n}|^{p}\diff x \leq  \int_{B_{\delta}}|\nabla v|^{p}\diff x.
			\end{equation*}
			Thus, we have proved the strong convergence of $\nabla v_n$ to $\nabla v$ in $L^p(B_{\delta};\mathbb{R}^{N})$. Using (\ref{4.10}), (\ref{4.11}) and passing to the limit, we therefore arrive at
	\begin{equation}
	\int_{B_{\varrho}}|\nabla v|^{p} \diff x \geq (C \varrho)^{1+\alpha}. \label{4.16}
	\end{equation}
	However, Lemma \ref{lem 4.1}, together with (\ref{4.12}) and (\ref{4.13}), says the following
	\[
	\int_{B_{\varrho}}|\nabla v|^{p}\diff x \leq C \varrho^{1+\alpha},
	\]
	which leads to a contradiction with (\ref{4.16}), since $\alpha, \varrho>0$ and $C>1$. This completes the proof of Lemma~\ref{lem 4.3}. 
	
\end{proof}

Now we want to establish an estimate for a weak solution to the $p$-Poisson equation in $B_{r}(x_{0})\backslash \Sigma$ that vanishes on $\Sigma \cap B_{r}(x_{0})$ in the case when $\Sigma$ is sufficiently close, in $\overline{B}_{r}(x_{0})$ and in the Hausdorff distance, to a diameter of $\overline{B}_{r}(x_{0})$. For that purpose, in the following lemma we control the difference between a weak solution to the $p$-Poisson equation and its $p$-Dirichlet replacement in a ball with a crack.

\begin{lemma}  \label{lem 4.4}
	Let $p\in (N-1, +\infty)$ and $f \in L^{q}(B_{r_{1}}(x_{0}))$ with $q> q_{0}$, where $q_{0}$ is defined in (\ref{Eq 1.1}). Let $\s$ be a closed arcwise connected set in $\mathbb{R}^{N}$ and $0<2 r_{0}\leq r_{1}\leq 1$ satisfy
	\begin{equation}\label{4.17}
	\s \cap B_{r_{0}}(x_{0})\neq \emptyset, \,\ \s \backslash B_{r_{1}}(x_{0}) \neq \emptyset \,\ \text{and} \,\ B_{r_{1}}(x_{0})\backslash \Sigma \neq \emptyset. 
	\end{equation}
	Let $u \in W^{1,p}(B_{r_{1}}(x_{0}))$ satisfying $u=0\,\ p\text{-q.e. on}\,\ \s\cap B_{r_{1}}(x_{0})$ be a unique solution to the Dirichlet problem 
	\[
	-\Delta_{p}v = f \,\ \text{in}\,\ B_{r_{1}}(x_{0})\backslash \s,
	\]
	which means that 
	\begin{equation}\label{4.18}
	\int_{B_{r_{1}}(x_{0})}|\nabla u|^{p-2}\nabla u \nabla \varphi \diff x = \int_{B_{r_{1}}(x_{0})} f\varphi \diff x \,\ \text{for all}\,\ \varphi \in W^{1,p}_{0}(B_{r_{1}}(x_{0})\backslash \s). 
	\end{equation}
	Let $w \in W^{1,p}(B_{r_{1}}(x_{0}))$ satisfying $w=0\,\ p\text{-q.e. on}\,\ \s\cap B_{r_{1}}(x_{0})$ be a unique solution to the Dirichlet problem
	\begin{equation*}
	\begin{cases}
	-\Delta_{p} v &=\,\  0  \,\ \text{in}\,\ B_{r_{1}}(x_{0})\backslash \s \\
	\qquad v & =\,\ u \ \text{on}\,\ (\s\cap B_{r_{1}}(x_{0})) \cup \partial B_{r_{1}}(x_{0}),
	\end{cases}
	\end{equation*}
	which means that $w - u \in W^{1,p}_{0}(B_{r_{1}}(x_{0})\backslash \s)$ and
	\begin{equation}
	\int_{B_{r_{1}}(x_{0})}|\nabla w|^{p-2}\nabla w \nabla \varphi \diff x = 0 \,\ \text{for all}\,\ \varphi \in W^{1,p}_{0}(B_{r_{1}}(x_{0})\backslash \s). \label{4.19}
	\end{equation}
	If $2\leq p<+\infty$, then
	\begin{equation}\label{4.20}
	\int_{B_{r_{1}}(x_{0})}|\nabla u-\nabla w|^{p}\diff x \leq Cr^{N+p^{\prime}-\frac{Np^{\prime}}{q}}_{1}, 
	\end{equation}
	where $C=C(N, p, q_{0}, q, \|f\|_{q})>0$.\\
	If $1<p <2$, then
	\begin{equation}\label{4.21}
	\int_{B_{r_{1}}(x_{0})}|\nabla u- \nabla w|^{p}\diff x \leq C (I(u))^{p} (r^{p-1}_{1})^{2+p^{\prime}-\frac{2p^{\prime}}{q}}, 
	\end{equation}
	where $C=C(p, q_{0},q,  \|f\|_{q})>0$ and $I(u)=2^{\frac{2}{p}}\Bigl(\int_{B_{r_{1}}(x_{0})}|\nabla u|^{p}\diff x \Bigr)^{\frac{2-p}{p}}$.
\end{lemma}
\begin{rem}
	Observe that for any $N\geq 2$ and any $p\in (N-1,+\infty)$, $N+p^{\prime}-Np^{\prime}/q$ is positive if $q>q_{0}$, where $q_{0}$ is defined in (\ref{Eq 1.1}).
\end{rem}
\begin{proof} We provide a proof of the estimate (\ref{4.20}), for a proof of the estimate (\ref{4.21}) see \cite[Lemma 4.9]{p-compl}. For convenience, we define $z=u-w$. Thanks to (\ref{4.17}) and the fact that $z=0$ $p$-q.e. on $\Sigma \cap B_{r_{1}}(x_{0})$, by Proposition \ref{prop 2.12}, there exists $C_{0}=C_{0}(N,p)>0$ such that
	\[
	\|z\|_{L^{p}(B_{r_{1}}(x_{0}))} \leq C_{0} r_{1}\|\nabla z\|_{L^{p}(B_{r_{1}}(x_{0}))}.
	\]
	Since $r_{1}\leq 1$, the above estimate leads to the following
	\begin{equation}
	\|z\|_{W^{1,p}(B_{r_{1}}(x_{0}))}\leq C \|\nabla z\|_{L^{p}(B_{r_{1}}(x_{0}))}, \label {4.22}
	\end{equation}
	where $C=C(N,p)>0$. Then, using the Sobolev embeddings (see \cite[Theorem 7.26]{PDE}) together with (\ref{4.22}) and in the case when $N<p<+\infty$ using also that $z(\xi)=0$ for some $\xi \in \Sigma \cap B_{r_{1}}(x_{0})$, yielding that $|z(x)|=|z(x)-z(\xi)|\leq C^{\prime}(2r_{1})^{1-\frac{N}{p}}\|z\|_{W^{1,p}(B_{r_{1}}(x_{0}))}$ for some $C^{\prime}=C^{\prime}(N,p)>0$, we deduce the following
	\begin{equation} \label{4.23}
	\|z\|_{L^{q^{\prime}_{0}}(B_{r_{1}}(x_{0}))} \leq \widetilde{C} r^{\alpha}_{1}\|\nabla z\|_{L^{p}(B_{r_{1}}(x_{0}))},
	\end{equation}
	where $\widetilde{C}=\widetilde{C}(N, p, q_{0})>0$ and
	\begin{equation*}
	 \alpha=0 \,\ \text{if} \,\  2\leq N-1<p<N,\,\ \,\ \alpha=\frac{N}{q^{\prime}_{0}} \,\ \text{if} \,\ p=N, \,\ \,\  \alpha=1-\frac{N}{p} \,\  \text{if} \,\ N<p<+\infty .
	\end{equation*}
	Next, according to \cite[Lemma 2.2]{ABC}, there exists $c_{0}=c_{0}(p)>0$ such that,
	\[
	\int_{B_{r_{1}}(x_{0})}|\nabla z|^{p}\diff x \leq c_{0}\int_{B_{r_{1}}(x_{0})}\langle |\nabla u|^{p-2}\nabla u-|\nabla w|^{p-2}\nabla w, \nabla z\rangle \diff x, 
	\]
	and, since $z$ is a test function for (\ref{4.18}) and (\ref{4.19}), we get 
	\begin{align*}
	\int_{B_{r_{1}}(x_{0})}|\nabla z|^{p}\diff x  \leq c_{0}\int_{B_{r_{1}}(x_{0})}\langle|\nabla u|^{p-2}\nabla u -|\nabla w|^{p-2}\nabla w, \nabla z \rangle\diff x
	 = c_{0} \int_{B_{r_{1}}(x_{0})}f z \diff x.
	\end{align*}
	Applying H\"older's inequality to the right-hand side of the above formula and using (\ref{4.23}), we obtain
	\begin{align*}
	\int_{B_{r_{1}}(x_{0})}|\nabla z|^{p}\diff x \leq c_{0} \|f\|_{L^{q_{0}}(B_{r_{1}}(x_{0}))}\|z\|_{L^{q^{\prime}_{0}}(B_{r_{1}}(x_{0}))} &\leq c_{0}|B_{r_{1}}(x_{0})|^{\frac{1}{q_{0}}-\frac{1}{q}}\|f\|_{L^{q}(B_{r_{1}}(x_{0}))}\|z\|_{L^{q^{\prime}_{0}}(B_{r_{1}}(x_{0}))}\\
	& \leq C r^{N(\frac{1}{q_{0}}-\frac{1}{q})+\alpha}_{1} \biggl(\int_{B_{r_{1}}(x_{0})}|\nabla z|^{p}\diff x\biggr)^{\frac{1}{p}}
	\end{align*}
	for some $C=C(N, p, q_{0},q,  \|f\|_{q})>0$. Therefore,
	\begin{align*}
	\int_{B_{r_{1}}(x_{0})}|\nabla z|^{p}\diff x & \leq C^{p^{\prime}}r^{Np^{\prime}(\frac{1}{q_{0}}-\frac{1}{q})+p^{\prime}\alpha}_{1}=C^{p^{\prime}}r^{N+p^{\prime}-\frac{N p^{\prime}}{q}}_{1}.
	\end{align*}
	This completes the proof of Lemma~\ref{lem 4.4}.
\end{proof}

Using together Lemma~\ref{lem 4.3} and Lemma~\ref{lem 4.4}, we obtain the following estimate for the solution $u_{\Sigma}$ to the Dirichlet problem $-\Delta_{p}u=f\,\ \text{in}\,\ \Omega\backslash \Sigma,\,\ u \in W^{1,p}_{0}(\Omega\backslash \Sigma)$. Notice that in the following statement the definition of $\gamma(p,q)$ also depends on $N$, but we decided not to mention it explicitly to simplify the notation.

\begin{lemma} \label{lem 4.5} Let $p\in (N-1, +\infty)$ and $f \in L^{q}(\Omega)$ with $q> q_{0}$, where $q_{0}$ is defined in (\ref{Eq 1.1}). Then there exist $a\in (0,1/2)$, $\varepsilon_0 \in (0,a)$  and $C=C(N,p,q_0,q,\|f\|_q, |\Omega|)>0$ such that the following holds.  Assume that $\s\subset \overline{\Omega}$ is a closed arcwise connected set, $0<2 r_{0}\leq r_{1}\leq 1$, $B_{r_{1}}(x_{0})\subset \Omega$,
	\begin{equation*}
	\s \cap B_{r_{0}}(x_{0})\neq \emptyset\,\ \text{and} \,\ \s \backslash B_{r_{1}}(x_{0}) \neq \emptyset.
	\end{equation*}
	In addition, suppose that there exists an affine line $L\subset \mathbb{R}^{N}$ passing through $x_{0}$ such that
	\begin{equation}\label{4.24}
	d_{H}(\s\cap \overline{B}_{r_{1}}(x_{0}), L\cap \overline{B}_{r_{1}}(x_{0})) \leq \varepsilon_{0}r_{1}.
	\end{equation}
	Then
	\begin{equation}\label{4.25}
	\frac{1}{ar_{1}}\int_{B_{ar_{1}}(x_{0})} |\nabla u_{\s}|^p \diff x \leq \frac{1}{2} \left( \frac{1}{r_{1}}\int_{B_{r_{1}}(x_{0})} |\nabla u_{\s}|^p \diff x \right)+ C r_{1}^{\gamma(p,q)},
	\end{equation}
	where 
	\begin{equation} 
	\gamma(p,q)=	N-1+p^{\prime}-\frac{Np^{\prime}}{q} \,\ \text{ if }  \; 2\leq p<+\infty, \qquad
	\gamma(p,q)=3p-3-\frac{2p}{q} \,\ \text{ if } \; 1<p <2.
\label{defgamma}
	\end{equation}
\end{lemma}
\begin{proof} Let $w \in W^{1,p}(B_{r_{1}}(x_{0}))$ be a unique solution to the Dirichlet problem
	\begin{equation*}
	\begin{cases}
	-\Delta_{p} u &=\,\  0  \,\ \text{in}\,\ B_{r_{1}}(x_{0})\backslash \s \\
	\qquad u & =\,\ u_{\s} \ \text{on}\,\ (\s\cap B_{r_{1}}(x_{0}))\cup \partial B_{r_{1}}(x_{0}),
	\end{cases}
	\end{equation*}
	which means that $w - u_{\s} \in W^{1,p}_{0}(B_{r_{1}}(x_{0})\backslash \s)$ and
	\begin{equation}\label{4.27}
	\int_{B_{r_{1}}(x_{0})}|\nabla w|^{p-2}\nabla w \nabla \varphi \diff x = 0 \,\ \text{for all}\,\ \varphi \in W^{1,p}_{0}(B_{r_{1}}(x_{0})\backslash \s). 
	\end{equation}
	Let $I(\cdot)$ be as in Lemma \ref{lem 4.4}. Using  (\ref{2.8}) and H\"older's inequality, it is easy to see that
	\begin{equation} \label{4.28}
	I(u_{\Sigma})\leq C_{1}
	\end{equation}
	for some $C_{1}=C_{1}(N, p,  q_{0},q, \|f\|_{q}, |\Omega|)>0$. Then, applying Lemma \ref{lem 4.4} and using (\ref{4.28}), we get
	\begin{equation}\label{4.29}
	\int_{B_{r_{1}}(x_{0})}|\nabla u_{\s}-\nabla w|^{p}\diff x \leq Cr_1^{1+\gamma(p,q)}, 
	\end{equation}
	where $C=C(N,p, q_{0},q,  \|f\|_{q}, |\Omega|)>0$ and $\gamma(p,q)$ is defined in \eqref{defgamma}. 
	Now let $\alpha, \delta \in (0,1)$ and $\widetilde{C}>1$, depending only on $N$ and $p$, be as in Lemma~\ref{lem 4.1}, where $\widetilde{C}$ is such that the estimate (\ref{4.2}) holds with $C$ replaced by $\widetilde{C}$. Define $a=\min\bigl\{\delta, (2^{-p}\widetilde{C}^{-1-\alpha})^{\frac{1}{\alpha}}\bigr\}$. For each $N\geq 2$ and $p\in (N-1,+\infty)$, the constant $a$ is fixed. Applying Lemma~\ref{lem 4.3} with $r=r_{1}$ and $\varrho=a$, we obtain some $\varepsilon_{0} \in (0,a)$ such that under the condition (\ref{4.24}),
	\begin{align}\label{4.30}
	\frac{1}{a}\int_{B_{ar_1}(x_{0})} |\nabla w|^{p}\diff x \leq \widetilde{C}^{1+\alpha}a^{\alpha} \int_{B_{r_1}(x_{0})} |\nabla w|^{p}\diff x \leq 2^{-p}\int_{B_{r_1}(x_{0})} |\nabla w|^{p}\diff x.
	\end{align} 
	Hereinafter in this proof, $C$ denotes a positive constant that can only depend on $N,\, p,\,q_{0},\,q$, $\|f\|_{q},\, |\Omega|$ and can be different from line to line. Since for any nonnegative numbers $c$ and $d$, $(c+d)^{p}\leq 2^{p-1}(c^{p}+d^{p})$, we have

	\begin{eqnarray}
	\frac{1}{a}\int_{B_{ar_1}(x_0)} |\nabla u_{\s}|^p \diff x &\leq & \frac{2^{p-1}}{a} \int_{B_{ar_1}(x_0)} |\nabla w|^p \diff x + \frac{2^{p-1}}{a} \int_{B_{ar_1}(x_0)} |\nabla u_{\s}-\nabla w|^p \diff x \notag \\
	&\leq & \frac{1}{2}  \int_{B_{r_1}(x_0)} |\nabla w|^p \diff x + \frac{2^{p-1}}{a} \int_{B_{r_1}(x_0)} |\nabla u_{\s}-\nabla w|^p \diff x \notag \\
	&\leq &  \frac{1}{2}  \int_{B_{r_1}(x_0)} |\nabla w|^p \diff x + Cr_1^{1+\gamma(p,q)} \notag \\
	&\leq & \frac{1}{2}  \int_{B_{r_1}(x_0)} |\nabla u_{\s}|^p \diff x + Cr_1^{1+\gamma(p,q)}, \notag
	\end{eqnarray}
	where we have used (\ref{4.30}), (\ref{4.29}), and to obtain the last estimate, Theorem~\ref{thm 2.2}. The proof of Lemma~\ref{lem 4.5} follows by dividing the resulting inequality by $r_{1}$.
 \end{proof}
Finally, by iterating Lemma~\ref{lem 4.5} in a sequence of balls $\{B_{a^{l}r_{1}}(x_{0})\}_{l}$, we obtain the desired  decay behavior of the $p$-energy $r\mapsto \int_{B_{r}(x_{0})}|\nabla u_{\Sigma}|^{p}\diff x$ under flatness control on $\Sigma$ at $x_{0}$. 

\begin{lemma}\label{lem 4.6}
	 Let $p\in (N-1, +\infty)$ and $f \in L^{q}(\Omega)$ with  $q> q_{1}$, where $q_{1}$ is defined in (\ref{1.4}). Then there exist   $\varepsilon_{0}, b, \overline{r} \in (0,1)$ and  $C=C(N,p,q_{0},q,\|f\|_q, |\Omega|)>0$ such that the following holds. Assume that $\s\subset \overline{\Omega}$ is a closed arcwise connected set, $0<2r_{0}\leq r_{1}\leq\overline{r}$,  $B_{r_{1}}(x_{0}) \subset \Omega$ and that for each  $r \in [r_{0}, r_{1}]$ there exists an affine line $L=L(r)$  passing through $x_{0}$ such that $d_{H}(\s\cap \overline{B}_{r}(x_{0}), L\cap \overline{B}_{r}(x_{0})) \leq \varepsilon_{0} r$. Assume also that $\Sigma \setminus B_{r_1}(x_0)\not = \emptyset$. Then for all $r\in[r_{0}, r_{1}]$,
	\begin{equation}\label{4.31}
	\int_{B_{r}(x_{0})}|\nabla u_\Sigma|^{p}\diff x \leq C \Bigl(\frac{r}{r_{1}}\Big)^{1+b}  \int_{B_{r_{1}}(x_{0})}|\nabla u_\Sigma|^{p}\diff x + Cr^{1+b}.
	\end{equation}
\begin{proof}
	Let  $a\in (0,1/2)$, $\varepsilon_0 \in (0,a)$  and $C=C(N,p,q_0,q,\|f\|_q, |\Omega|)>0$ be the constants given by Lemma~\ref{lem 4.5}. The definition of $q_{1}$ and the assumption $q>q_{1}$ have been made in order to guarantee that $\gamma(p,q)>0$, where $\gamma(p,q)$ is defined in (\ref{defgamma}).
	Let us now define 
	\[b= \min\biggl\{\frac{\gamma(p,q)}{2},  \frac{\ln(3/4)}{\ln(a)} \biggr\}  \;,\quad \bar r =  \left(\frac{1}{4}\right)^{\frac{1}{b}}.\]
	Notice that for all $t \in (0, \bar r]$, 
	\begin{eqnarray}
	\frac{1}{2}t^{b} + t^{\gamma(p,q)} \leq (at)^{b}.  \label{4.32}
	\end{eqnarray}
	Indeed, since $0<2 b \leq \gamma(p,q)$, $b\leq \ln(3/4)/\ln(a)$ and $a, \overline{r}\in (0,1)$, $t^{\gamma(p,q)}\leq t^{2b}\leq \overline{r}^{b}t^{b}$ and $3/4\leq a^{b}$, so
	\[\frac{1}{2}t^{b} + t^{\gamma(p,q)} \leq \frac{1}{2}t^{b} + \overline{r}^{b} t^{b} \leq \frac{3}{4}t^{b} \leq (at)^{b}.\]
	 Under the assumptions of Lemma~\ref{lem 4.6}, we can apply Lemma~\ref{lem 4.5} in all the balls $B_{a^{l}r_{1}}(x_{0})$, $l\in \{0,...,k\}$, where $k\in \mathbb{N}$ is such that $a^{k+1}r_1<r_0\leq a^{k}r_{1}$. 
	Next, we define $\Psi(r)=\frac{1}{r}\int_{B_r(x_0)}|\nabla u_{\s}|^p \diff x$, $r \in (0,r_{1}]$
	and  prove by induction  that for each $l\in \{0,...,k\}$, 
	\begin{eqnarray}
	\Psi(a^{l} r_1) &\leq& \frac{1}{2^{l}} \Psi(r_1) + C(a^{l} r_{1})^{b}. \label{induction00}
	\end{eqnarray}
	It is clear that (\ref{induction00}) holds for $l= 0$. Assume that (\ref{induction00}) holds for some $l\in \{0,...,k-1\}$. Then, applying   Lemma~\ref{lem 4.5} and using the induction hypothesis, we get
	$$\Psi(a^{l+1}r_1)\leq \frac{1}{2} \Psi(a^l r_1) + C (a^l r_1)^{\gamma(p,q)} \leq \frac{1}{2}\biggl(\frac{1}{2^{l}}\Psi( r_1)  + C(a^{l}r_{1})^{b}\biggr)+ C(a^l r_1)^{\gamma(p,q)}.$$
	Thanks to (\ref{4.32}), we finally conclude that 
	$$\Psi\bigl(a^{l+1}r_1\bigr)\leq \frac{1}{2^{l+1}} \Psi( r_1)  +  C\bigl(a^{l+1} r_1\bigr)^{b}.$$
	Thus \eqref{induction00} is proved. Now let $r\in [r_0,r_1]$ and $l\in \{0,...,k\}$ be such that $a^{l+1}r_1< r\leq a^lr_1$. Then 
	\begin{align*}
	\Psi(r) \leq   \frac{1}{a}\Psi(a^l r_1) \leq   \frac{1}{a}\frac{1}{2^l}\Psi(r_1) + \frac{C}{a}(a^{l}r_1)^{b} \notag &\leq  \frac{2}{a} (a^{l+1})^{b} \Psi(r_1)  + C' (a^{l+1}r_{1})^{b} \\
	&\leq  C'' \left(\frac{r}{r_1}\right)^{b} \Psi(r_1) + C'' r^{b},
	\end{align*}
	where $C''=C''(a,N,p,q_{0},q,\|f\|_{q},|\Omega|)>0$. Since $a$ is fixed for each $N\geq 2$ and $p\in (N-1,+\infty)$, we can assume that $C''$ depends only on $N, p, q_{0},q, \|f\|_{q}$ and $|\Omega|$.  This completes the proof of Lemma~\ref{lem 4.6}.
\end{proof}

\section{Absence of loops}

\begin{theorem} \label{thm absence of loops} Let $\Omega \subset \mathbb{R}^{N}$ be open and bounded, $p \in (N-1, +\infty)$ and $f \in L^{q}(\Omega)$  with $q>q_1$, where $q_{1}$ is defined in (\ref{1.4}). Then every solution $\s$ to Problem~\ref{P 1.1} cannot contain closed loops (i.e., homeomorphic images of a circumference $S^{1}$).
\end{theorem}

\textup{The next lemma will be used in the proof of Theorem~\ref{thm absence of loops}.}
\begin{lemma} \label{lem 5.2} Let $\s$ be a closed connected set in $\mathbb{R}^{N}$ with $\mathcal{H}^{1}(\s)<+\infty$. Then the following assertions hold. 
	\begin{itemize}
		\item  If $\Sigma$ contains a simple closed curve $\Gamma$, then $\mathcal{H}^{1}$-a.e. point $x\in \Gamma$ is a ``noncut" point, namely, there exists a sequence of relatively open sets $D_{n}\subset \s$ satisfying
		\begin{enumerate}[label=(\roman*)]
			\item $x \in D_{n}$ for all sufficiently large $n$;
			\item $\s \backslash D_{n}$ are connected for all $n$;
			\item $\diam D_{n} \searrow 0$ as $n\to +\infty$;
			\item $D_{n}$ are connected for all $n$.
		\end{enumerate}
		\item ``flatness" : for $\mathcal{H}^{1}$-a.e. point $x\in \Sigma $ there exists the ``tangent" line $T_{x}$ to $\s$ at x in the sense that $x \in T_{x}$ and 
		\[\frac{1}{r}d_{H}(\s \cap \overline{B}_{r}(x), T_{x}\cap \overline{B}_{r}(x)) \underset{r\to 0+}{\to 0}. \]
	\end{itemize}
\end{lemma}

\begin{proof} By \cite[Lemma 5.6]{Paolini-Stepanov}, $\mathcal{H}^{1}$-a.e. point $x\in \Gamma$ is a noncut point for $\s$ (i.e., a point such that $\s\backslash \{x\}$ is connected). Then, by \cite[Lemma 5.3]{Stepanov-Paolini}, it follows that for each noncut point there are connected neighborhoods $D_{n}$ that can be cut leaving the set connected and $\diam(D_{n})\to 0$, so $(i)$-$(iv)$ are satisfied for a suitable sequence $D_{n}$. Let us now prove the second assertion of Lemma~\ref{lem 5.2}. First, notice that, there is a Lipschitz surjective mapping $g:[0,L]\to \Sigma$, where $L=\mathcal{H}^{1}(\Sigma)$ (see, for instance, \cite[Proposition~30.1]{David}). Furthermore, in \cite[Proposition 3.4]{Miranda-Paolini-Stepanov}, it was proved that $\mathcal{H}^{1}(\Sigma \backslash \Sigma_{0})=0$, where
	\begin{align*}
	\Sigma_{0}=\{x \in \Sigma: t\in (0,L),\,\ g^{\prime}(t)\,\ \text{exists},\,\ &|g^{\prime}(t)|=1\,\ \text{whenever}\,\ g(t)=x,\,\ g^{-1}(x) \,\ \text{is finite}\\ & \qquad \qquad  \text{and if}\,\ g(t)=g(s)=x,\,\ \text{then} \,\ g^{\prime}(t)=\pm g^{\prime}(s)\},
	\end{align*}
	and that for all $x \in \Sigma_{0}$,
\begin{equation}\label{5.1}
\frac{1}{r}\max_{y \in \Sigma \cap \overline{B}_{r}(x)} \dist (y, T_{x}\cap \overline{B}_{r}(x)) \underset{r \to 0+}{\to 0},
\end{equation}
where $T_{x}=x + Span(g^{\prime}(t))$,\, $x = g(t)$. In order to prove that 
\begin{equation} \label{5.2}
\frac{1}{r}\max_{y \in T_{x}\cap \overline{B}_{r}(x)} \dist(y, \Sigma \cap \overline{B}_{r}(x)) \underset{r\to 0+}{\to 0},
\end{equation}
we shall follow the same approach as in \cite[Proposition 2.2]{Approx}. Observe that for each $x \in \Sigma_{0}$ there exists a mapping $h\mapsto \xi(h)$ such that $\xi(h)\to 0$ when $h \to 0$ and $g(t+h)=g(t)+hg^{\prime}(t)+h\xi(h)$ when $|h|>0$ is small enough, where $g(t)=x$. Next, let $\delta \in (0,1)$ be given. We can choose a sufficiently small $r_{0}>0$ such that $|\xi(h)|<\delta/2$ for all $h \in (-r_{0}, r_{0})\backslash\{0\}$. Then for each $r \in (0, r_{0})$ and each $z \in T_{x}\cap \overline{B}_{(1-\delta/2)r}(x)$, there exists $\lambda \in [(\delta/2-1)r, (1-\delta/2 )r]$ such that $z=g(t)+\lambda g^{\prime}(t)$. So, defining $y= g(t+\lambda)$ and observing that $g(t+\lambda)=g(t)+\lambda g^{\prime}(t)+\lambda\xi(\lambda)$, we deduce that $y \in \Sigma \cap B_{r}(x)$ and $|z-y|<\delta r/2$. This implies that $\max_{z \in T_{x}\cap \overline{B}_{r}(x)} \dist(z, \Sigma \cap \overline{B}_{r}(x))<\delta r$ for all $r \in (0, r_{0})$ and, therefore, proves (\ref{5.2}). Observing that (\ref{5.1}) and (\ref{5.2}) together prove the second assertion of Lemma~\ref{lem 5.2}, we complete the proof.
\end{proof}

\begin{proof}[Proof of Theorem \ref{thm absence of loops}] For the sake of contradiction, assume that for some $\lambda>0$ a minimizer $\s$ of $\mathcal{F}_{\lambda,f, \Omega}$ over $\mathcal{K}(\Omega)$ contains a simple closed curve $\Gamma\subset \s$. Notice that there is no a relatively open subset in $\Sigma$ contained in both $\Gamma$ and $\partial \Omega$, because otherwise, according to Lemma~\ref{lem 5.2}, there would be a relatively open subset $D\subset \Sigma$ such that $D \subset \partial \Omega$ and $\Sigma\backslash D$ would remain connected, but, observing that in this case $u_{\Sigma\backslash D}=u_{\Sigma}$ and $\mathcal{H}^{1}(D)>0$, we would obtain a contradiction with the optimality of $\Sigma$. Thus, by Lemma~\ref{lem 5.2}, there is a point $x_{0} \in \Gamma \cap \Omega$ which is a noncut point for $\Sigma$ and such that $\Sigma$ is flat at $x_{0}$.
	Therefore for $x_{0}$ there exist the sets $D_{n} \subset \s$ and the tangent line $T_{x_{0}}$ to $\Sigma$ at $x_{0}$ as in Lemma~\ref{lem 5.2}. Let $\varepsilon_{0},b,\overline{r},C$   be the constants of Lemma~\ref{lem 4.6} and let $B_{t_{0}}(x_{0})\subset \Omega$ with $t_{0}<\overline{r}$.  We define $r_{n}:=\diam D_{n}$ so that $D_{n}\subset \s\cap \overline{B}_{r_{n}}(x_{0})$. The flatness of $\s$ at $x_{0}$ implies that for any given $\varepsilon>0$ there is $\delta \in (0, t_{0}]$ such that 
	\begin{equation*}
	d_{H}(\s\cap \overline{B}_{r}(x_{0}), T_{x_{0}}\cap \overline{B}_{r}(x_{0}))\leq \varepsilon r \,\ \text{for all}\,\ r\in (0,\delta].
	\end{equation*}
	For each $n \in \mathbb{N}$, we define $\s_{n}:=\s\backslash D_{n}$, which, by Lemma~\ref{lem 5.2}, remains closed and connected. 
	We fix  $\varepsilon = \varepsilon_{0}/2$ and $r\in (0,\delta]$. Next, we want to  apply  Lemma~\ref{lem 4.6} to $\Sigma_n$, but we have to control the Hausdorff distance between $\Sigma_{n}\cap \overline{B}_{r}(x_{0})$ and a diameter of $\overline{B}_{r}(x_{0})$. We already know that $\Sigma$ is $\varepsilon  r$-close, in $\overline{B}_{r}(x_{0})$ and in the Hausdorff distance, to $T_{x_{0}}\cap \overline{B}_{r}(x_{0})$ for all $r\in (0,\delta]$.  Furthermore, if 
	$r_n \leq \varepsilon_{0} r/2$, then
	\begin{align*}
	d_{H}(\Sigma_{n}\cap \overline{B}_{r}(x_{0}), T_{x_{0}}\cap \overline{B}_{r}(x_{0}))&\leq d_{H}(\Sigma_{n}\cap \overline{B}_{r}(x_{0}), \s\cap \overline{B}_{r}(x_{0}))+ d_{H}(\Sigma\cap \overline{B}_{r}(x_{0}), T_{x_{0}}\cap \overline{B}_{r}(x_{0})) \\
	&\leq r_{n} +  \frac{\varepsilon_{0}r}{2} \leq \frac{\varepsilon_{0}r}{2} +\frac{\varepsilon_{0}r}{2} =  \varepsilon_{0} r.
	\end{align*}
	Thus, if $2 r_{n}/\varepsilon_{0}<\delta/2$, we can apply Lemma~\ref{lem 4.6} to $\Sigma_n$ for the interval $[2 r_{n}/\varepsilon_{0}, \delta]$, which says that 
	\begin{equation*} 
	\int_{B_{r}(x_{0})}|\nabla u_{\Sigma_n}|^{p}\diff x \leq C \Bigl(\frac{r}{\delta}\Bigr)^{1+b}  \int_{B_{\delta}(x_{0})}|\nabla u_{\Sigma_n}|^{p}\diff x + Cr^{1+b}\,\ \text{for all}\,\  r\in\biggl[\frac{2 r_{n}}{\varepsilon_0}, \delta\biggr],
	\end{equation*}
	where $C=C(N,p,q_{0},q,\|f\|_{q},|\Omega|)>0$. Hereinafter in this proof, $C$ denotes a positive constant that does not depend on $r_{n}$ and can be different
	from line to line. Next, using the above estimate for $r=2 r_{n}/\varepsilon_{0}$ and using also (\ref{2.8}), we get
	\begin{equation}
	\int_{B_{\frac{2 r_{n}}{\varepsilon_0}}(x_{0})}|\nabla u_{\Sigma_n}|^{p}\diff x \leq Cr_n^{1+b} \notag
	\end{equation}
	for each $n\in \mathbb{N}$ such that $2 r_{n}/\varepsilon_{0}<\delta/2$. Recall that the exponent $b$ given by Lemma~\ref{lem 4.6} is positive provided $q>q_{1}$, which is one of our assumptions. Now, since $\Sigma$ is a minimizer of Problem~\ref{P 1.1} and $\s_{n}$ is a competitor for $\s$, we get the following
	\begin{align*} 
	0&\leq \mathcal{F}_{\lambda, f, \Omega}({\s}_{n})- \mathcal{F}_{\lambda, f, \Omega}(\s)  \leq E_{f,\Omega}(u_{\s})-E_{f,\Omega}(u_{\s_{n}}) - \lambda r_{n}\\
	& \leq C \int_{B_{2r_{n}}(x_{0})}|\nabla u_{\s_{n}}|^{p}\diff x + Cr^{N+p^{\prime}-\frac{Np'}{q}}_{n} -\lambda r_{n} \,\ (\text{by Corollary \ref{cor 2.15}})\\
	&\leq C \int_{B_{\frac{2r_{n}}{\varepsilon_{0}}}(x_{0})}|\nabla u_{\Sigma_{n}}|^{p}\diff x  + Cr^{N+p^{\prime}-\frac{Np'}{q}}_{n} -\lambda r_{n}  \\
	&\leq  Cr_n^{1+b} +Cr^{N+p^{\prime}-\frac{N p'}{q}}_{n} -\lambda r_{n}. 
	\end{align*}
	Notice that $N+p^{\prime}-Np^{\prime}/q>1$ if and only if $q>Np/(Np-N+1)$,
	which is always true under the assumption $q>q_1$. Therefore, letting $n$ tend to $+\infty$, we arrive to a contradiction. This completes the proof of Theorem~\ref{thm absence of loops}.
\end{proof}
	
\end{lemma}

\section{Proof of partial regularity}
In this section, we prove that every solution $\Sigma$ to Problem~\ref{P 1.1} is  locally $C^{1,\alpha}$ regular at $\mathcal{H}^{1}$-a.e. point $x \in \Sigma \cap \Omega$.

Throughout this section, $\Omega$ denotes an open bounded subset in $\mathbb{R}^{N}$. Recall that $\mathcal{K}(\Omega)$ is the class of all closed connected proper subsets of $\overline{\Omega}$.

The factor $\lambda$ in the statement of Problem~\ref{P 1.1} affects the shape of an optimal set minimizing the functional $\mathcal{F}_{\lambda,f,\Omega}$ over $\mathcal{K}(\Omega)$, and, according to Proposition~\ref{prop 2.19}, we know that  there exists a number $\lambda_{0}=\lambda_{0}(N, p, f, \Omega)>0$  such that if $\lambda \in (0, \lambda_{0}]$, then each minimizer $\s$ of the functional $\mathcal{F}_{\lambda, f,\Omega}$ over $\mathcal{K}(\Omega)$ has positive $\mathcal{H}^{1}$-measure. Throughout this section, we assume that $\lambda=\lambda_{0}=1$ for simplicity. This is not restrictive regarding to the regularity theory.

\subsection{Control on defect of minimality}
We begin with the definition of the flatness.
\begin{defn} \label{def 6.1} \textit{For each closed set $\Sigma \subset \mathbb{R}^{N}$, each point $x \in \mathbb{R}^{N}$ and radius $r>0$, we define the flatness of $\s$ in $\overline{B}_{r}(x)$ as follows}
\begin{equation*}
\beta_{\s}(x,r)=\inf_{L\ni x} \frac{1}{r}d_{H}(\s\cap \overline{B}_{r}(x), L\cap \overline{B}_{r}(x)),
\end{equation*}
\textit{where the infimum is taken over the set of all affine lines (1-dimensional planes) $L$ passing through $x$.} 
\end{defn}
Notice that if $\beta_{\s}(x,r)<+\infty$, then it is easy to prove that  the infimum above is actually the minimum, and in this case $\beta_{\Sigma}(x,r) \in [0, \sqrt{2}]$ and $\beta_{\Sigma}(x,r)=\sqrt{2}$ if and only if $\Sigma\cap \overline{B}_{r}(x)$ is a point in $\partial B_{r}(x)$. Furthermore, it is worth noting that if $\kappa \in (0,1)$ and $\beta_{\Sigma}(x, \kappa r)<+\infty$, then the following inequality holds 
	\begin{equation}
	\beta_{\s}(x, \kappa r)\leq \frac{2}{\kappa}\beta_{\s}(x,r) \label{6.1}
	\end{equation}
(for a proof of the inequality (\ref{6.1}), we refer the reader to the proof of \cite[Proposition 6.1]{p-compl}, which actually applies for the general spatial dimension $N\geq 2$).

Now we introduce the following notions of the local energy and the density, which will play a crucial role in the proof of partial regularity.

\begin{defn}\label{def 6.2} \textit{Let $\s\in \mathcal{K}(\Omega)$  and  $\tau \in [0, \sqrt{2}]$. For each $x_{0} \in \overline{\Omega}$ and $r>0$, we define}
	\begin{equation}\label{6.2}
	w^{\tau}_{\s}(x_{0},r)=\sup_{\substack{\s^{\prime} \in \mathcal{K}(\Omega),\, \s^{\prime}\Delta \s \subset \overline{B}_{r}(x_{0}) \\ \mathcal{H}^{1}(\Sigma^{\prime})\leq 100 \mathcal{H}^{1}(\Sigma),\, \beta_{\s^{\prime}}(x_{0},r)\leq \tau}} \frac{1}{r} \int_{B_{r}(x_{0})}|\nabla u_{\s^{\prime}}|^{p}\diff x. 
	\end{equation}
\end{defn}
The condition $\mathcal{H}^{1}(\Sigma^{\prime})\leq 100 \mathcal{H}^{1}(\Sigma)$, together with the facts that $\mathcal{H}^{1}(\Sigma)<+\infty$, $\Sigma^{\prime}\in \mathcal{K}(\Omega)$ in the definition of $w^{\tau}_{\Sigma}$ above, guarantees that $\Sigma^{\prime}$ is arcwise connected (see Remark \ref{arcwise connected}).
\begin{defn}\label{def 6.3} \textit{Let $\Sigma \subset \mathbb{R}^{N}$ be $\mathcal{H}^{1}$-mesurable. For each $x_{0} \in \Sigma$ and $r>0$, we define}
	\begin{equation*}
	\theta_{\Sigma}(x_{0},r)=\frac{1}{r}\mathcal{H}^{1}(\Sigma \cap B_{r}(x_{0})).
    \end{equation*}
    \end{defn}

\begin{rem} \label{rem 6.4} Assume that $\Sigma\in\mathcal{K}(\Omega)$,  $\tau\in [0,\sqrt{2}]$, $x_{0}\in \overline{\Omega}$ and $\beta_{\s}(x_{0}, r)\leq\tau$. Then there exists a solution to problem (\ref{6.2}). Indeed, $\Sigma$ is a competitor in the definition of $w^{\tau}_{\Sigma}(x_{0},r)$. Thus, according to Proposition~\ref{prop 2.16}, $w^{\tau}_{\Sigma}(x_{0},r) \in [0,+\infty)$. We can then conclude using  the direct method in the  Calculus of Variations, standard compactness results and the fact that $\mathcal{H}^{1}$-measure is lower semicontinuous with respect to the topology generated by the Hausdorff distance.
\end{rem}

We shall use the following proposition in order to establish a decay behavior for $w_{\Sigma}^{\tau}(x_{0},r)$ whenever $\Sigma$ is flat enough in all balls $\overline{B}_{r}(x_{0})$ with $r \in [r_{0},r_{1}]$.

\begin{prop} \label{prop 6.5} Let $\Sigma \subset \overline{\Omega}$ be closed and arcwise connected, $x \in \overline{\Om}$, $\tau \in [0,1/10]$ and let $\beta_{\Sigma}(x,r_{1})\leq \varepsilon$ for some $\varepsilon \in [0, \tau]$. In addition, assume that $0<r_{0}<r_{1}$, $\beta_{\Sigma}(x,r)\leq \tau$ for all $r \in [r_{0},r_{1}]$ and $\Sigma \backslash \overline{B}_{r_{1}}(x)\neq \emptyset$. If $r \in [r_{0}, r_{1}]$, then for any closed arcwise connected set $\Sigma^{\prime}\subset \overline{\Omega}$ such that $\s^{\prime} \Delta \s \subset \overline{B}_{r}(x)$ and $\beta_{\s^{\prime}}(x, r) \leq \tau$ we have that
	\begin{enumerate}[label=(\roman*)]
		\item 
		\begin{equation}\label{6.3}
		\beta_{\s^{\prime}}(x, r_{1}) \leq \frac{5\tau r}{r_{1}} + \varepsilon, 
		\end{equation}
		\item 
		\begin{equation}\label{6.4}
		\beta_{\s^{\prime}}(x, s) \leq 6\tau\,\ \text{for all}\,\ s \in [r, r_{1}]. 
		\end{equation}
	\end{enumerate}
\end{prop}
\begin{proof} 
	Every ball in this proof is centered at $x$.  Let $L_{1}, \, L$ and $L^{\prime}$ realize the infimum, respectively, in the definitions of $\beta_{\Sigma}(x,r_{1}),\, \beta_{\Sigma}(x,r)$ and $\beta_{\Sigma^{\prime}}(x,r)$. Notice that
	\begin{equation} \label{6.5}
	d_{H}(\s \cap \overline{B}_{r}, L\cap \overline{B}_{r}) \leq \tau r.
	\end{equation}
	On the other hand,
	\begin{align*}
	d_{H}(\Sigma^{\prime} \cap \overline{B}_{r_{1}},L_{1}\cap \overline{B}_{r_{1}}) & \leq d_{H}(\Sigma^{\prime} \cap \overline{B}_{r_{1}}, \s \cap \overline{B}_{r_{1}}) + d_{H}(\s \cap \overline{B}_{r_{1}}, L_{1}\cap \overline{B}_{r_{1}}) \\
	 &\leq d_{H}(\Sigma^{\prime} \cap \overline{B}_{r}, \s \cap \overline{B}_{r}) + \varepsilon r_{1}, \numberthis \label{6.6}
	\end{align*}
	where the latter inequality comes because $\Sigma^{\prime} \Delta \s \subset \overline{B}_{r}$ and $\beta_{\s}(x,r_{1})\leq \varepsilon$. In addition, 
	\begin{align*}
	d_{H}(\Sigma^{\prime} \cap \overline{B}_{r}, \s \cap \overline{B}_{r}) & \leq d_{H}(\Sigma^{\prime} \cap \overline{B}_{r}, L^{\prime} \cap \overline{B}_{r}) + d_{H}(L \cap \overline{B}_{r}, L^{\prime} \cap \overline{B}_{r}) + d_{H}(\s \cap \overline{B}_{r}, L \cap \overline{B}_{r})\\
	&\leq 2\tau r + d_{H}(L \cap \overline{B}_{r}, L^{\prime} \cap \overline{B}_{r}),\numberthis \label{6.7}
	\end{align*}
	where we have used (\ref{6.5}) and the assumption $\beta_{\Sigma^{\prime}}(x,r) \leq \tau$. Notice that, since $\Sigma \cap B_{r}\neq \emptyset$, $\Sigma \backslash \overline{B}_{r_{1}} \neq \emptyset$ and $\Sigma$ is arcwise connected, there is a sequence $(x_{n})_{n}\subset \Sigma\backslash \overline{B}_{r}$ converging to some point $y \in \partial B_{r}$. We conclude that $y\in \Sigma^{\prime}\cap \Sigma \cap \partial B_{r}$ because $\Sigma^{\prime}\Delta \Sigma \subset\overline{B}_{r}$ and $\Sigma^{\prime}$, $\Sigma$ are closed. If $y \in L\cap L^{\prime}$, then $L=L^{\prime}$. Assume that $y \not \in L$. Let $\Pi$ be the 2-dimensional plane passing through $L$ and $y$, and let $\xi \in L\cap \partial B_{r}$ be such that $|y-\xi|= \dist (y,L\cap \partial B_{r})$. Denote by $\gamma$ the geodesic in the circle $\Pi \cap \partial B_{r}$ connecting $y$ with $\xi$. Then
	\[
	\mathcal{H}^{1}(\gamma)\leq \arcsin(\beta_{\Sigma}(x,r))r \leq \arcsin(\tau)r\leq \frac{3}{2}\tau r,
	\]
	where we have used the assumption $\beta_{\Sigma}(x,r)\leq \tau$ and the fact that $\arcsin(t)\leq 3t/2$ for all $t \in [0,1/10]$. Notice that if $y \in L^{\prime}$, then $d_{H}(L\cap \overline{B}_{r}, L^{\prime}\cap \overline{B}_{r})\leq \mathcal{H}^{1}(\gamma)$, otherwise let $\xi^{\prime} \in L^{\prime}\cap \partial B_{r}$ be such that $|y-\xi^{\prime}|=\dist(y, L^{\prime}\cap \partial B_{r})$ and let $\gamma^{\prime}$ be the geodesic in the circle $\Pi^{\prime}\cap \partial B_{r}$ connecting $y$ and $\xi^{\prime}$, where $\Pi^{\prime}$ is the 2-dimensional plane passing through $L^{\prime}$ and $y$. Then, using the assumption $\beta_{\Sigma^{\prime}}(x,r)\leq \tau$ and proceeding as before, we get
	\[
	\mathcal{H}^{1}(\gamma^{\prime})\leq \frac{3}{2}\tau r.
	\]
	Finally, we can conclude that 
	\[
	d_{H}(L\cap \overline{B}_{r}, L^{\prime}\cap \overline{B}_{r})\leq \mathcal{H}^{1}(\gamma)+\mathcal{H}^{1}(\gamma^{\prime})\leq 3\tau r.
	\]
	This, together with (\ref{6.7}), gives the following
	\begin{equation}\label{6.8}
	d_{H}(\Sigma^{\prime} \cap \overline{B}_{r},\Sigma \cap \overline{B}_{r})\leq 5\tau r.
	\end{equation}
	Using (\ref{6.6}) and (\ref{6.8}), we get
	\begin{align*}
	d_{H}(\Sigma^{\prime} \cap \overline{B}_{r_{1}}, L_{1} \cap \overline{B}_{r_{1}}) & \leq 5\tau r + \varepsilon r_{1}.
	\end{align*}
	Thus, we have proved $(i)$. Now let $s \in [r, r_{1}]$ and let $L_{s}$ be an affine line realizing the infimum in the definition of $\beta_{\s}(x,s)$. As in the proof of $(i)$, we get
	\begin{align*}
	d_{H}(\Sigma^{\prime} \cap \overline{B}_{s}, L_{s} \cap \overline{B}_{s}) & \leq d_{H}(\Sigma^{\prime} \cap \overline{B}_{s}, \s \cap \overline{B}_{s}) + d_{H}(\s \cap \overline{B}_{s}, L_{s} \cap \overline{B}_{s}) \\
	& \leq d_{H}(\Sigma^{\prime} \cap \overline{B}_{r}, \s \cap \overline{B}_{r}) + d_{H}(\s \cap \overline{B}_{s}, L_{s} \cap \overline{B}_{s}).
	\end{align*}
	This, together with (\ref{6.8}) and the fact that $\beta_{\Sigma}(x,s)\leq \tau$, implies
	\begin{align*}
	d_{H}(\Sigma^{\prime} \cap \overline{B}_{s}, L_{s} \cap \overline{B}_{s}) \leq 5\tau r + \tau s\leq 6 \tau s,
	\end{align*}
	concluding the proof of Proposition~\ref{prop 6.5}.
\end{proof}

Hereinafter in this section, $\tau$ is a fixed constant such that $\tau \in (0,\varepsilon_{0}/6]$, where $\varepsilon_{0}$ is the constant of Lemma~\ref{lem 4.6}. Notice that $\varepsilon_{0}$ is fairly small.

Now we establish a decay behavior for $w^{\tau}_{\Sigma}(x,\cdot)$, provided that $\beta_{\Sigma}(x,\cdot)$ is small enough.

\begin{prop} \label{prop 6.6} Let $ p\in (N-1,+\infty)$ and $f \in L^{q}(\Omega)$ with $q> q_1$, where $q_1$ is defined in (\ref{1.4}). Let $\varepsilon_{0},b,\overline{r}\in(0,1),\, C>0$ be the constants of Lemma~\ref{lem 4.6}. Assume that $\s \in \mathcal{K}(\Omega)$, $\mathcal{H}^{1}(\Sigma)<+\infty$, $0<r_{0}\leq r_{1}/10$ and $B_{r_{1}}(x_{0})\subset \Omega$ with $r_{1}\in (0, \min\{\overline{r},\diam(\Sigma)/2\})$. Assume also that
	\begin{equation*}
	\beta_{\s}(x_{0}, r) \leq \tau/2
	\end{equation*}
	for all $r \in [r_{0},r_{1}]$. Then, for all $r \in [r_{0}, r_{1}/10]$,
	\begin{equation}\label{6.9}
	w^{\tau}_{\s}(x_{0}, r) \leq C \left(\frac{r}{r_{1}}\right)^{b} w^{\tau}_{\s}(x_{0}, r_{1}) + Cr^{b}. 
	\end{equation}
	
\end{prop}

\begin{proof} According to Remark~\ref{arcwise connected}, $\Sigma$ is arcwise connected. From Remark~\ref{rem 6.4} it follows that there is $\Sigma_{r}\subset \overline{\Omega}$ realizing the supremum in the definition of $w^{\tau}_{\Sigma}(x_{0},r)$ which, by Remark~\ref{arcwise connected}, is arcwise connected. Furthermore, Proposition \ref{prop 6.5} says that
	\[
	\beta_{\Sigma_{r}}(x_{0}, r_{1}) \leq \tau \,\ \text{and} \,\ \beta_{\Sigma_{r}}(x_{0}, s) \leq 6\tau \leq  \varepsilon_{0} \,\ \text{for all}\,\ s\in [r, r_{1}].
	\]
	Thus, we can apply Lemma~\ref{lem 4.6} to $u_{\s_r}$, which yields
	
	\begin{align*}
	w^{\tau}_{\s}(x_{0}, r) =\frac{1}{r}\int_{B_{r}(x_{0})}|\nabla u_{\Sigma_{r}}|^{p}\diff x 
	\leq C \Bigl(\frac{r}{r_{1}}\Bigr)^{b}\frac{1}{r_{1}}\int_{B_{r_{1}(x_{0})}}|\nabla u_{\Sigma_{r}}|^{p}\diff x + Cr^{b}
	 \leq  C  \Bigl(\frac{r}{r_{1}}\Bigr)^{b} w^{\tau}_{\s}(x_{0},r_{1}) + Cr^{b}.
	\end{align*}
	Notice that  to obtain the last estimate we have used the definition of $w^{\tau}_{\s}(x_{0},r_{1})$ and the fact that  $\beta_{\Sigma_{r}}(x_{0}, r_{1}) \leq \tau$. 
\end{proof}

Now we are in position to control a defect of minimality via $w^{\tau}_{\Sigma}$.

\begin{prop} \label{prop 6.7} Let $p\in (N-1,+\infty)$ and $f \in L^{q}(\Omega)$ with $q>q_{1}$, where $q_{1}$ is defined in (\ref{1.4}), and let $\varepsilon_{0}, b, \overline{r}\in (0,1)$ be the constants of Lemma~\ref{lem 4.6}. Assume that $\s \in \mathcal{K}(\Omega)$, $\mathcal{H}^{1}(\Sigma)<+\infty$, $0<r_{0}\leq r_{1}/10$, $B_{r_{1}}(x_{0})\subset \Omega$ with $r_{1} \in (0, \min\{\overline{r},\diam(\s)/2\})$. Assume also that 
	\begin{equation*}
	\beta_{\s}(x_{0}, r) \leq \tau/2
	\end{equation*}
	for all $r \in [r_{0}, r_{1}]$. Then there exists a constant $C>0$, possibly depending only on $N,\,p,\, q_{0}, \, q,\, \|f\|_{q},\, |\Omega|$, such that if $r \in [r_{0}, r_{1}/10]$, then for any  $\Sigma^{\prime}\in \mathcal{K}(\Omega)$ satisfying $\Sigma^{\prime} \Delta \Sigma \subset \overline{B}_{r}(x_{0})$,\,  $\mathcal{H}^{1}(\Sigma^{\prime}) \leq 100 \mathcal{H}^{1}(\Sigma)$ and $\beta_{\Sigma^{\prime}}(x_{0},r) \leq \tau$, 
	\begin{equation}\label{6.10}
	E_{f,\Omega}(u_{\s})-E_{f,\Omega}(u_{\Sigma^{\prime}})\leq Cr \Bigl(\frac{r}{r_{1}}\Bigr)^{b} w^{\tau}_{\s}(x_{0},r_{1}) + Cr^{1+b}. 
	\end{equation}

\end{prop}
\begin{proof} According to Remark \ref{arcwise connected}, $\Sigma$ and $\Sigma^{\prime}$ are arcwise connected and by Corollary~\ref{cor 2.15},
	\begin{equation}\label{6.11}
	E_{f,\Omega}(u_{\s})-E_{f,\Omega}(u_{\s^{\prime}})\leq C\int_{B_{2r}(x_{0})}|\nabla u_{\s^{\prime}}|^{p}\diff x + Cr^{N+p^{\prime}-\frac{Np^{\prime}}{q}}, 
	\end{equation}
	where $C= C(N, p,q_{0}, q,  \|f\|_{q})>0$. On the other hand, by Proposition~\ref{prop 6.5}, 
	\[
	\beta_{\Sigma^{\prime}}(x_{0}, r_{1}) \leq \tau \,\ \text{and} \,\ \beta_{\Sigma^{\prime}}(x_{0}, s) \leq \varepsilon_{0} \,\ \text{for all}\,\ s\in [r, r_{1}].
	\]
	Thus, applying Lemma~\ref{lem 4.6} to $u_{\Sigma^{\prime}}$, we obtain that
	\begin{equation} \label{6.12}
	\int_{B_{2r}(x_{0})}|\nabla u_{\Sigma^{\prime}}|^{p}\diff x \leq C\Bigl(\frac{2r}{r_{1}}\Bigr)^{1+b}\int_{B_{r_{1}}(x_{0})}|\nabla u_{\Sigma^{\prime}}|^{p}\diff x + C(2r)^{1+b},
	\end{equation}
	where $C=C(N, p, q_{0},q,  \|f\|_{q}, |\Omega|)>0$. Hereinafter in this proof, $C$ denotes a positive constant that can only depend on $N,\, p,\, q_{0}$, $q,\, \|f\|_{q},\, |\Omega|$ and can be different from line to line. Using (\ref{6.11}), (\ref{6.12}) and the fact that $r^{N+p^{\prime}-\frac{Np^{\prime}}{q}} < r^{1+b}$ (because $r\in (0,1)$ and $0<b < N-1+p^{\prime}-N{p^{\prime}}/q$), we deduce the following chain of estimates
	\begin{align*}
	E_{f,\Omega}(u_{\s})-E_{f,\Omega}(u_{\s^{\prime}}) &\leq C\Bigl(\frac{r}{r_{1}}\Bigr)^{1+b}\int_{B_{r_{1}(x_{0})}}|\nabla u_{\Sigma^{\prime}}|^{p}\diff x + Cr^{1+b} \\
	& \leq Cr\Bigl(\frac{r}{r_{1}}\Bigr)^{b}\frac{1}{r_{1}}\int_{B_{r_{1}}(x_{0})}|\nabla u_{\s^{\prime}}|^{p}\diff x +Cr^{1+b} \\
	& \leq Cr\Bigl(\frac{r}{r_{1}}\Bigr)^{b}w^{\tau}_{\Sigma}(x_{0},r_{1}) + Cr^{1+b}, 
	\end{align*}
	where the last estimate is obtained using the  definition of $w^{\tau}_{\Sigma}(x_{0},r_{1})$ and the fact that $\beta_{\Sigma^{\prime}}(x_{0},r_{1})\leq \tau$. This completes the proof of Proposition~\ref{prop 6.7}.
\end{proof}

\subsection{Density control}

The following proposition says that there exists a constant $\kappa \in (0,1/100)$ such that if $\Sigma$ is a solution to Problem~\ref{P 1.1}, $\beta_{\Sigma}(x_{0},r)$, $w^{\tau}_{\Sigma}(x_{0},r)$ are fairly small provided that $B_{r}(x_{0})\subset \Omega$ with $x_{0}\in \Sigma$, and if $\theta_{\Sigma}(x_{0},r)$ is also small enough, then there exists $t \in [\kappa r, 2\kappa r]$ such that $\mathcal{H}^{0}(\Sigma \cap \partial B_{t}(x_{0}))=2$. This allows to construct a nice competitor for $\Sigma$ and derive the estimate (\ref{6.16}) leading to the regularity.
\begin{prop}\label{prop 6.8} Let $p \in (N-1,+\infty)$, $f \in L^{q}(\Omega)$ with $q>q_{1}$, where $q_{1}$ is defined in (\ref{1.4}). Then there exist $\delta, \varepsilon, \kappa\in (0,1/100)$ and $C=C(N,p,q_{0},q,\|f\|_{q},|\Omega|)>0,$ where $q_{0}$ is defined in (\ref{Eq 1.1}), such that the following holds. Assume that $\s$ is a solution to Problem \ref{P 1.1}, $x_{0} \in \Sigma$, $0<r <\min\{\delta,\diam(\s)/2 \}$, $B_{r}(x_{0}) \subset \Omega$ and
	\begin{equation}\label{6.13}
	\beta_{\s}(x_{0}, r) + w^{\tau}_{\Sigma}(x_{0},r)\leq \varepsilon.
	\end{equation}
	Assume also that 
    \begin{equation}\label{6.14}
    \theta_{\Sigma}(x_{0}, r)\leq 10\overline{\mu},
    \end{equation}
	where $\overline{\mu}$ is a unique positive solution to the equation $\mu=5+\mu^{1-\frac{1}{N}}$. Then the following assertions hold.
	\begin{enumerate}[label=(\roman*)]
		\item There exists $t \in [\kappa r, 2\kappa r]$ such that
		\begin{equation}\label{6.15}
		\mathcal{H}^{0}(\Sigma \cap \partial B_{t}(x_{0}))=2.
		\end{equation}
		\item Let $t \in [\kappa r, 2\kappa r]$ be such that $\mathcal{H}^{0}(\s \cap \partial B_{t}(x_{0}))=2$. Then 
		\begin{enumerate}[label=(ii-\arabic*)]
			\item the two points of $\s \cap \partial B_{t}(x_{0})$ belong to two different connected components of 
			\[
			\partial B_{t}(x_{0})\cap \{y: dist(y,L)\leq \beta_{\s}(x_{0},t)t\},
			\]
			where $L$ is an affine line realizing the infimum in the definition of $\beta_{\s}(x_{0},t).$
			\item $\s \cap \overline{B}_{t}(x_{0})$ is arcwise connected.
			\item If $\{z_{1}, z_{2}\}= \s \cap \partial  B_{t}(x_{0})$, then 
			\begin{equation} \label{6.16}
			\begin{split}
			\mathcal{H}^{1}(\s\cap B_{t}(x_{0}))\leq |z_{2} -z_{1}| +  Ct\Bigl(\frac{t}{r}\Bigr)^{b}w^{\tau}_{\Sigma}(x_{0}, r)+Ct^{1+b},
			\end{split}
			\end{equation} 
			where $b \in (0,1)$ is the constant given by Lemma~\ref{lem 4.6}.
		\end{enumerate}
	
	\end{enumerate}
\end{prop}

\begin{rem}\label{rem 6.9} If the situation of item $(ii$-1) occurs, we say that the two points lie ``on different sides''.
\end{rem}

\begin{proof} Let $\varepsilon_{0}, b, \overline{r} \in (0,1)$ be the constants of Lemma \ref{lem 4.6} and let $C=C(N,p,q_{0},q,\|f\|_{q},|\Omega|)>0$ be the constant of Proposition~\ref{prop 6.7}. We define \begin{equation}\label{6.17}
	\varepsilon=\frac{1}{\overline{\mu} C}\Bigl(\frac{\tau}{10}\Bigr)^{10}, \qquad k=\frac{\tau}{200}.
	\end{equation} 
	Fix $\delta \in (0, \overline{r})$ such that $\delta^{b}\leq \varepsilon$ and hence
	\begin{equation}\label{6.18}
	w^{\tau}_{\Sigma}(x_{0},r)+\delta^{b}\leq 2\varepsilon.
	\end{equation}
\textit{Step 1}. Let us first prove $(i)$. Thanks to (\ref{6.1}) and (\ref{6.13}), for all $s \in [\kappa r, r]$, it holds 
	\begin{equation}\label{6.19}
	\beta_{\Sigma}(x_{0}, s) \leq \frac{2}{\kappa}\beta_{\Sigma}(x_{0},r)\leq \frac{2\varepsilon}{\kappa}.
	\end{equation}
	On the other hand, for all $s \in [\kappa r, r]$,
	\begin{equation}\label{6.20}
	\theta_{\Sigma}(x_{0}, s)\leq \frac{r}{s} \theta_{\Sigma}(x_{0},r)\leq \frac{r}{\kappa r} \theta_{\Sigma}(x_{0},r)\leq \frac{10\overline{\mu}}{\kappa},
	\end{equation}
where the last estimate is due to (\ref{6.14}). Fix an arbitrary $s \in [\kappa r, 2\kappa r]$. By the coarea inequality (see, for instance, \cite[Theorem 2.1]{Paolini-Stepanov}),
\begin{equation}\label{6.21}
\mathcal{H}^{1}(\Sigma\cap B_{(1+\kappa)s}(x_{0}))\geq \int_{0}^{(1+\kappa)s}\mathcal{H}^{0}(\Sigma \cap \partial B_{\varrho}(x_{0}))\diff \varrho > \int^{(1+\kappa)s}_{s}\mathcal{H}^{0}(\Sigma \cap \partial B_{\varrho}(x_{0}))\diff \varrho,
\end{equation}
where the latter estimate comes from the fact that $\mathcal{H}^{0}(\Sigma \cap \partial B_{\varrho}(x_{0}))\geq 1$ for all $\varrho \in (0, r]$, since $x_{0}\in \Sigma$, $\Sigma$ is arcwise connected and $r<\diam(\Sigma)/2$. Then there exists $\varrho \in [s, (1+\kappa)s]$ such that
\begin{equation*}
\frac{1}{\kappa s}\mathcal{H}^{1}(\Sigma \cap B_{(1+\kappa)s}(x_{0}))\geq \mathcal{H}^{0}(\Sigma \cap \partial B_{\varrho}(x_{0})).
\end{equation*}
This, together with (\ref{6.20}) and the fact that $s\in [\kappa r, 2\kappa r]$, implies that
\begin{equation}\label{6.22}
\mathcal{H}^{0}(\Sigma \cap \partial B_{\varrho}(x_{0}))\leq \frac{1+\kappa}{\kappa}\theta_{\Sigma}(x_{0}, (1+\kappa)s) \leq \frac{10(1+\kappa)\overline{\mu}}{\kappa^{2}}.
\end{equation} 
Let $L$ realize the infimum in the definition of $\beta_{\Sigma}(x_{0}, \varrho)$ and let $\{\xi_{1},\xi_{2}\}=\partial B_{\varrho}(x_{0}) \cap L$. For each $z_{i} \in \Sigma \cap \partial B_{\varrho}(x_{0})$, let $z_{i}^{\prime}$ denote the projection of $z_{i}$ to $[\xi_{1}, \xi_{2}]$. Define $W$ and $\Sigma^{\prime}$ by
\[
W:=\bigcup_{i=1}^{\mathcal{H}^{0}(\Sigma \cap \partial B_{\varrho}(x_{0}))}[z_{i}, z_{i}^{\prime}], \qquad \Sigma^{\prime}:=W\cup [\xi_{1}, \xi_{2}] \cup (\Sigma \backslash B_{\varrho}(x_{0})).
\]
Then $\Sigma^{\prime} \in \mathcal{K}(\Omega)$, $\Sigma^{\prime}\Delta \Sigma \subset \overline{B}_{\varrho}(x_{0})$ and from (\ref{6.19}) it follows that $\beta_{\Sigma^{\prime}}(x_{0},\varrho) \leq 2\varepsilon/\kappa$. Furthermore, using (\ref{6.20}) and the facts that $\Sigma$ is arcwise connected and $r<\diam(\Sigma)/2$, it is easy to see that $\mathcal{H}^{1}(\Sigma^{\prime})\leq 100\mathcal{H}^{1}(\Sigma)$. Since $\Sigma^{\prime}$ is a competitor,
\[
\mathcal{H}^{1}(\Sigma)\leq \mathcal{H}^{1}(\Sigma^{\prime})+E_{f,\Omega}(u_{\Sigma})-E_{f,\Omega}(u_{\Sigma^{\prime}}),
\]
and then, using Proposition \ref{prop 6.7}, we get
\begin{align*}
&\mathcal{H}^{1}(\Sigma \cap B_{s}(x_{0})) \leq \mathcal{H}^{1}(\Sigma \cap B_{\varrho}(x_{0})) \leq 2 \varrho + \mathcal{H}^{1}(W)+ C\varrho\Bigl(\frac{\varrho}{r}\Bigr)^{b}w^{\tau}_{\Sigma}(x_{0},r)+C\varrho^{1+b}\\
& \qquad \leq 2(1+\kappa)s+\frac{10(1+\kappa)^{2}\overline{\mu}}{\kappa^{2}}\beta_{\Sigma}(x_{0},\varrho)s+ C(1+\kappa)s\biggl(\frac{(1+\kappa)s}{r}\biggr)^{b} w^{\tau}_{\Sigma}(x_{0}, r)+C((1+\kappa)s)^{1+b}, \numberthis \label{6.23}
\end{align*}
where we have used that $\mathcal{H}^{1}(W)\leq (\mathcal{H}^{0}(\Sigma \cap \partial B_{\varrho}(x_{0})))\beta_{\Sigma}(x_{0}, \varrho)\varrho$, (\ref{6.22}) and the fact that $\varrho \leq (1+\kappa)s$. 
Now we define the next three sets
\begin{align*}
&E_{1}:=\{t \in (0, 2\kappa r]: \mathcal{H}^{0}(\Sigma \cap \partial B_{t}(x_{0}))=1\}, \qquad E_{2}:=\{t \in (0, 2\kappa r]: \mathcal{H}^{0}(\Sigma \cap \partial B_{t}(x_{0}))=2\}, \\
&E_{3}:=\{t \in (0, 2\kappa r]: \mathcal{H}^{0}(\Sigma \cap \partial B_{t}(x_{0}))\geq 3\}. 
\end{align*}
We claim that  either $E_{1}=\emptyset$ or $E_{1}\subset  (0, \kappa r/200).$ Assume by contradiction that there exists some $t \in [\kappa r/200, 2\kappa r]$ such that $\mathcal{H}^{0}(\s \cap \partial B_{t}(x_{0})) = 1 $. Then  the set
\[
\Sigma^{\prime \prime}=\Sigma \backslash B_{t}(x_{0})
\]
would be arcwise connected, $\Sigma^{\prime \prime} \Delta \s \subset B_{t}(x_{0})$, $\mathcal{H}^{1}(\Sigma^{\prime \prime})<\mathcal{H}^{1}(\Sigma)$ and
\begin{equation}\label{6.24}
\beta_{\Sigma^{\prime \prime}}(x_{0},r)\leq 2\kappa + \varepsilon <\tau.
\end{equation}  
Since $\Sigma^{\prime \prime}$ is a competitor, $\mathcal{H}^{1}(\Sigma) \leq \mathcal{H}^{1}(\Sigma^{\prime \prime}) + E_{f,\Omega}(u_{\Sigma})-E_{f,\Omega}(u_{\Sigma^{\prime \prime}})$. On the other hand, we observe that $t \leq \mathcal{H}^{1}(\Sigma \cap B_{t}(x_{0}))$, because $t<\diam(\Sigma)/2$, $x_{0}\in \s$ and $\Sigma$ is arcwise connected. Thus
\begin{equation}\label{6.25}
t \leq \mathcal{H}^{1}(\Sigma\cap B_{t}(x_{0})) \leq E_{f, \Omega}(u_{\Sigma})-E_{ f, \Omega}(u_{\Sigma^{\prime \prime}}).
\end{equation}
Notice that, by assumption, the estimate (\ref{6.10}) holds with $C$, but looking at the proof of Proposition~\ref{6.7}, we observe that (\ref{2.7}) in Corollary~\ref{cor 2.15} also holds with $C$. Then, using (\ref{6.25}), Corollary~\ref{cor 2.15}, the fact that $t^{N+p^{\prime}-\frac{Np^{\prime}}{q}} < t^{1+b}$ (because $t\in (0,1)$ and $0<b<N-1+p^{\prime}-Np^{\prime}/q$) and (\ref{6.24}) together with the definition of $w^{\tau}_{\Sigma}(x_{0},r)$, we obtain the following chain of estimates
\begin{align*}
t \leq \mathcal{H}^{1}(\Sigma\cap B_{t}(x_{0})) \leq E_{f, \Omega}(u_{\Sigma})-E_{ f, \Omega}(u_{\Sigma^{\prime \prime}}) &\leq C \int_{B_{2t}(x_{0})}|\nabla u_{\Sigma^{\prime \prime}}|^{p}\diff x + C t^{N+p^{\prime}-\frac{Np^{\prime}}{q}}\\
& \leq C \int_{B_{r}(x_{0})}|\nabla u_{\Sigma^{\prime \prime}}|^{p}\diff x + C r^{1+b}\\
& \leq C r w^{\tau}_{\s}(x_{0}, r) + C r ^{1+b}, 
\end{align*}
leading to a contradiction with the fact that  $\kappa r/200\leq t$, since $Crw^{\tau}_{\Sigma}(x_{0},r)+Cr^{1+b}\leq 2Cr\varepsilon<\kappa r/200$ by (\ref{6.18}) and (\ref{6.17}). Thus, either $E_{1}=\emptyset$ or
\begin{equation}\label{6.26}
E_{1}\subset (0,\kappa r/200).
\end{equation}
Next, by the coarea inequality, 
\begin{equation}\label{6.27}
\mathcal{H}^{1}(\Sigma \cap B_{2\kappa r}(x_{0})) \geq \int^{2\kappa r}_{0}\mathcal{H}^{0}(\Sigma \cap \partial B_{t}(x_{0}))\diff t.
\end{equation}
Also, applying (\ref{6.23}) with $s=2\kappa r$ and using (\ref{6.17}), (\ref{6.18}) and the fact that $\beta_{\Sigma}(x_{0},\varrho) \leq 2\varepsilon/\kappa$, we get the following estimate
\begin{equation}\label{6.28}
\mathcal{H}^{1}(\Sigma \cap B_{2\kappa r}(x_{0})) \leq 4\kappa r + \frac{\kappa r}{200}.
\end{equation}
Then, (\ref{6.26}), (\ref{6.27}) and (\ref{6.28}) together imply 
\begin{align*}
4 \kappa r + \frac{\kappa r}{200} &\geq \mathcal{H}^{1}(E_{1})+ 2 \mathcal{H}^{1}(E_{2})+3\mathcal{H}^{1}(E_{3})\\
&\geq \mathcal{H}^{1}(E_{1})+ 2(2\kappa r - \mathcal{H}^{1}(E_{1})- \mathcal{H}^{1}(E_{3}))+ 3 \mathcal{H}^{1}(E_{3}) \\
&= 4\kappa r - \mathcal{H}^{1}(E_{1})+\mathcal{H}^{1}(E_{3})\\
&>4 \kappa r- \frac{\kappa r}{200} + \mathcal{H}^{1}(E_{3})
\end{align*}
and hence 
\begin{equation}\label{6.29}
\mathcal{H}^{1}(E_{3})<\frac{\kappa r}{100}.
\end{equation}
Notice that (\ref{6.26}) and (\ref{6.29}) yield the following estimate
\[
\mathcal{H}^{1}(E_{2}\cap [\kappa r, 2\kappa r]) > \frac{\kappa r}{2}.
\]
This completes the proof of ($i$).\\
\textit{Step 2.} We prove $(ii)$. Let $t \in E_{2} \cap [\kappa r, 2\kappa r]$. Assume that $(ii$-1) does not hold for $t$. Let $L$ be an affine line realizing the infimum in the definition of $\beta_{\Sigma}(x_{0},t)$, $\{P_{1}, P_{2}\}= L\cap \partial B_{t}(x_{0})$ and $\{z_{1}, z_{2}\}~=\Sigma \cap \partial B_{t}(x_{0})$. Assume that $\dist(z_{i}, \{P_{1}, P_{2}\})= \dist(z_{i}, P_{2}),\, i=1,2.$ Then we can take as a competitor the set 
\begin{equation*}
\Sigma^{\prime \prime \prime}=(\s \backslash B_{t}(x_{0})) \cup \gamma_{z_{1}, P_{2}} \cup \gamma_{z_{2}, P_{2}},
\end{equation*}
where $\gamma_{z_{i}, P_{2}}$ is the geodesic in $\partial B_{t}(x_{0})$ connecting $z_{i}$ and $P_{2}$ for $i=1,2$.
So 
\[
\mathcal{H}^{1}(\s \cap B_{t}(x_{0})) \leq \mathcal{H}^{1}(\gamma_{z_{1}, P_{2}})+\mathcal{H}^{1}(\gamma_{z_{2}, P_{2}}) + E_{f,\Omega}(u_{\s}) - E_{f,\Omega}(u_{\s^{\prime \prime \prime}}).
\]
Arguing as in the proof of the fact that $E_{1}\subset (0, \kappa r/200)$ in \textit{Step 1}, we obtain the estimate
\begin{equation*}
E_{f, \Omega }(u_{\s}) - E_{f, \Omega}(u_{\s^{\prime \prime \prime}}) < \frac{\kappa r}{200}. 
\end{equation*}
In addition, thanks to (\ref{6.19}) and to the fact that $\arcsin(s)\leq 2s$ for all $s \in [0, 1/10]$, 
\[
\mathcal{H}^{1}(\gamma_{z_{1}, P_{2}})+ \mathcal{H}^{1}(\gamma_{z_{2}, P_{2}})\leq 2t\arcsin(\beta_{\Sigma}(x_{0}, t)) \leq \frac{8\varepsilon t}{\kappa}.
\]
But then 
\[
\mathcal{H}^{1}(\Sigma \cap B_{t}(x_{0})) < \frac{\kappa r}{100}
\]
and this leads to a contradiction because $\mathcal{H}^{1}(\s \cap B_{t}(x_{0})) \geq t \geq \kappa r$. Therefore $(ii$-1) holds. Next, assume that $\s \cap \overline{B}_{t}(x_{0})$ is not arcwise connected. Then, from \cite[Lemma 5.13]{Opt}, it follows that $\s \backslash B_{t}(x_{0})$ is arcwise connected. Thus, taking the set $\s \backslash B_{t}(x_{0})$ as a competitor, by analogy with \textit{Step 1}, we get
\[
\mathcal{H}^{1}(\Sigma \cap B_{t}(x_{0})) < \frac{\kappa r}{200},
\]
which, as before, leads to a contradiction. Thus $(ii$-2) holds. Since $\s \cap \partial B_{t}(x_{0})= \{z_{1}, z_{2}\}$, where $z_{1}, z_{2}$ lie ``on different sides", the set $(\s \backslash B_{t}(x_{0})) \cup [z_{1}, z_{2}]$ is a competitor for $\s$, moreover, it fulfills the conditions of Proposition \ref{prop 6.7} and hence  (\ref{6.16}) holds. This proves $(ii)$ and completes the proof of Proposition~\ref{prop 6.8}.
\end{proof}

Now our purpose is to control the density $\theta_{\Sigma}$ from above on a smaller scale by its value on a larger scale, provided that on a larger scale $\beta_{\Sigma}$ and $w^{\tau}_{\Sigma}$ are small enough. Adapting some of the approaches of Stepanov and Paolini in \cite{Stepanov-Paolini}, we prove the following proposition.
\begin{prop}\label{prop 6.10} Let $p \in (N-1,+\infty)$, $f \in L^{q}(\Omega)$ with $q>q_{1}$, where $q_{1}$ is defined in (\ref{1.4}). Then there exists $\delta \in (0,1)$ and for each $a \in (0,1/20]$ there exists $\varepsilon \in (0,1)$ such that the following holds. Assume that $\s$ is a solution to Problem~\ref{P 1.1},  $x_{0} \in \Sigma$, $r\in (0, \min\{\delta, \diam(\Sigma)/2\}),$ $B_{r}(x_{0}) \subset \Omega$ and
	\begin{equation}\label{6.30}
	\beta_{\s}(x_{0}, r) + w^{\tau}_{\Sigma}(x_{0},r)\leq \varepsilon.
	\end{equation}
	Then the following estimate holds
    \begin{equation}\label{6.31}
    \theta_{\Sigma}(x_{0},ar)\leq 5+ \theta_{\Sigma}(x_{0},r)^{1-\frac{1}{N}}.
    \end{equation}
\end{prop}

\begin{proof} Let $\varepsilon_{0}, b, \overline{r} \in (0,1)$ be the constants of Lemma~\ref{lem 4.6} and let $C>0$ be the constant of Proposition~\ref{prop 6.7}. Recall that $\tau \in (0,\varepsilon_{0}/6]$. We define $\delta, \varepsilon \in (0,1)$ as follows
	\begin{equation}\label{6.32}
	\delta=\min\Biggl\{\overline{r}, \biggl(\frac{1}{4C}\biggr)^{\frac{1}{b}}\Biggr\}, \qquad \varepsilon =\frac{a^{2} c_{0}\tau}{10^{7}},
	\end{equation}
	where $c_{0}>0$ is a constant that will be fixed later for the proof to work. 
		It is worth noting that, according to (\ref{6.1}) and (\ref{6.30}), for all $s \in [ar, r]$, it holds
	\begin{equation}\label{6.33}
	\beta_{\Sigma}(x_{0},s)\leq \frac{2}{a}\beta_{\Sigma}(x_{0},r)\leq \frac{2\varepsilon}{a}.
	\end{equation}
	Applying the coarea inequality (see, for instance, \cite[Theorem 2.1]{Paolini-Stepanov}), we get
	\[ 
	\mathcal{H}^{1}(\Sigma \cap B_{r}(x_{0}))\geq \int^{r}_{0}\mathcal{H}^{0}(\Sigma \cap \partial B_{\varrho}(x_{0}))\diff \varrho > \int^{2ar}_{ar}\mathcal{H}^{0}(\Sigma \cap \partial B_{\varrho}(x_{0}))\diff \varrho,
	\]
	where the latter inequality comes from the fact that for all $\varrho \in (0, r]$, $\mathcal{H}^{0}(\Sigma \cap \partial B_{\varrho}(x_{0}))\geq 1$, since $\Sigma$ is arcwise connected (see Remark~\ref{arcwise connected}), $x_{0}\in \Sigma$ and $r<\diam(\Sigma)/2$. Then there exists $\varrho \in [ar, 2ar]$ such that 
	\begin{equation}\label{6.34}
	\mathcal{H}^{0}(\Sigma \cap \partial B_{\varrho}(x_{0}))\leq \frac{1}{a}\theta_{\Sigma}(x_{0}, r).
	\end{equation}
	Next, we construct  the competitor $\Sigma^{\prime}$ for $\Sigma$ such that $\Sigma^{\prime} \Delta \Sigma\subset \overline{B}_{\varrho}(x_{0})$, $\mathcal{H}^{1}(\Sigma^{\prime})\leq 100\mathcal{H}^{1}(\Sigma)$ and $\beta_{\Sigma^{\prime}}(x_{0},\varrho) \leq \beta_{\Sigma}(x_{0}, \varrho)$. Let $L\subset \mathbb{R}^{N}$ be an affine line realizing the infimum in the definition of $\beta_{\Sigma}(x_{0},\varrho)$. We denote by $A_{1}$ and $A_{2}$ the two points in $\partial B_{\varrho}(x_{0})\cap L$ and denote by $G_{n}$ the set of all points $(x^{\prime}, x_{N})$ in $[-1,1]^{N}$ such that $nx_{i} \in \mathbb{Z}$ for all $i=1,...,N$ except for at most one (i.e., $G_{n}$ is a uniform 1-dimensional grid of step $1/n$ in $[-1,1]^{N}$). Notice that $G_{n}$ is arcwise connected and 
	 \begin{equation}\label{6.35}
	 \mathcal{H}^{1}(G_{n}) \leq 2^{N}N(n+1)^{N-1},\,\ \dist (y, G_{n})\leq \frac{\sqrt{N}}{2n}
	 \end{equation}
	 for all $y \in [-1,1]^{N}$. Let $h:\mathbb{R}^{N} \to \mathbb{R}^{N}$ be the rotation around the origin such that $h(\mathbb{R}e_{N})=L-x_{0}$, where $\{e_{1},...,e_{N}\}$ is the canonical basis for $\mathbb{R}^{N}$. Next, we define $Q^{i}_{n}:=A_{i}+\beta_{\Sigma}(x_{0},\varrho)\varrho h(G_{n}),\, i=1,2$. In addition, we observe that \[\Sigma \cap \partial B_{\varrho}(x_{0}) \subset \partial B_{\varrho}(x_{0}) \cap \bigl\{x \in \mathbb{R}^{N}: \dist(x,L)\leq \beta_{\Sigma}(x_{0},\varrho)\varrho\bigr\} \subset \bigcup_{i=1}^{2}\bigl(A_{i}+\beta_{\Sigma}(x_{0},\varrho)\varrho h\bigl([-1,1]^{N}\bigr)\bigr).
	 \] 
	 For each point $z_{j} \in \Sigma \cap \partial B_{\varrho}(x_{0})$, we denote by $z^{n}_{j}$ an arbitrary projection of $z_{j}$ to $Q^{1}_{n}\cup Q^{2}_{n}$ and by $[z_{j}, z^{n}_{j}]$ the segment connecting these two points. Then the set 
	 \[
	 S_{n}=Q^{1}_{n}\cup Q^{2}_{n}\cup \Biggl(\bigcup^{\mathcal{H}^{0}(\Sigma \cap \partial B_{\varrho}(x_{0}))}_{j=1}[z_{j},z^{n}_{j}]\Biggr)
	 \]
	 contains all the points of $\Sigma \cap \partial B_{\varrho}(x_{0})$, $S_{n}\cup (L\cap B_{\varrho}(x_{0}))$ is arcwise connected, and, using (\ref{6.35}), we have that
	 \[
	 \mathcal{H}^{1}(S_{n})\leq 2^{N+1} N(n+1)^{N-1}\beta_{\Sigma}(x_{0},\varrho)\varrho+\frac{\sqrt{N}}{2n}\mathcal{H}^{0}(\Sigma\cap \partial B_{\varrho}(x_{0}))\beta_{\Sigma}(x_{0}, \varrho)\varrho.
	 \]
	 Let $\widetilde{S}_{n}$ be the projection of $S_{n}$ to $\{x \in \mathbb{R}^{N}: \dist(x,L)\leq \beta_{\Sigma}(x_{0},\varrho)\varrho\}\cap \overline{B}_{\varrho}(x_{0})$. Since the projection onto a nonempty closed convex set is a $1$-Lipschitz mapping, it follows that $\mathcal{H}^{1}(\widetilde{S}_{n})\leq \mathcal{H}^{1}(S_{n})$. Moreover, notice that $\widetilde{S}_{n}\cup (L\cap B_{\varrho}(x_{0}))$ is arcwise connected. Thus, defining
	 \[
	 \Sigma^{\prime}=(\Sigma \backslash B_{\varrho}(x_{0})) \cup \widetilde{S}_{n} \cup (L\cap B_{\varrho}(x_{0}))
	 \]
	 and choosing $n=\floor{(\mathcal{H}^{0}(\Sigma \cap \partial B_{\varrho}(x_{0})))^{\frac{1}{N}}}$, where $\floor{\cdot}$ denotes the integer part, we observe that 
	 \begin{equation}\label{6.36}
	 \mathcal{H}^{1}(\widetilde{S}_{n})\leq M_{0}(\mathcal{H}^{0}(\Sigma \cap \partial B_{\varrho}(x_{0})))^{1-\frac{1}{N}}\beta_{\Sigma}(x_{0},\varrho)\varrho,
	 \end{equation}
	 where $M_{0}=M_{0}(N)>0$. Now we can set 
      \begin{equation}\label{6.37}
      c_{0}=(M_{0}C)^{-1}. 
      \end{equation}
      Thanks to (\ref{6.34}) and (\ref{6.36}), we obtain
      \[ 
      \mathcal{H}^{1}(\widetilde{S}_{n})< M_{0}\biggl(\frac{1}{a}\theta_{\Sigma}(x_{0}, r)\biggr)^{1-\frac{1}{N}}\beta_{\Sigma}(x_{0},\varrho)\varrho.
      \]
      This, together with (\ref{6.33}), (\ref{6.32}), (\ref{6.37}) and the fact that $2\varrho\leq 4ar < \diam(\Sigma)\leq \mathcal{H}^{1}(\Sigma)$, implies the following \[\mathcal{H}^{1}(\Sigma^{\prime})<100 \mathcal{H}^{1}(\Sigma).\] Also notice that $\Sigma^{\prime}\subset \overline{\Omega}$ is closed, arcwise connected, $\Sigma^{\prime}\Delta \Sigma \subset \overline{B}_{\varrho}(x_{0})$, 
      \[
      \beta_{\Sigma^{\prime}}(x_{0}, \varrho)\leq \beta_{\Sigma}(x_{0}, \varrho)\leq \frac{2\varepsilon}{a}<\tau
     \] 
      (see (\ref{6.33}), (\ref{6.32})). So we can apply Proposition~\ref{prop 6.7} to $\Sigma$ and $\Sigma^{\prime}$.
	 Thus, by the optimality of $\Sigma$ and Proposition~\ref{prop 6.7},  
	 \begin{equation*}
	 \mathcal{H}^{1}(\Sigma)\leq E_{f, \Omega}(u_{\Sigma})-E_{f,\Omega}(u_{\Sigma^{\prime}})+ \mathcal{H}^{1}(\Sigma^{\prime}) \leq C \varrho\Bigl(\frac{\varrho}{r}\Bigr)^{b}w^{\tau}_{\Sigma}(x_{0}, r) +C\varrho^{1+b}+\mathcal{H}^{1}(\Sigma^{\prime}).
	 \end{equation*}
     Altogether we have
    \[
    \mathcal{H}^{1}(\Sigma \cap B_{ar}(x_{0}))\leq \mathcal{H}^{1}(\Sigma \cap B_{\varrho}(x_{0})) \leq C\varrho\Bigl(\frac{\varrho}{r}\Bigr)^{b} w^{\tau}_{\Sigma}(x_{0}, r)+C\varrho^{1+b} +2\varrho+ M_{0}\biggl(\frac{1}{a}\theta_{\Sigma}(x_{0},r)\biggr)^{1-\frac{1}{N}}\beta_{\Sigma}(x_{0},\varrho)\varrho.
    \]
    Next, recalling that $\varrho \in [ar, 2ar]$, $r<\delta$, $(2a)^{b} < 1$ and (\ref{6.33}), we obtain
    \begin{equation*}
    \theta_{\Sigma}(x_{0}, ar) \leq 2C\bigl(w^{\tau}_{\Sigma}(x_{0},r)+\delta^{b}\bigr)+4 + \frac{4 \varepsilon M_{0}}{a}\biggl(\frac{1}{a}\theta_{\Sigma}(x_{0},r)\biggr)^{1-\frac{1}{N}}.
    \end{equation*}
    However, this, together with (\ref{6.30}), (\ref{6.32}) and (\ref{6.37}), yields the estimate
   \[
    \theta_{\Sigma}(x_{0},ar)\leq 5 + \theta_{\Sigma}(x_{0},r)^{1-\frac{1}{N}}
   \]
   and completes the proof of Proposition~\ref{6.10}.
\end{proof}

\begin{subsection}{Control of the flatness}
The next proposition asserts that if $\beta_{\Sigma}(x,r)$ and $w^{\tau}_{\Sigma}(x,r)$ are pretty small and $\theta_{\Sigma}(x,r)$ is controlled from above by $10 \overline{\mu}$, where $\overline{\mu}$ is a unique positive solution to the equation $\mu=5+\mu^{1-\frac{1}{N}}$, then $\beta_{\Sigma}$, $w^{\tau}_{\Sigma}$ stay small and $\theta_{\Sigma}$ remains controlled from above by $10\overline{\mu}$ on smaller scales, and, in addition, in some sense $w^{\tau}_{\Sigma}$ controls the square of $\beta_{\Sigma}$.
\begin{prop} \label{prop 6.11} Let $p \in (N-1,+\infty)$, $f \in L^{q}(\Om)$  with $q> q_{1}$, where $q_{1}$ is defined in (\ref{1.4}). Then there exist constants $a, r_{0}\in (0,1/100)$, $b\in (0,1)$, $0<\delta_{1}<\delta_{2}<1/100$ and $C=C(N, p, q_{0},q,  \|f\|_{q}, |\Omega|)>0$ with $q_{0}$ defined in (\ref{Eq 1.1}) such that the following holds. Assume that $\s$ is a solution to Problem~\ref{P 1.1}, $x \in \s$, $0<r<\min\{r_{0},\diam(\Sigma)/2\}$, ${B}_{r}(x) \subset \Omega$, 
\begin{equation} \label{6.38}
w^{\tau}_{\Sigma}(x,r)\leq \delta_{1}, \,\ \beta_{\Sigma}(x,r)\leq \delta_{2} \,\ \text{and} \,\ \theta_{\Sigma}(x,r) \leq 10\overline{\mu},
\end{equation}
where $\overline{\mu}>0$ is a unique positive solution to the equation $\mu=5+\mu^{1-\frac{1}{N}}$. Then 
\begin{enumerate}[label=(\roman*)]
	\item 
	\begin{equation}\label{6.39}
	\beta_{\s}(x, a r) \leq C(w^{\tau}_{\s}(x, r))^{\frac{1}{2}} + C r^{\frac{b}{2}};
	\end{equation}
	
	\item 
	
	\begin{equation}\label{6.40}
	w^{\tau}_{\Sigma}(x, a r)\leq \frac{1}{2} w^{\tau}_{\Sigma}(x, r)+ C(a r)^{b};
	\end{equation}
	\item 
	\begin{equation}\label{6.41}
	 w^{\tau}_{\s}(x, a^{n}r) \leq \delta_{1}, \,\ \beta_{\s}(x, a^{n}r)\leq \delta_{2}, \,\ \theta_{\Sigma}(x,a^{n}r)\leq 10\overline{\mu} \,\ \text{for all} \,\ n\in \mathbb{N}.
	\end{equation}
\end{enumerate}
\end{prop}
\begin{proof} Let $C_{0}$ be the constant such that the estimate (\ref{4.31}) holds with $C_{0}$, and let $C_{1}$ be the constant such that the estimate (\ref{6.16}) holds with $C_{1}$. Without loss of generality, we can assume that $C_{0}<C_{1}$. Let $b \in (0,1)$ be the constant of Lemma~\ref{lem 4.6}, and let $a,\delta, \varepsilon, \kappa \in (0,1/100)$ be such that $\delta, \varepsilon, \kappa$ are the constants of Proposition~\ref{prop 6.8} and, at the same time, $a, \delta, \varepsilon$ are the constants of Proposition~\ref{prop 6.10} with
	\[ a=\min\Biggl\{\kappa ,\biggl(\frac{1}{2C_{0}}\biggr)^{\frac{1}{b}}\Biggr\}.\,\ 
	\]
	Now we can set 
	\begin{equation}\label{the constants}
	\delta_{2}:=\frac{a\varepsilon}{2} ,\,\ \delta_{1}:=\biggl(\frac{a \delta_{2}}{50 C_{1}}\biggr)^{2},\,\  C:=\frac{24C_{1}}{a} 
	\end{equation}
	and fix $r_{0} \in (0, \delta)$ such that 
	\begin{equation}\label{6.42}
	C r_{0}^{\frac{b}{2}} \leq \frac{\delta_{1}}{2}. 
	\end{equation}
	\textit{Step 1}. Let us first prove $(i)$. By Proposition~\ref{prop 6.8}, there exists $t \in [\kappa r, 2 \kappa r]$ such that $ \s \cap \partial B_{t}(x)=\{z_{1}, z_{2}\}$, $z_{1}$ and $z_{2}$ lie  ``on different sides" (see Remark \ref{rem 6.9}). According to Proposition \ref{prop 6.8} $(ii$-3), we get 
	\begin{align*}
	\mathcal{H}^{1}(\s \cap B_{t}(x))& \leq |z_{1}-z_{2}| + C_{1}t\Bigl(\frac{t}{r}\Bigr)^{b} w^{\tau}_{\s}(x, r) + C_{1}t^{1+b} := |z_{1} - z_{2}| + M. 
	\end{align*}
	Recall that, by Proposition~\ref{prop 6.8} ($ii$-2), $\Sigma \cap \overline{B}_{t}(x)$ is arcwise connected. Let $\Gamma \subset \s \cap \overline{B}_{t}(x)$ be an arc connecting $z_{1}$ with $z_{2}$. Then, using Lemma~\ref{lem A.3}, we obtain 
	\begin{align*}
	\max_{y \in \Gamma}\dist(y, [z_{1}, z_{2}]) & \leq (2t(\mathcal{H}^{1}(\Gamma)- |z_{1}-z_{2}|))^{\frac{1}{2}} \leq (4 \kappa rM)^{\frac{1}{2}}. 
	\end{align*}
	Since $\s \cap \overline{B}_{t}(x)$ is arcwise connected, $\Sigma \cap \partial B_{t}(x)=\{z_{1},z_{2}\}$ and $\mathcal{H}^{1}(\Gamma)\geq |z_{1}-z_{2}|$,
	\begin{align*}
	\sup_{y\in (\s \cap \overline{B}_{t}(x))\backslash(\Gamma \cap B_{t}(x))} \dist (y, \Gamma) \leq \mathcal{H}^{1}(\s \cap B_{t}(x) \backslash \Gamma) & \leq \mathcal{H}^{1}(\s \cap B_{t}(x)) - |z_{1} -z_{2}| \leq M.
	\end{align*}
	Thus
	\begin{equation*}
	\max_{y \in \s \cap \overline{B}_{t}(x)} \dist(y, [z_{1}, z_{2}]) \leq (4\kappa rM)^{\frac{1}{2}}+M
	\end{equation*}
	but this yields the following estimate
	\begin{equation}\label{6.43}
	d_{H}(\s \cap \overline{B}_{t}(x), [z_{1}, z_{2}]) \leq (4\kappa rM)^{\frac{1}{2}}+M,
	\end{equation}
	because $\Sigma\cap \overline{B}_{t}(x)$ is arcwise connected and $\Sigma$ escapes $\partial B_{t}(x)$ either through $z_{1}$ or through $z_{2}$. Without loss of generality, assume that $[z_{1},z_{2}]$ is not a diameter of $\overline{B}_{t}(x)$, otherwise we can pass directly to the estimate (\ref{6.45}). So let $\widetilde{L}$ be the line passing through $x$ and collinear to $[z_{1},z_{2}]$. Now observe that if $\Pi$ is the $2$-dimensional plane passing through $\widetilde{L}$ and $[z_{1},z_{2}]$, then the intersection of $\Pi$ with $\partial B_{t}(x)$ is the circle $S$ on $\Pi$ with center $x$ and radius $t$. Then, denoting by $\xi_{1}$ and $\xi_{2}$ the two points in $\widetilde{L} \cap \partial B_{t}(x)$ in such a way that $\dist(\xi_{i}, \{z_{1},z_{2}\})=\dist(\xi_{i}, z_{i})$ for $i=1,2$, we get
	\begin{equation}\label{estimate inside}
	d_{H}([z_{1},z_{2}], \widetilde{L} \cap \overline{B}_{t}(x)) \leq \mathcal{H}^{1}(\gamma_{z_{1}, \xi_{1}})=\mathcal{H}^{1}(\gamma_{z_{2}, \xi_{2}}),
	\end{equation}
	where $\gamma_{z_{i}, \xi_{i}}$ is the geodesic in $S$ connecting $z_{i}$ with $\xi_{i}$. Since $\dist(x, [z_{1},z_{2}])\leq (4\kappa rM)^{\frac{1}{2}}+M$ (see (\ref{6.43})), 
	\begin{equation}\label{6.44}
	\mathcal{H}^{1}(\gamma_{z_{1},\xi_{1}}) \leq \arcsin \biggl(\frac{(4\kappa rM)^{\frac{1}{2}}+M}{t}\biggr)t \leq 2((4\kappa rM)^{\frac{1}{2}}+M), 
	\end{equation}
	where the latter estimate holds because $((4 \kappa rM)^{\frac{1}{2}}+M)/t<1/10$ and $\arcsin(s)\leq 2s$ for all $s \in [0, 1/10]$. Using (\ref{6.43}) together with (\ref{estimate inside}) and (\ref{6.44}), we obtain that
	\begin{equation*}
	d_{H}(\s\cap \overline{B}_{t}(x), \widetilde{L} \cap \overline{B}_{t}(x)) \leq 3((4\kappa rM)^{\frac{1}{2}}+M)
	\end{equation*}
	and hence $\beta_{\s}(x,t) \leq 3((4\kappa rM)^{\frac{1}{2}}+M)/t.$ Next, since $t \in [\kappa r, 2\kappa r]$ and $a \in (0, \kappa]$, if $ar=\lambda t$ for some $\lambda \in (0,1]$, then $2/\lambda\leq 4\kappa/a$ and, thanks to (\ref{6.1}),
	\begin{align*}
	\beta_{\s}(x, ar)=\beta_{\Sigma}(x, \lambda t) & \leq \frac{4\kappa}{a} \beta_{\s}(x,t) \leq \frac{12}{a r}((4\kappa rM)^{\frac{1}{2}}+M). \numberthis\label{6.45}
	\end{align*}
	On the other hand, since $\kappa, w^{\tau}_{\Sigma}(x,r), r \in (0,1/100)$ and $b \in(0,1)$, we can conclude the following
	\begin{align*}
	(4\kappa r M)^{\frac{1}{2}} \leq \bigl(C_{1} r^{2}w^{\tau}_{\s}(x,r) + C_{1} r^{2+b}\bigr)^{\frac{1}{2}} \leq C_{1} r (w^{\tau}_{\s}(x, r))^{\frac{1}{2}} + C_{1} r^{1+\frac{b}{2}} \numberthis \label{9.40}
	\end{align*}
	and, moreover, 
	\begin{equation}\label{6.47}
	M= C_{1}t\Bigl(\frac{t}{r}\Bigr)^{b} w^{\tau}_{\s}(x, r)+ C_{1}t^{1+b} \leq C_{1} r (w^{\tau}_{\s}(x, r))^{\frac{1}{2}} + C_{1} r^{1+\frac{b}{2}}. 
	\end{equation}
	By (\ref{6.45})-(\ref{6.47}),
	\[
	\beta_{\s}(x,ar) \leq C(w^{\tau}_{\s}(x, r))^{\frac{1}{2}} + Cr^{\frac{b}{2}},
	\]
	with $C=24  C_{1}/a$. Using (\ref{6.38}), the above estimate, (\ref{the constants}) and (\ref{6.42}), we get 
	\[
	\beta_{\s}(x, ar) \leq C(\delta_{1})^{\frac{1}{2}}+ C r_{0}^{\frac{b}{2}} < \delta_{2}.
	\]  Next, observe that $a<1/100$ and $\beta_{\Sigma}(x, s)$ is fairly small for all $s \in [ar, r]$, so we can apply Proposition~\ref{prop 6.6} with $r_{0}=ar$ and $r_{1}=r$ to get the following
	\begin{equation*}
	w^{\tau}_{\s}(x, ar) \leq C_{0} a^{b} w^{\tau}_{\s}(x, r) + C_{0} (ar)^{b} \leq \frac{1}{2} w^{\tau}_{\s}(x, r) + C(ar)^{b} \leq \frac{\delta_{1}}{2} + \frac{\delta_{1}}{2} = \delta_{1},
	\end{equation*}
	where we have used that $a\leq (1/2C_{0})^{\frac{1}{b}}$, $C_{0}(ar)^{b}<C(ar)^{b}<Cr^{\frac{b}{2}}_{0}$, (\ref{6.38}) and (\ref{6.42}). We have proved the assertions $(i)$, $(ii)$ and that $w^{\tau}_{\Sigma}(x, ar)\leq \delta_{1}$, $\beta_{\Sigma}(x, ar)\leq \delta_{2}$.
     \\
    \textit{Step 2}. We prove $(iii)$. Recall that $a, \delta, \varepsilon \in (0,1/100)$ are the constants of Proposition~\ref{prop 6.10} and, by definition, $\delta_{1}<\delta_{2}=a\varepsilon/2$. Then, according to (\ref{6.38}), 
    \[\beta_{\Sigma}(x,r) +w^{\tau}_{\Sigma}(x,r)\leq \varepsilon.\] Thus, applying Proposition~\ref{prop 6.10} and using again (\ref{6.38}), we get
    \[\theta_{\Sigma}(x,ar)\leq 5+\theta_{\Sigma}(x,r)^{1-\frac{1}{N}}\leq 5+(10\overline{\mu})^{1-\frac{1}{N}}\leq  10\bigl(5+\overline{\mu}^{1-\frac{1}{N}}\bigr)=10\overline{\mu}.\]
      At this point, we have shown that (\ref{6.38}) holds with $r$ replaced by $ar$. So, repeating the arguments above, we observe that (\ref{6.38}) holds with $r$ replaced by $a^{2}r$. Therefore, iterating, we deduce $(iii)$. This completes the proof of Proposition~\ref{prop 6.11}.
\end{proof}

Now we prove that there exist a critical threshold $\delta_{0}\in (0,1/100)$ and an exponent $\alpha \in (0,1)$ such that if $\beta_{\s}(x, r) + w^{\tau}_{\s}(x,r)$ falls below $\delta_{0}$ and if $\theta_{\Sigma}(x ,r)$ is small enough for   $x \in \s \cap \Omega$ and fairly small $r>0$, then $\beta_{\s}(x,\varrho) \leq C\varrho^{\alpha}$ for all sufficiently small $\varrho>0$, where $C>0$ is a constant independent of $x$ but depending on $r$. This leads to the $C^{1,\alpha}$ regularity.

\begin{prop}\label{prop 6.12} Let $p \in (N-1,+\infty)$, $f \in L^{q}(\Om)$  with $q> q_{1}$, where $q_{1}$ is defined in (\ref{1.4}). Let $a \in (0,1/100)$ be the constant of Proposition~\ref{prop 6.11}. Then there exist constants $\delta_{0}, \overline{r}_{0} \in (0,1/100)$ and $\alpha\in (0,1)$ such that the following holds. Assume that $\Sigma$ is a solution to Problem~\ref{P 1.1}. If $x \in \s$ and $0<r<\min\{\overline{r}_{0}, \diam(\Sigma)/2\}$ satisfy $B_{r}(x) \subset \Omega$,
	\begin{equation}\label{6.48}
	\beta_{\s}(x,r) + w^{\tau}_{\s}(x,r) \leq \delta_{0} \,\ \text{and}\,\ \theta_{\Sigma}(x, r) \leq 10\overline{\mu}
	\end{equation}
	with $\overline{\mu}$ being a unique positive solution to the equation $\mu=5+\mu^{1-\frac{1}{N}}$, then
	\begin{equation}\label{6.49}
	\beta_{\s}(x, \varrho) \leq C \varrho^{\alpha} \,\ \text{for all}\,\ \varrho \in (0, ar)
	\end{equation}
	and for some constant $C=C(N, p,q_{0}, q,\|f\|_{q}, |\Omega|,r)>0$, where $q_{0}$ is defined in (\ref{Eq 1.1}).
\end{prop}
\begin{proof} Let $a, \delta_{1}, r_{0} \in (0,1/100),$ $b\in (0,1)$ and $C>0$ be as in Proposition~\ref{prop 6.11}. We define \[\delta_{0}:=\delta_{1},\,\ \gamma:=\min\biggl\{\frac{b}{2}, \frac{\ln(3/4)}{\ln(a)}\biggr\},\,\ \overline{r}_{0}:=\min\Biggl\{r_{0}, \biggl(\frac{1}{4}\biggr)^{\frac{1}{\gamma}}\Biggr\}.\]  It is easy to check that for all $t \in (0, \overline{r}_{0}]$,
	\begin{equation} \label{6.50}
	\frac{1}{2}t^{\gamma}+t^{b}\leq (at)^{\gamma}.
	\end{equation} 
	Indeed, since $0<2\gamma \leq b$, $\gamma\leq \ln(3/4)/\ln(a)$ and $a,\overline{r}_{0}\in (0,1)$, $t^{b}\leq t^{2\gamma}\leq \overline{r}_{0}^{\gamma}t^{\gamma}$ and $3/4\leq a^{\gamma}$, so
	\[
	\frac{1}{2}t^{\gamma}+t^{b}\leq \frac{1}{2}t^{\gamma}+\overline{r}_{0}^{\gamma}t^{\gamma} \leq \frac{3}{4}t^{\gamma} \leq (at)^{\gamma}.
	\]
	We prove by induction that for all $n \in \mathbb{N}$,
	\begin{equation}\label{6.51}
	w^{\tau}_{\s}(x,a^{n}r) \leq \frac{1}{2^{n}}w^{\tau}_{\s}(x,r) + C(a^{n+1}r)^{\gamma}.
	\end{equation}
	Clearly, (\ref{6.51}) holds for $n=0$. Suppose (\ref{6.51}) holds for some $n \in \mathbb{N}$. Then, applying (\ref{6.40}) with $r$ replaced by $a^{n}r$ and using the induction hypothesis, we get
	\begin{align*}
	w^{\tau}_{\s}(x,a^{n+1}r) & \leq \frac{1}{2}w^{\tau}_{\s}(x, a^{n}r) + C(a^{n+1}r)^{b}\\
	& \leq \frac{1}{2^{n+1}}w^{\tau}_{\s}(x, r) + \frac{C}{2}(a^{n+1}r)^{\gamma} + C(a^{n+1}r)^{b}                  \\
	& \leq \frac{1}{2^{n+1}}w^{\tau}_{\s}(x,r) + C(a^{n+2}r)^{\gamma},
	\end{align*}
	where the last estimate comes by using (\ref{6.50}). This proves (\ref{6.51}). Now let $\varrho \in (0, ar)$ and let $l\geq1$ be the integer such that $a^{l+1}r<\varrho \leq a^{l}r$. Then, using if necessary (\ref{6.1}), we see that $\beta_{\Sigma}(x, \varrho)\leq 2\beta_{\Sigma}(x, a^{l}r)/a$. Furthermore, Proposition~\ref{prop 6.11} $(i)$ says that 
	\[
	\beta_{\Sigma}(x, a^{l}r) \leq C(w^{\tau}_{\Sigma}(x,a^{l-1}r))^{\frac{1}{2}}+C(a^{l-1}r)^{\frac{b}{2}}. 
	\]
	On the other hand, using (\ref{6.51}) and the fact that $w^{\tau}_{\Sigma}(x,r) < 1$, we get  
	\begin{align*}
	w^{\tau}_{\Sigma}(x,a^{l-1}r)\leq \frac{1}{2^{l-1}}w^{\tau}_{\Sigma}(x,r)+C(a^{l}r)^{\gamma}
	\leq C^{\prime}\Bigl(\frac{3}{4}\Bigr)^{l+1}+C^{\prime}(a^{l+1}r)^{\gamma}
	&\leq C^{\prime}a^{\gamma (l+1)} + C^{\prime}\varrho^{\gamma}\\
	&\leq C^{\prime}\Bigl(\frac{\varrho}{r}\Bigr)^{\gamma}+C^{\prime}\varrho^{\gamma}
	\end{align*}
	for some $C^{\prime}=C^{\prime}(N,p,q_{0},q,\|f\|_{q},|\Omega|)>0$. So we can control $\beta_{\Sigma}(x,\varrho)$ as follows 
	\begin{align*}
	\beta_{\Sigma}(x, \varrho)\leq \frac{2}{a}\beta_{\Sigma}(x, a^{l}r) &\leq \frac{2C}{a} (w^{\tau}_{\s}(x, a^{l-1}r))^{\frac{1}{2}} + \frac{2C}{a} (a^{l-1}r)^{\frac{b}{2}} \\  & \leq C^{\prime \prime} \Bigl(\frac{\varrho}{r}\Bigr)^{\frac{\gamma}{2}} + C^{\prime \prime}\varrho^{\frac{\gamma}{2}} + C^{\prime \prime}\varrho^{\frac{b}{2}}\\
	&  \leq \widetilde{C} \varrho^{\frac{\gamma}{2}} \,\  (\gamma \leq b/2),
	\end{align*}
  where $\widetilde{C}=\widetilde{C}(N, p,  q_{0}, q,\|f\|_{q}, |\Omega|,r)>0$.  Setting $\alpha = \gamma/2$ and $C:=\widetilde{C}$, we complete the proof of Proposition~\ref{prop 6.12}.
\end{proof}

\begin{cor} \label{cor 6.13} \textit{Let $\s$ be a solution to Problem~\ref{P 1.1} and $a,\,\alpha,\,  \delta_{0},\, \overline{r}_{0},\,\overline{\mu}$ be the constants as in the statement of Proposition~\ref{prop 6.12}. Assume that $x \in \s$, $0<r<\min\{\overline{r}_{0}, \diam(\Sigma)/2\}$, $B_{r}(x) \subset \Omega$,
		\begin{equation*}
		\beta_{\s}(x, r) + w^{\tau}_{\s}(x, r) \leq \varepsilon \,\ \text{and} \,\ \theta_{\Sigma}(x, r) \leq \overline{\mu}
		\end{equation*}
		with $\varepsilon := \delta_{0}/200$. Then for any point $y \in \s \cap B_{a r/10}(x)$ and radius $\varrho \in (0, ar/10)$ the following estimate holds 
		\[
		\beta_{\s}(y,\varrho) \leq C \varrho^{\alpha},
		\]
		where $C=C(N,p,q_{0},q,\|f\|_{q},|\Omega|,r)>0$. In particular, there exists $t_{0}\in (0,1)$ such that $\s \cap \overline{B}_{t_{0}}(x)$ is a $C^{1,\alpha}$ regular curve.}
\end{cor}
\begin{proof}[Proof of Corollary \ref{cor 6.13}] Recall that $a \in (0,1/100)$. Let $y \in \s \cap B_{a r/10}(x)$ and $L_{x}$ realize the infimum in the definition of $\beta_{\s}(x, r)$. Notice that $d_{H}(\s \cap \overline{B}_{r/10}(y), L_{x} \cap \overline{B}_{r/10}(y)) \leq 5\varepsilon r$. Let $L$ be the affine line passing through $y$ and collinear to $L_{x}$. It is easy to see that $d_{H}(L_{x} \cap \overline{B}_{r/10}(y), L\cap \overline{B}_{r/10}(y))\leq 5\varepsilon r$ and hence
	\[
	d_{H}\bigl(\Sigma \cap \overline{B}_{\frac{r}{10}}(y), L\cap \overline{B}_{\frac{r}{10}}(y)\bigr) \leq d_{H}\bigl(\Sigma \cap \overline{B}_{\frac{r}{10}}(y), L_{x}\cap \overline{B}_{\frac{r}{10}}(y)\bigr)+ d_{H}\bigl(L_{x}\cap \overline{B}_{\frac{r}{10}}(y), L \cap \overline{B}_{\frac{r}{10}}(y)\bigr)\leq 10\varepsilon r.
	\] 
	Thus $\beta_{\s}(y,r/10) \leq \delta_{0}/2.$ Next, let $\s^{\prime}$ realize the supremum in the definition of $w^{\tau}_{\s}(y, r/10)$. Such $\Sigma^{\prime}$ exists due to the condition $\beta_{\Sigma}(y,r/10)\leq \delta_{0}/2\leq \tau$ (see Remark~\ref{rem 6.4}). Then we have that
	\begin{align*}
	w^{\tau}_{\s}\biggl(y, \frac{r}{10}\biggr)  = \frac{10}{ r} \int_{B_{\frac{r}{10}}(y)} |\nabla u_{\s^{\prime}}|^{p} \diff z \leq \frac{10}{ r} \int_{B_{r}(x)} |\nabla u_{\s^{\prime}}|^{p} \diff z \leq 10 w^{\tau}_{\s}(x, r) < \frac{\delta_{0}}{2},
	\end{align*}
	where we have used the facts that $B_{r/10}(y) \subset B_{(1+a)r/10}(x)$, $\beta_{\Sigma}(y,r/10)$ and $\beta_{\Sigma}(x, r)$ are pretty small, namely, proceeding as in the proof of Proposition~\ref{prop 6.5}, we can show that $\beta_{\Sigma^{\prime}}(x, r) \leq \tau$. Thus
	\begin{equation*}
	\beta_{\s}\biggl(y, \frac{r}{10}\biggr) + w^{\tau}_{\s}\biggl(y, \frac{r}{10}\biggr) < \delta_{0}.
	\end{equation*}
	On the other hand, $\theta_{\Sigma}(y, r/10)\leq 10 \theta_{\Sigma}(x,r)\leq 10 \overline{\mu}$. Then, according to ~Proposition~\ref{prop 6.12}, $\beta_{\s}(y, \varrho) \leq C\varrho^{\alpha}$ for all $\varrho \in (0, a r/10)$. Since the point $y$ was arbitrarily chosen in $\s \cap B_{a r/10}(x)$, there exists $t_{0} \in (0, a r/10)$ such that $\s \cap \overline{B}_{t_{0}}(x)$ is a $C^{1,\alpha}$ regular curve (see, for instance, \cite[Proposition 9.1]{David-Kenig-Toro}).
\end{proof}

\begin{proof}[Proof of Theorem \ref{thm 1.2}] Let $\varepsilon_{0},b,\overline{r} \in (0,1),\, C>0$ be the constants of Lemma~\ref{lem 4.6}. Since closed connected sets with finite $\mathcal{H}^{1}$-measure ~are $\mathcal{H}^{1}$-rectifiable (see \cite[Proposition 30.1, p. 186]{David}), then (see Lemma~\ref{lem 5.2}) for $\mathcal{H}^{1}$-a.e. point $x$ in $\s$ there exists the affine line $T_{x}$ passing through $x$ such that
	\begin{equation}\label{6.52}
	\frac{1}{r}d_{H}(\s \cap \overline{B}_{r}(x), T_{x} \cap \overline{B}_{r}(x)) \underset{r \to 0+}{\to} 0. 
	\end{equation}
	On the other hand,
	\begin{equation}\label{6.53}
	\theta_{\Sigma}(x,r) \underset{r \to 0+}{\to} 2
	\end{equation}
	for $\mathcal{H}^{1}$-a.e. $x \in \Sigma$, in view of Besicovitch-Marstrand-Mattila Theorem (see \cite[Theorem 2.63]{APD}). Let $x \in \s \cap \Om$ be such a point that (\ref{6.52}) and (\ref{6.53}) hold with $x$. According to (\ref{6.52}),
	\begin{equation}\label{6.54}
	\beta_{\s}(x, r) \underset{r \to 0+}{\to 0}. 
	\end{equation}
	We claim that $w^{\tau}_{\s}(x, r)\to 0$ as $r\to 0+$. Indeed, by (\ref{6.54}), for any $\varepsilon \in (0, \varepsilon_{0})$ there is $t_{\varepsilon} \in (0, \overline{r})$ such that 
	\begin{equation}\label{6.55}
	\beta_{\s}(x, r) \leq \varepsilon \,\ \text{for all} \,\ r \in (0, t_{\varepsilon}]. 
	\end{equation}
	We assume that $B_{t_{\varepsilon}}(x) \subset \Omega$, $t_{\varepsilon}<\diam(\Sigma)/2$ and $\varepsilon<\tau/2$. Recall that $\tau \in (0, \varepsilon_{0}/6]$. Then, by Proposition~\ref{prop 6.6}, for all  $r\in (0, t_{\varepsilon}/10]$,
	\begin{equation}\label{6.56}
	w^{\tau}_{\Sigma}(x, r) \leq C\Bigl(\frac{r}{t_{\varepsilon}}\Bigr)^{b}w^{\tau}_{\Sigma}(x,t_{
		\varepsilon}) + Cr^{b}. 
	\end{equation}
	On the other hand, by Remark~\ref{rem 6.4} and Proposition~\ref{prop 2.16}, $w^{\tau}_{\Sigma}(x,t_{\varepsilon})< +\infty$. Thus, letting $r$ tend to $0+$ in (\ref{6.56}), we get
	\begin{equation}\label{6.57}
	w^{\tau}_{\Sigma}(x,r) \underset{r \to 0+}{\to 0}. 
	\end{equation}
	By (\ref{6.54}) and (\ref{6.57}),
	\[
	\beta_{\s}(x,r) + w^{\tau}_{\s}(x,r) \underset{r \to 0+}{\to} 0.
	\]
	This, together with (\ref{6.53}), Corollary \ref{cor 6.13} and the fact that for each integer $N\geq2$, the unique positive solution $\overline{\mu}$ to the equation $\mu=5+\mu^{1-\frac{1}{N}}$ is strictly greater than 5, completes the proof of Theorem~\ref{thm 1.2}.
	
\end{proof}

\end{subsection}
\section{Acknowledgments}
I thank Antoine Lemenant for his valuable comments and suggestions which helped me to improve the paper. I wish to thank Guy David for several useful comments and suggestions on the regularity theory. I am also grateful to Guy Bouchitté and Antonin Chambolle for useful discussions on the optimal $p$-compliance problem. This work was partially supported by the project ANR-18-CE40-0013 SHAPO financed by the French Agence Nationale de la Recherche (ANR).
\appendix
\section{Auxiliary results}
Recall that we write points of $\mathbb{R}^{N}$ as $x=(x^{\prime}, x_{N})$ with $x^{\prime} \in \mathbb{R}^{N-1}$ and $x_{N}\in \mathbb{R}$.
\begin{lemma} \label{lem A.1} Let $N\geq 2$, $p \in (N-1,+\infty)$, $\beta=(p-N+1)/(p-1)$ and $\gamma \in (0, \beta)$. There exists $\delta \in (0,1)$, depending only on $N,p$ and $\gamma$, such that $\hat{u}(x)=|x^{\prime}|^{\gamma}+x_{N}^{2}$ is a supersolution to the $p$-Laplace equation in $\{0<|x^{\prime}|<\delta\}\cap \{|x_{N}|<1\}$.
\end{lemma}
\begin{proof} To simplify the notation, we denote $\{0<|x^{\prime}|<\delta\}\cap \{|x_{N}|<1\}$ by $C^{o}_{\delta,1}$. We need to prove that there exists $\delta=\delta(N,p,\gamma) \in (0,1)$ such that 
	\begin{equation} \label{A.1}
	\Delta_{p}\hat{u}=\Delta \hat{u}|\nabla \hat{u}|^{p-2}+(p-2)|\nabla \hat{u}|^{p-4}\Delta_{\infty} \hat{u} \leq 0 \,\ \text{in} \,\ C^{o}_{\delta,1},
	\end{equation}
where $\Delta \hat{u}=\Delta_{2} \hat{u}:=\sum^{N}_{i=1}\hat{u}_{x_{i},x_{i}}$ and $\Delta_{\infty} \hat{u}:=\sum_{i,j=1}^{N}\hat{u}_{x_{i}}\hat{u}_{x_{j}}\hat{u}_{x_{i},x_{j}}$. Since $|\nabla \hat{u}|\neq 0$ in $C^{o}_{\delta,1}$, (\ref{A.1}) is equivalent to the following
\begin{equation} 
\hat{\Delta}_{p} \hat{u}:=\Delta \hat{u} |\nabla \hat{u}|^{2}+(p-2)\Delta_{\infty} \hat{u}\leq 0 \,\ \text{in} \,\ C^{o}_{\delta,1}. \notag
\end{equation}
Calculating the partial derivatives of $\hat{u}$ in $C^{o}_{\delta,1}$, we have: $\hat{u}_{x_{i}}=\gamma x_{i}|x^{\prime}|^{\gamma-2}$, $i \in \{1,...,N-1\}$; $\hat{u}_{x_{N}}=2x_{N}$; $\hat{u}_{x_{i},x_{j}}=\gamma(\gamma-2)x_{i}x_{j}|x^{\prime}|^{\gamma-4}+\delta_{ij}\gamma|x^{\prime}|^{\gamma-2}$, where $i,j \in \{1,...,N-1\}$ and $\delta_{i,j}$ is the Kronecker delta; $\hat{u}_{x_{N},x_{N}}=2$. Next, we deduce that 
\[
\Delta \hat{u}=\gamma(\gamma+N-3)|x^{\prime}|^{\gamma-2}+2,\,\ |\nabla \hat{u}|^{2}=\gamma^{2}|x^{\prime}|^{2\gamma-2}+4x_{N}^{2}\,\ \text{and}\,\ \Delta_{\infty} \hat{u}=\gamma^{3}(\gamma-1)|x^{\prime}|^{3\gamma-4}+8x_{N}^{2}
\]
in $C^{o}_{\delta,1}$. This yields the following
\begin{equation} \label{A.2}
\hat{\Delta}_{p}\hat{u}=\gamma^{3}(\gamma p -p-\gamma+N-1)|x^{\prime}|^{3\gamma-4}+4\gamma(\gamma+N-3)|x^{\prime}|^{\gamma-2}x^{2}_{N}+2\gamma^{2}|x^{\prime}|^{2\gamma-2}+(p-1)8x^{2}_{N}
\end{equation}
in $C^{o}_{\delta,1}$. Since $0<\gamma<\beta$, $\gamma^{3}(\gamma p-p-\gamma+N-1)<0$ and $3\gamma -4<\gamma -2<0$. Thus, analyzing (\ref{A.2}), we deduce that there exists $\delta=\delta(N,p,\gamma)\in (0,1)$ such that $\hat{\Delta}_{p} \hat{u} \leq 0$ in $C^{o}_{\delta,1}$. This completes the proof.
\end{proof}

The next lemma is classical, however, we could not find a precise reference in the exact following form, thus we provide a proof for the reader’s convenience.

\begin{lemma} \label{lem A.2} Let $N\geq 2$, $p\in (1,+\infty)$, $\Sigma \subset \mathbb{R}^{N}$ be a closed set and $u \in W^{1,p}(B_{1})$ be a $p$-harmonic function in $B_{1}\backslash \Sigma$, continuous in $B_{1}$ with $u=0$ on $\Sigma \cap B_{1}$. Then $u^{+}=\max\{u,0\}$ and $u^{-}=-\min\{u,0\}$ are $p$-subharmonic in $B_{1}$. 
\end{lemma}
\begin{proof} Since $u^{-}=(-u)^{+}$ and $(-u)$ is $p$-harmonic in $B_{1}\backslash \Sigma$, it is enough to prove that the function $u^{+}$ is $p$-subharmonic in $B_{1}$. Let us fix an arbitrary nonnegative function $\varphi \in C^{\infty}_{0}(B_{1})$ and for all $\varepsilon, \eta \in (0,1)$ define $\varphi_{\eta,\varepsilon}=((\eta+(u-\varepsilon)^{+})^{\varepsilon}-\eta^{\varepsilon})\varphi$. Since $u \in W^{1,p}(B_{1})$ is $p$-harmonic in $B_{1}\backslash \Sigma$ and $\varphi_{\eta,\varepsilon}\in W^{1,p}_{0}(B_{1}\backslash \Sigma)$,
	\[
	\int_{B_{1}}|\nabla u|^{p-2}\nabla u \nabla \varphi_{\eta,\varepsilon} \diff x=0.
	\]
This implies that 
\[
\int_{B_{1}}((\eta+(u-\varepsilon)^{+})^{\varepsilon}-\eta^{\varepsilon})|\nabla u^{+}|^{p-2}\nabla u^{+}\nabla \varphi\diff x + \varepsilon \int_{B_{1}\cap \{u>\varepsilon\}}|\nabla u^{+}|^{p}(\eta+(u-\varepsilon)^{+})^{\varepsilon-1}\varphi\diff x=0
\]
and hence
\begin{equation} \label{A.3}
\int_{B_{1}}((\eta+(u-\varepsilon)^{+})^{\varepsilon}-\eta^{\varepsilon})|\nabla u^{+}|^{p-2}\nabla u^{+}\nabla \varphi \diff x \leq 0. 
\end{equation}
Letting $\eta$ and then $\varepsilon$ tend to $0+$ in (\ref{A.3}), by Lebesgue's dominated convergence theorem, we get
\[
\int_{B_{1}}|\nabla u^{+}|^{p-2}\nabla u^{+}\nabla \varphi \diff x \leq 0,
\]
which concludes the proof.
\end{proof}

The next lemma is a refined version of \cite[Lemma 5.14]{Opt}.
\begin{lemma}\label{lem A.3} Let  $N\geq 2$ and let $\gamma: [0,1] \to \mathbb{R}^{N}$ be a curve such that $\Gamma:=\gamma([0,1])\subset \overline{B}_{r}(x_{0})$. Assume that $\xi_{1}=\gamma(0)\in \partial B_{r}(x_{0})$ and $\xi_{2}=\gamma(1) \in \partial B_{r}(x_{0})$. Then
		\begin{equation*}
		\max_{y \in \Gamma}\dist(y, [\xi_{1}, \xi_{2}]) \leq (2r (\mathcal{H}^{1}(\Gamma) - |\xi_{2}-\xi_{1}|))^{\frac{1}{2}}.
		\end{equation*}
	\end{lemma}
\begin{proof} Let $z\in \argmax_{y\in \Gamma} \dist(y,[\xi_{1},\xi_{2}])$. Assume that $\dist(z,[\xi_{1},\xi_{2}]):=h>0$ and $|\xi_{1}-\xi_{2}|>0$, otherwise the proof follows. Let $z^{\prime}\in \mathbb{R}^{N}$ be a point making $(\xi_{1},z^{\prime},\xi_{2})$ an isosceles triangle such that $\dist(z^{\prime}, [\xi_{1},\xi_{2}])=h$. Notice that $h\leq 2r$, $|\xi_{1}-\xi_{2}|/2\leq r$ and hence 
	\[
     |z^{\prime}-\xi_{2}|\leq h + \frac{|\xi_{1}-\xi_{2}|}{2}\leq 3r.
	\]
	On the other hand, $\mathcal{H}^{1}(\Gamma)\geq 2|z^{\prime}-\xi_{2}|$. Then, using the Pythagorean theorem, we get
	\begin{align*}
	h^{2}=|z^{\prime}-\xi_{2}|^{2}-\frac{|\xi_{1}-\xi_{2}|^{2}}{4}&=\biggl(|z^{\prime}-\xi_{2}|-\frac{|\xi_{1}-\xi_{2}|}{2}\biggr)\biggl(|z^{\prime}-\xi_{2}|+\frac{|\xi_{1}-\xi_{2}|}{2}\biggr)\\
	&\leq \biggl(\frac{\mathcal{H}^{1}(\Gamma)}{2}-\frac{|\xi_{1}-\xi_{2}|}{2}\biggr)(3r+r)\\
	&=2r(\mathcal{H}^{1}(\Gamma)-|\xi_{1}-\xi_{2}|).
	\end{align*}
	This completes the proof of Lemma \ref{lem A.3}.
\end{proof}

\bibliography{bib1}

\begin{thebibliography}{10}

\bibitem{Potential}
D.~R. Adams and L.~I. Hedberg.
\newblock {\em Function spaces and potential theory}, volume 314 of {\em
  Grundlehren der Mathematischen Wissenschaften [Fundamental Principles of
  Mathematical Sciences]}.
\newblock Springer-Verlag, Berlin, 1996.

\bibitem{APD}
L.~Ambrosio, N.~Fusco, and D.~Pallara.
\newblock {\em Functions of bounded variation and free discontinuity problems}.
\newblock Oxford Mathematical Monographs. The Clarendon Press, Oxford
  University Press, New York, 2000.

\bibitem{BAGBY}
T.~Bagby.
\newblock Quasi topologies and rational approximation.
\newblock {\em Journal of Functional Analysis}, 10(3):259 -- 268, 1972.

\bibitem{BOCCARDO}
L.~Boccardo and F.~Murat.
\newblock Almost everywhere convergence of the gradients of solutions to
  elliptic and parabolic equations.
\newblock {\em Nonlinear Anal.}, 19:581 -- 597, 1992.

\bibitem{Bonnoet}
A.~Bonnet.
\newblock On the regularity of edges in image segmentation.
\newblock {\em Ann. Inst. H. Poincar\'{e} Anal. Non Lin\'{e}aire},
  13(4):485--528, 1996.

\bibitem{Approx}
M.~Bonnivard, A.~Lemenant, and F.~Santambrogio.
\newblock Approximation of length minimization problems among compact connected
  sets.
\newblock {\em SIAM J. Math. Anal.}, 47(2):1489--1529, 2015.

\bibitem{Bucur}
D.~Bucur and P.~Trebeschi.
\newblock Shape optimisation problems governed by nonlinear state equations.
\newblock {\em Proc. Roy. Soc. Edinburgh Sect. A}, 128(5):945--963, 1998.

\bibitem{importance}
B.~Bulanyi.
\newblock On the importance of the connectedness assumption in the statement of
  the optimal $p$-compliance problem, to appear in \textit{J. Math. Anal. Appl.
  (JMAA),} 2021.

\bibitem{p-compl}
B.~Bulanyi and A.~Lemenant.
\newblock Regularity for the planar optimal $p$-compliance problem.
\newblock {\em Preprint arXiv:1911.09240}, 2020.

\bibitem{Butazzo-Santambrogio}
G.~Buttazzo and F.~Santambrogio.
\newblock Asymptotical compliance optimization for connected networks.
\newblock {\em Netw. Heterog. Media}, 2(4):761--777, 2007.

\bibitem{Opt}
A.~Chambolle, J.~Lamboley, A.~Lemenant, and E.~Stepanov.
\newblock Regularity for the optimal compliance problem with length
  penalization.
\newblock {\em SIAM J. Math. Anal.}, 49(2):1166--1224, 2017.

\bibitem{Dacorogna}
B.~Dacorogna, W.~Gangbo, and N.~Sub\'{\i}a.
\newblock Sur une g\'{e}n\'{e}ralisation de l'in\'{e}galit\'{e} de {W}irtinger.
\newblock {\em Ann. Inst. H. Poincar\'{e} Anal. Non Lin\'{e}aire}, 9(1):29--50,
  1992.

\bibitem{ABC}
G.~Dal~Maso and F.~Murat.
\newblock Asymptotic behaviour and correctors for {D}irichlet problems in
  perforated domains with homogeneous monotone operators.
\newblock {\em Ann. Scuola Norm. Sup. Pisa Cl. Sci. (4)}, 24, 1997.

\bibitem{David}
G.~David.
\newblock {\em Singular Sets of Minimizers for the Mumford-Shah Functional},
  volume 233 of {\em Progress in mathematics}.
\newblock Birkh\"{a}user-Verlag, Basel, 2005.

\bibitem{David-Kenig-Toro}
G.~David, C.~Kenig, and T.~Toro.
\newblock Asymptotically optimally doubling measures and reifenberg flat sets
  with vanishing constant.
\newblock {\em Communications on Pure and Applied Mathematics}, 54(4):385--449,
  2001.

\bibitem{PDE}
D.~Gilbarg and N.S. Trudinger.
\newblock {\em Elliptic Partial Differential Equations of Second Order}, volume
  224 of {\em Grundlehren der mathematischen Wissenschaften}.
\newblock Springer-Verlag, Berlin, second edition, 2001.

\bibitem{Hedberg}
L.I. Hedberg.
\newblock Non-linear potentials and approximation in the mean by analytic
  functions.
\newblock {\em Math. Z.}, 129:299--319, 1972.

\bibitem{NPT}
J.~Heinonen, T.~Kilpel{\"a}inen, and O.~Martio.
\newblock {\em Nonlinear potential theory of degenerate elliptic equations}.
\newblock Dover Publications, Inc. Mineola, New York, 2006.

\bibitem{Lindqvist}
P.~Lindqvist.
\newblock {\em Notes on the stationary {$p$}-{L}aplace equation}.
\newblock SpringerBriefs in Mathematics. Springer, Cham, 2019.

\bibitem{Lundstrom7}
N.~L.~P. Lundström.
\newblock Estimates for $p$-harmonic functions vanishing on a flat.
\newblock {\em Nonlinear Anal.}, 74(18):6852 -- 6860, 2011.

\bibitem{Lundstrom}
N.~L.~P. Lundström.
\newblock Phragmén-{L}indelöf theorems and $p$-harmonic measures for sets
  near low-dimensional hyperplanes.
\newblock {\em Potential Analysis}, 44:313--330, 2016.

\bibitem{Miranda-Paolini-Stepanov}
M.~Miranda, E.~Paolini, and E.~Stepanov.
\newblock On one-dimensional continua uniformly approximating planar sets.
\newblock {\em Calc. Var.}, 27:287--309, 2006.

\bibitem{MR3063566}
A.~Nayam.
\newblock Asymptotics of an optimal compliance-network problem.
\newblock {\em Netw. Heterog. Media}, 8(2):573--589, 2013.

\bibitem{MR3195349}
A.~Nayam.
\newblock Constant in two-dimensional {$p$}-compliance-network problem.
\newblock {\em Netw. Heterog. Media}, 9(1):161--168, 2014.

\bibitem{Stepanov-Paolini}
E.~Paolini and E.~Stepanov.
\newblock Qualitative properties of maximum distance minimizers and average
  distance minimizers in $\mathbb{R}^{N}$.
\newblock {\em Journal of Mathematical Sciences}, 122:3290–3309, 2004.

\bibitem{Paolini-Stepanov}
E.~Paolini and E.~Stepanov.
\newblock Existence and regularity results for the {S}teiner problem.
\newblock {\em Calc. Var. Partial Differential Equations}, 46(3-4):837--860,
  2013.

\bibitem{MR3165284}
D.~Slep\v{c}ev.
\newblock Counterexample to regularity in average-distance problem.
\newblock {\em Ann. Inst. H. Poincar\'{e} Anal. Non Lin\'{e}aire},
  31(1):169--184, 2014.

\bibitem{bestconstant}
G.~Talenti.
\newblock Best constant in {S}obolev inequality.
\newblock {\em Ann. Mat. Pura Appl.}, 110:353--372, 1976.

\bibitem{Ziemer}
W.~P. Ziemer.
\newblock {\em Weakly differentiable functions}, volume 120 of {\em Graduate
  Texts in Mathematics}.
\newblock Springer-Verlag, New York, 1989.
\newblock Sobolev spaces and functions of bounded variation.

\bibitem{sverak}
V.~Šverák.
\newblock On optimal shape design.
\newblock {\em J. Math. Pures Appl.}, 72:537--551, 1993.

\end{thebibliography}
\bibliographystyle{plain}

\end{document}